\pdfoutput=1

\documentclass[11pt,leqno]{amsart}

\usepackage{graphicx,verbatim,lineno,titletoc}
\usepackage{amssymb,mathrsfs}
\usepackage{amsmath}
\usepackage{adjustbox}
\usepackage{calc}
\usepackage{ytableau}
\usepackage{array}
\usepackage{tikz}
\usetikzlibrary{shapes.multipart,patterns,arrows}
\usepackage[font=footnotesize]{caption}
\usepackage[T1]{fontenc} 
\usepackage{fourier} 
\usepackage[english]{babel} 
\usepackage{appendix}
\usepackage[all]{xy}
\usepackage{enumerate}

\usepackage{multirow}
\usepackage{float}
\usepackage{enumitem}
\usepackage{accents}
\usepackage[numbers]{natbib}
\usepackage{algorithm}
\usepackage[noend]{algpseudocode}
\usepackage{mathtools}
\mathtoolsset{showonlyrefs}

\usepackage{wrapfig}
\usepackage{subcaption}

\usepackage{hyperref}
\usepackage[top=1.0in, left=1.1in, right=1.1in, bottom=1.0in]{geometry}
\hypersetup{colorlinks,linkcolor={red},citecolor={olive},urlcolor={red}}

\newcolumntype{M}[1]{>{\centering\arraybackslash}m{#1}}
\newcolumntype{N}{@{}m{0pt}@{}}

\newtheorem{theorem}{Theorem}[section]
\newtheorem{prop}[theorem]{Proposition}

\newtheorem{lemma}[theorem]{Lemma}
\newtheorem{corollary}[theorem]{Corollary}

\newtheorem{remark}[theorem]{Remark}

\def\liminf{\mathop{\rm lim\,inf}\limits}

\def\R{\mathbb{R}}

\def\E{\mathbb{E}}

\def\h{\mathbf{h}}

\def\eps{\varepsilon}

\def\X{\mathbf{X}}
\def\x{\mathbf{x}}

\def\param{\boldsymbol{\theta}}
\def\Param{\boldsymbol{\Theta}}

\definecolor{hancolor}{rgb}{0.1, 0.0, 0.9}
\newcommand{\commHL}[1]{{\textcolor{black}{#1}}}

\DeclareMathOperator{\Out}{\texttt{Out}}

\DeclareMathOperator*{\argmin}{arg\,min}

\usepackage{pifont}

\newcommand{\xmark}{\text{\ding{55}}}

\newlength\myindent
\newtheorem{innercustomassumption}{}
\newenvironment{customassumption}[1]
{\renewcommand\theinnercustomassumption{#1}\innercustomassumption}
{\endinnercustomassumption}

\theoremstyle{definition}

\allowdisplaybreaks

\definecolor{mycolor}{rgb}{0, 0, 0}
\newcommand{\revision}[1]{{\textcolor{black}{#1}}}

\makeatletter
\newcommand{\addresseshere}{%
	\enddoc@text\let\enddoc@text\relax
}
\makeatother

\begin{document}
	
	\title[Convergence and Complexity of BMM with diminishing radius]{Block majorization-minimization with diminishing radius \\ for constrained  nonsmooth nonconvex optimization}

	\author{Hanbaek Lyu}
	\address{Hanbaek Lyu, Department of Mathematics, University of Wisconsin - Madison, WI 53706, USA}
	\email{\texttt{hlyu@math.wisc.edu}}

	\author{Yuchen Li}
	\address{Yuchen Li, Department of Mathematics, University of Wisconsin - Madison, WI 53706, USA}
	\email{\texttt{yli966@wisc.edu}}

	\thanks{The codes for the main algorithm and simulations are provided in \url{https://github.com/HanbaekLyu/BMM-DR}}

% REQUIRED
\begin{abstract}
Block majorization-minimization (BMM) is a simple iterative algorithm for constrained nonconvex optimization that sequentially minimizes majorizing surrogates of the objective function in each block while the others are held fixed. BMM entails a large class of optimization algorithms such as block coordinate descent and its proximal-point variant, expectation-minimization, and block projected gradient descent. We first establish that for general constrained nonsmooth nonconvex optimization, BMM with $\rho$-strongly convex \revision{and $L_g$-smooth} surrogates can produce an $\eps$-approximate first-order optimal point within  \revision{$\widetilde{O}((1+L_g+\rho^{-1})\eps^{-2})$} iterations and asymptotically converges to the set of first-order optimal points. Next, we show that BMM combined with trust-region methods with diminishing radius has an improved complexity of \revision{$\widetilde{O}((1+L_g) \eps^{-2})$}, independent of the inverse strong convexity parameter $\rho^{-1}$, allowing improved theoretical and practical performance with `flat' surrogates. Our results hold robustly even when the convex sub-problems are solved as long as the optimality gaps are summable. \revision{Central to our analysis is a novel continuous first-order optimality measure, by which we bound the worst-case sub-optimality in each iteration by the first-order improvement the algorithm makes.} We apply our general framework to obtain new results on various algorithms such as the celebrated multiplicative update algorithm for nonnegative matrix factorization by Lee and Seung, regularized nonnegative tensor decomposition, and the classical block projected gradient descent algorithm. Lastly, we numerically demonstrate that the additional use of diminishing radius can improve the convergence rate of BMM in many instances. 
\end{abstract}
	\maketitle
	\vspace{-0.8cm}
\tableofcontents

	\section{Introduction}
	\label{Introduction}
	
	Throughout this paper, we are interested in the minimization of a continuous function $F :\mathcal{E}:=\R^{I_{1}}\times \dots \times \R^{I_{m}}\rightarrow [0,\infty)$ on a cartesian product of closed convex sets $\Param=\Theta^{(1)}\times \dots \times \Theta^{(m)}$: 
	\begin{align}\label{eq:block_minimization}
		\param^{*} \in \argmin_{\param=[\theta^{(1)},\dots,\theta^{(m)}]\in \Param} \left(F(\param):=f(\param) +p(\param)\right).
	\end{align} 
    The objective function $F$ is the sum of a smooth (possibly nonconvex) part $f$ and a nonsmooth (continuous and convex) part $p$. 
    %\commHL{We are interested in obtaining convergence and complexity results of a class of practical optimization algorithms under a minimal set of assumptions, and investigate how the results depend on optimization hyperparameters.}
    \commHL{Under a minimal set of assumptions, we investigate how to obtain first-order optimal points of \eqref{eq:block_minimization} from an arbitrary initialization.} 
    %Since in general $F$ is nonconvex, the convergence of any algorithm for solving \eqref{eq:block_minimization} to a global minimum can hardly be expected. Instead, global convergence to first-order optimal points of the objective function is desired.

	To obtain first-order optimal solutions to \eqref{eq:block_minimization}, we consider algorithms based on  \textit{Block Majorization-Minimization} (BMM) \cite{hong2015unified}. The high-level idea of BMM is that, in order to minimize a multi-block objective, one can minimize a majorizing surrogate of the objective in each block in a cyclic order\footnote{The order of block updates need not be cyclic, see \cite{hong2015unified} for other update rules.}: For $n\ge 1$ and $i=1,\dots,m$, 
	\begin{align}\label{eq:EBMM}
		\textbf{\textup{BMM}}\quad  
		\begin{cases}
			g_{n}^{(i)} \leftarrow \left[ \begin{matrix} \textup{Majorizing surrogate of} \\ \textup{$\theta\mapsto f_{n}^{(i)}(\theta):=f\left(\theta_{n}^{(1)},\cdots,\theta_{n}^{(i-1)},\theta,\theta_{n-1}^{(i+1)},\cdots, \theta_{n-1}^{(m)}\right)$}
            %\\ 
			%\textup{$g_{n}^{(i)}\ge f_{n}^{(i)}$,\, $g_{n}^{(i)}=f_{n}^{(i)}$ and \revision{$\nabla g_{n}^{(i)}=\nabla f_{n}^{(i)}$ at $\theta_{n-1}^{(i)}$}}
			\end{matrix} \right] \\
   p_{n}^{(i)}(\theta):=p\left(\theta_{n}^{(1)},\cdots,\theta_{n}^{(i-1)},\theta,\theta_{n-1}^{(i+1)},\cdots, \theta_{n-1}^{(m)}\right)\\
			\theta_{n}^{(i)}\in \argmin_{\theta\in \Theta^{(i)}}  \left(G_{n}^{(i)}(\theta):=g_{n}^{(i)}(\theta)+p_{n}^{(i)}(\theta)\right).
		\end{cases}
	\end{align}

	BMM entails numerous well-known algorithms for constrained nonconvex minimization. First, when the smooth part $f$ of the objective function is convex in each block (i.e., block multi-convex) and the surrogate $g_{n}^{(i)}$  in \eqref{eq:EBMM} is identical to $f_{n}^{(i)}$, then BMM reduces to block coordinate descent (BCD), also known as nonlinear Gauss-Seidel \cite{ wright2015coordinate}, where one sequentially minimizes the objective in each block coordinate while the others  are held fixed:
	\begin{align}\label{eq:BCD_factor_update}
		\hspace{-1.3cm} \textbf{\textup{BCD}} \hspace{2cm} 	\theta_{n}^{(i)}\in \argmin_{\theta\in \Theta^{(i)}}   \left(F_{n}^{(i)}(\theta)=f_{n}^{(i)}(\theta) + p_{n}^{(i)}(\theta)\right),
	\end{align}
 where $f_{n}^{(i)}$ and $p_{n}^{(i)}$ are defined in \eqref{eq:EBMM}. 
	Due to its simplicity, BCD has been widely used in various optimization problems such as nonnegative matrix or tensor factorization \cite{lee2001algorithms, kolda2009tensor}. %lee1999learning, 
  Using proximal surrogates in \eqref{eq:EBMM}, BMM becomes BCD with proximal regularization (BCD-PR): 
	\begin{align}\label{eq:BCD_factor_update_proximal}
		\hspace{-0.1cm} 
		\textbf{\textup{BCD-PR}} \hspace{0.8cm}	\theta_{n}^{(i)}\in \argmin_{\theta\in \Theta^{(i)}}  \left( G_{n}^{(i)}(\theta):= f_{n}^{(i)}(\theta) + \frac{\lambda}{2}\lVert \theta - \theta_{n-1}^{(i)} \rVert^{2} + p_{n}^{(i)}(\theta) \right), 
	\end{align}
	where $\lambda\ge 0$ is a fixed constant. If we use prox-linear surrogates in \eqref{eq:EBMM}, then BMM becomes the block prox-linear minimization \cite{xu2013block}, which is equivalent to the block projected gradient descent (BPGD) \cite{tseng2009coordinate} when the nonsmooth part $p$ is nonexistent:
 \iffalse
 \footnote{
 \vspace{-0.2cm}
 \begin{align}
\eqref{eq:BProxLinear} &=  \argmin_{\theta\in \Theta^{(i)}} \left\lVert   \theta - \left( \theta^{(i)}_{n-1} - \frac{1}{\lambda} \nabla f_{n}^{(i)}(\theta_{n-1}^{(i)}) \right)  \right\rVert^{2} \nonumber \\
&=  \argmin_{\theta\in \Theta^{(i)}} \left\lVert   \theta - \left( \lambda \theta^{(i)}_{n-1} -  \nabla f_{n}^{(i)}(\theta_{n-1}^{(i)}) \right)  \right\rVert^{2}=\eqref{eq:BPGD} \nonumber
 \end{align}
 }
 \fi
	%For minimizing a differentiable function $f$ defined on the Euclidean space, 
	\begin{align}\label{eq:BProxLinear}\hspace{-0.6em}\textbf{BPGD}  \quad 
    \begin{cases}
		g_{n}^{(i)}(\theta):= f_{n}^{(i)}(\theta_{n-1}^{(i)} ) + \langle \nabla f_{n}^{(i)}(\theta_{n-1}^{(i)}) ,\, \theta- \theta_{n-1}^{(i)}  \rangle + \frac{\rho}{2} \lVert \theta - \theta_{n-1}^{(i)}\rVert^{2} \\ 
        \theta_{n}^{(i)} \leftarrow \argmin_{\theta\in \Theta^{(i)}} \left( \commHL{G_{n}^{(i)}(\theta) = g_{n}^{(i)}(\theta) + p_{n}^{(i)}(\theta)} \right)  \\
		\hspace{1.1cm}=  \textup{Proj}_{\Theta^{(i)}} \left( \theta^{(i)}_{n-1} - \frac{1}{\rho} \nabla f_{n}^{(i)}(\theta_{n-1}^{(i)}) \right)\; \; \textup{if $p=0$.}%\label{eq:BPGD}
  \end{cases}
	\end{align}
	The function $g_{n}^{(i)}$ in \eqref{eq:BProxLinear} is indeed a majorizing surrogate of $f_{n}^{(i)}$ when the smooth part $f$ of the objective has $L$-Lipschitz gradient and $\rho \ge L$. 
	BPGD has applications in nonnegative matrix factorization \cite{lin2007projected}, nonnegative tensor completion \cite{liu2012tensor}, %blind source separation \cite{jutten1991blind} , hyperspectral data analysis \cite{pauca2006nonnegative}, sparse dictionary learning \cite{mairal2009online,aharon2006k} 
    and many other problems where the objective function is generally nonconvex and the constraint set is convex in each block.

	A key advantage of BMM over BCD is that one can work with user-constructed majorizing surrogates $g_{n}^{(i)}$ that are strongly convex, while the smooth part $f_{n}^{(i)}$ of the marginal block objective may not even be convex. This advantage is implicit in BPGD but becomes apparent if we view it as a BMM with prox-linear surrogates in \eqref{eq:BProxLinear}, which is $\lambda$-strongly convex.  For instance, it ensures the uniqueness of their minimizer, which is a key property that warrants asymptotic convergence to stationary points. In addition, strong convexity also plays a key role in iteration complexity analysis \cite{xu2013block, hong2017iteration,kwon2023complexity} since it often implies square-summability of one-step parameter changes. 
 
    \revision{In this work, we establish that  the iteration complexity of BMM with $\rho$-strongly convex and $L_g$-smooth surrogates for general nonsmooth nonconvex constrained optimization problem 
    is $\widetilde{O}((1+L_g +\rho^{-1})\eps^{-2})$, where $\widetilde{O}(\cdot)$ is the usual $O(\cdot)$ hiding a poly-logarithmic factor in $\eps$. 
    Since $\rho\le L_{g}$, our result indicates that choosing surrogates that are not too steep ($L_{g}$ large) or too flat ($\rho^{-1}$ large) will be beneficial.    } 
    %While it is known that similar iteration complexity result for BMM with randomized block updates \textit{hold expectation} \cite{razaviyayn2014parallel}, there was no prior deterministic iteration complexity for BMM. 
    %
    ~
    \hspace{-0.5cm}\revision{\indent We move one step further beyond what the classical BMM can offer. Can we modify BMM in such a way that the complexity bound does not involve the undesirable factor of $\rho^{-1}$? The reason why this factor appears in the bound is as follows. When the surrogates are nearly flat ($\rho>0$ but small), one seeks to make aggressive parameter changes whenever possible, but it becomes ineffective when the objective cannot be greatly improved. }
    Thus, one may try to improve BMM by keeping the nearly flat surrogates if possible but gradually limiting the range of parameter changes as the algorithm proceeds. From this motivation, we propose BMM in conjunction with trust-region techniques.
    Namely, fix a sequence $(r_{n})_{n\ge 1}$ of numbers in $(0,\infty]$ (including $\infty$) that acts as the radii of the trust-region. We then generalize  \eqref{eq:EBMM} as 
	\begin{align}\label{eq:BMM_DR_highlevel}
	\hspace{0cm}	\textbf{BMM-DR}\hspace{0.18cm}	
	\begin{cases}
		g_{n}^{(i)}\leftarrow \textup{Majorizing surrogate of $f_{n}^{(i)}$ at $\theta_{n-1}^{(i)}$ as in \eqref{eq:EBMM}}  \\
		\theta_{n}^{(i)}\in \argmin_{\theta\in \Theta^{(i)},\, \lVert \theta-\theta_{n}^{(i)} \rVert \le r_{n}}  \left(G_{n}^{(i)}(\theta) = g_{n}^{(i)}(\theta)+ p_{n}^{(i)}(\theta)\right).
	\end{cases}
	\end{align}
 	The majorizing surrogate $g_{n}^{(i)}$ in \eqref{eq:BMM_DR_highlevel} is assumed to satisfy the following properties:
	\begin{description}[itemsep=0.1cm]
		\item{(1)} (Majorization) $g_{n}^{(i)}(\theta)- f_{n}^{(i)}(\theta) \ge 0$ for all $\theta \in \Theta^{(i)}$;
		%\item{(3)} (Smoothness) $g_{n}^{(i)}$ is $L_{g}$-geodesically-smooth for some constant $L_{g}\ge 0$ for all $n\ge 1$ and $i=1,\dots,m$. 
		\item{(2)} (Sharpness) $g_{n}^{(i)}(\theta_{n-1}^{(i)})=f_{n}^{(i)}(\theta_{n-1}^{(i)})$ and \revision{$\nabla g_{n}^{(i)}(\theta_{n-1}^{(i)})=\nabla f_{n}^{(i)}(\theta_{n-1}^{(i)})$};%\commHL{"Equality, sharp"}
		\item{(3)} (Strong convexity) $g_{n}^{(i)}$ is $\rho$-strongly convex on $\Theta^{(i)}$ for some $\rho\ge 0$. 
	\end{description}
    \vspace{0.1cm}
	%Note that if $g_{n}^{(i)}$ is a majorizing surrogate of $f_{n}^{(i)}$, then $G_{n}^{(i)}$ is a majorizing surrogate of $F_{n}^{(i)}$. Also 
    Note that \eqref{eq:BMM_DR_highlevel} is identical to BMM \eqref{eq:EBMM} except that we restrict the range of parameter search within a radius $r_{n}$ from the previous estimation. When $r_{n}\equiv \infty$, then this additional radius constraint becomes vacuous and we recover the standard BMM \eqref{eq:EBMM}. The resulting algorithm, which we call BMM with diminishing radius (BMM-DR) is stated in Algorithm \eqref{eq:BMM_DR_highlevel}.  %See Figure \ref{fig:DR_PR_comparison} for an illustration of the effect of proximal regularization and trust region. 
 
 \commHL{Our key finding is that the additional trust-region constraint in BMM-DR improves the complexity bound to \revision{$\widetilde{O}((1+L_g)\eps^{-2})$}, removing the dependence on the inverse strong convexity parameter $\rho^{-1}$, even allowing convex surrogates $(\rho=0)$ (see Thm. \ref{thm:complexity}). The improvement is significant for BMM with `nearly flat surrogates' (e.g., linear surrogates for concave objectives \cite{breloy2021majorization}). We also note that unlike the classical use of trust-region where the radii are computed adaptively depending on the algorithm's progress \cite{conn2000trust}, we use a simple non-adaptive sequence of radii (e.g., $r_{n}=O(1/\sqrt{n})$). %which suffices for our purpose.
 }
 %Furthermore, using trust regions allows a larger class of surrogates by weakening the strongly convex assumption on surrogates to convex, which is applicable to the problems when one cannot use strongly convex surrogates (e.g., linear surrogates for concave objectives \cite{breloy2021majorization}, and low-rank matrix factorization). 
 We also establish asymptotic convergence to first-order optimal points of BMM-DR. Our theoretical establishments are verified by numerical experiments in Section \ref{sec:exp}.

	\subsection*{Related work}
	There are several stylized examples of BMM in a wide range of problems. For the single block case ($m=1$), BMM reduces to the well-known majorization-minimization algorithm \cite{lange2000optimization}, which entails the EM algorithm for maximum likelihood estimation, 
    %\cite{neal1998view, cappe2009line}
    forward-backward splitting \cite{combettes2011proximal}, and iterative reweighted least squares \cite{daubechies2010iteratively}. 
    With multiple blocks  ($m\ge 2$): Multiplicative update for nonnegative matrix factorization by Lee and Seung \cite{lee2001algorithms}, %block EM algorithm \cite{lee2018block}, 
    the convex-concave procedure for the difference of convex programs \cite{yuille2003concave}, alternating least squares for nonnegative CP decomposition \cite{carroll1970analysis}, and the classical proximal point algorithm \cite[Sec. 3.4.3]{bertsekas2015parallel}.

	Asymptotic convergence to stationary points of BCD for nonconvex objectives has been extensively studied in the literature.
 %hildreth1957quadratic, tseng1991decomposition, sargent1973convergence,
 \cite{luo1992convergence}. It is well-known that BCD does not always converge to the stationary points of the non-convex objective function that is convex in each block  \cite{powell1973search}, but such global convergence is guaranteed under additional assumptions such as two-block ($m=2$), strict quasiconvexity for $m-2$ blocks \cite{ grippo2000convergence}, 
 %grippof1999globally,
 or uniqueness of minimizer per block \cite[Sec. 2.7]{bertsekas1997nonlinear}. Due to the additional proximal regularization, BCD-PR is guaranteed to converge to stationary points as long as the proximal surrogates (see \eqref{eq:BCD_factor_update_proximal}) are strongly convex \cite{grippo2000convergence, xu2013block,kwon2023complexity}. 
 %, attouch2010proximal
 In \cite{razaviyayn2013unified}, BMM \eqref{eq:EBMM} for minimizing smooth objectives is known to converge to the set of stationary points when the surrogates $g_{n}^{(i)}$ have unique minimizer over the constraint sets $\Theta^{(i)}$. 
 For nonsmooth nonconvex constrained optimization, Xu and Yin \cite{xu2013block} showed that BCD, BCD-PR, and BPGD  converge asymptotically to the set of Nash equilibria (a weaker notion than stationary points) when the nonsmooth part of the objective (possibly in conjunction with the indicator function of convex constraint set) is not necessarily continuous but block-separable. However, asymptotic convergence to stationary points in the nonconvex nonsmooth constrained setting is still unknown.

For minimizing convex objectives, BMM reduces the gap between the current objective value and the global minimum at rate $O(1/n)$ in $n$ iterations \cite{hong2017iteration}, assuming strong convexity of the surrogates. A series of works including \cite{cai2023cyclic, beck2013convergence} 
%,saha2013nonasymptotic,li2017faster,sun2015improved,nesterov2012efficiency
 proved the complexity of BMM and its variants for convex objectives under different settings. A summary of some techniques used in the proofs can also be found in \cite{beck2017first}. %These results, however, require certain types of convexity assumptions. 

{\color{black}

 Compared to the convex minimization case, the iteration complexity for BMM \eqref{eq:EBMM} for the constrained nonconvex nonsmooth setting is more limited. Here by `iteration complexity', we mean the worst-case number of iterations until an $\eps$-approximate first-order optimal point is obtained, using a suitable measure of sub-optimality. Xu and Yin \cite{xu2013block} obtain the local rate of convergence of BCD, BCD-PR, and BPGD under the additional assumption that the objective function satisfies the Kurdyka-\L{}ojasiewicz inequality. 
Recently, Lyu and Kwon showed that BCD-PR has iteration complexity of \revision{$\widetilde{O}((1+L_g+\rho^{-1})\eps^{-2})$} \cite{kwon2023complexity} both for the constrained and unconstrained settings, but their result do not cover nonsmooth objectives and general surrogates. A Riemannian counterpart of BPGD for compact manifolds was recently shown to have iteration complexity of $\widetilde{O}(\eps^{-2})$ \cite{peng2023block}. However, this result does not hold for the Euclidean setting with or without constraints, as the underlying manifold should be compact without boundary. Razaviyayn et al. \cite{razaviyayn2014parallel} shows that BMM with randomized coordinate update has complexity of 
\revision{$O_{\E}((1+\frac{L_{g}^{2}}{\rho}+\rho^{-1})\eps^{2})$}, where the subscript $\E$ means that the iteration complexity holds after taking the expectation over the randomness of the block selection during the algorithm. However, for the cyclic update rule, the (almost sure) complexity of BMM is not yet known. 

We are the first to propose BMM-DR so there are no prior results on its iteration complexity or asymptotic convergence properties. 
 
 \begin{table}[H]
 
		\centering
		\begin{tabular}{ccccccccc}
			\multicolumn{1}{c}{Methods} & Objective  & $\begin{matrix}\textup{Block}\\ \textup{update} \end{matrix}$ & Complexity & $\begin{matrix}\textup{Asymp.}\\ \textup{conv.} \end{matrix}$  &  $\begin{matrix} \text{Inexact} \\ \text{computation} \end{matrix}$ \\ \hline
			\multicolumn{1}{c}{BPGD \cite{xu2013block}}  &C \& NS & cyclic  &  \textup{Depends on KL-ineq.} &\checkmark & \xmark\\
			\multicolumn{1}{c}{BPGD \cite{beck2013convergence}}&C \& S & cyclic &$\widetilde{O}(\eps^{-1})$ & \xmark & \xmark\\
   BCD-PR \cite{kwon2023complexity} &NC \& S &cyclic &\revision{$\widetilde{O}((1+L_g+\rho^{-1})\eps^{-2})$} &\checkmark &\checkmark \\ 
             BMM \cite{razaviyayn2014parallel} &NC \& NS& random & \revision{$O_{\E}((1+\frac{L_{g}^{2}}{\rho}+\rho^{-1})\eps^{-2})$} & $\checkmark$  & $\xmark$ \\ 
              BMM  \cite{razaviyayn2014parallel} &NC \& NS & cyclic & $\xmark$ & $\checkmark$ & $\xmark$ \\ 
             %BPGD \cite{tseng2009coordinate} & NC \& NS & cyclic & $\xmark$ & $\checkmark$ & $\checkmark$ & $\xmark$ \\ 
            \hline           
			 \textbf{BMM (Ours)}   & NC \& NS & cyclic & \revision{$\widetilde{O}((1+L_g+\rho^{-1})\eps^{-2})$} & $\checkmark$  & $\checkmark$ \\
    \textbf{BMM-DR (Ours)}   & NC \& NS & cyclic & \revision{$\widetilde{O}((1+ L_g)\eps^{-2})$} & $\checkmark$ & $\checkmark$ \\
   \hline 
		\end{tabular}
        
		\caption{ Survey known results on the iteration complexity of BMM for multi-block minimization problems with $\rho$-strongly convex and $L_{g}$-smooth surrogates. $C=$ Convex, $NC=$ Nonconvex, $S=$ Smooth, and $NS=$ Nonsmooth. %The complexity for BMM-DR does not depend on $\rho$. 
		}
		\label{table:summary} 
	\end{table}
 
	\subsection*{Contribution}

	In this work, we analyze BMM with the optional trust-region  \eqref{eq:BMM_DR_highlevel}. Our main results are summarized below:
	\begin{description}
        \item[(1)] \commHL{We obtain a worst-case (anytime) bound  of \revision{$\widetilde{O}((1+L_g+\rho^{-1})\eps^{-2})$} on the number of iterations to achieve $\eps$-approximate `stationary-Nash' points (see Def. \ref{eq:stationary3}) using $\rho$-strongly convex and \revision{$L_g$-smooth} surrogates. (\revision{If $p$ is block-separable, stationary-Nash points are simply stationary points.}) 
        }

		\item[(2)] \commHL{Using an optional trust region with diminishing radius with the same surrogates in (1), we  obtain an improved iteration complexity \revision{$\widetilde{O}((1+L_g)\eps^{-2})$} that is independent of the strong convexity parameter $\rho$ of the surrogates; %Furthermore, we numerically validate that 
        }

        \item[(3)] Global asymptotic convergence of BMM(-DR) to \revision{stationary points from arbitrary initialization assuming block-separability of the nonconvex part $p$}; 
		%\item[(2)] An upper bound on the rate of convergence to the set  of stationary points of order $n^{-1/2}\log n$; 

		\item[(4)] All the aforementioned results hold under inexact execution of the algorithm. 
        %\item[(5)] Allowing non-strongly-convex surrogates by utilizing trust-regions with diminishing radii.
	\end{description}
 
	To the best of our knowledge, we believe that our work provides the first result on the global rate of convergence and iteration complexity of BMM(-DR) for minimizing nonsmooth nonconvex objectives under convex constraints. Especially, our rate of convergence does not depend on $\rho^{-1}$ if trust-region with diminishing radius is used. %We do not impose any additional assumptions such as the Kurdyka-$\L$ojasiewicz property. 
    For gradient descent methods with an unconstrained nonconvex objective, it is known that such a rate of convergence cannot be faster than $\widetilde{O}(\eps^{-2})$ \cite{cartis2010complexity}, so our rate bound matches the optimal result up to a polylogarithmic factor. \revision{Furthermore, prior result on asymptotic convergence on BMM for nonconvex nonsmooth optimization shows that the limit points are only Nash but not stationary, even with the block separability assumption on the nonsmooth part \cite{xu2013block}}. \revision{We also show that these results continue to hold if convex sub-problems are solved inexactly, allowing for easier practical implementation of the algorithm with the same theoretical guarantees. Such robustness results are not often provided in the literature.} 
    
    %We allow either using strongly convex surrogates, or convex surrogates with possibly non-unique solutions for the sub-problems together with trust-regions of diminishing radii.

    We apply our general framework to various stylized examples such as nonnegative matrix factorization (NMF), nonnegative CANDECOMP/PARAFAC decomposition (NCPD), and block projected gradient descent (BPGD) and get the following results: 
    \begin{description}[itemsep=0.05cm]
        \item[(5)]  For BPGD, we obtain iteration complexity \revision{$\widetilde{O}((1+\rho+\rho^{-1})\eps^{-2})$}.%independent of the smoothness parameter of the smooth objective $f$; 
        
        \item[(6)] A regularized version of the multiplicative update algorithm for nonnegative matrix/tensor factorization with guaranteed asymptotic convergence to stationary points and iteration complexity of \revision{$\widetilde{O}(\eps^{-2})$, where the implicit constant depends on the hyperparameters.
        }
	\end{description}
    We believe that these are the first iteration complexity results for NMF and NCPD as well as BPGD for nonconvex objectives. \commHL{We experimentally validate our theoretical results with both synthetic and real-world data. We show that using trust regions improves the performance of BMM on both matrix factorization and tensor decomposition problems. Moreover, we find that our algorithms outperform the existing ones especially when the matrix and tensors to be factorized are sparse. }

    \revision{Central to our analysis is a novel continuous first-order optimality measure \eqref{eq:stationary3}, by which we bound the worst-case sub-optimality in each iteration by the first-order improvement the algorithm makes.}
 
 %that using auxiliary search radius restriction can in fact improve the rate of convergence over the standard BCD  \eqref{eq:BCD_factor_update}. However, we have not proved that our algorithm converges faster than the standard BCD in general. 

	\subsection{Notation}
    %Define a bivariate function $V:\Param^{2}\rightarrow\R$  as 
  %\begin{align}\label{eq:V_def}
   %     V(\param^{*}, \param):= \left\langle -\nabla f(\param^{*}),\, \param-\param^{*} \right\rangle + p(\param^{*})-p(\param).
  %\end{align}
%which is Lipschitz continuous in both arguments by \ref{assumption:A1}.  
	Denote $a\land b=\min(a\land b)$ for $a,b\in \R$. For a convex function $p:\Param\rightarrow \R$ and $\param\in \Param$, let $\partial p(\theta)$ denote the \textit{subdifferential set}
 \begin{align}
    \partial p(\param) := \{ \eta \in \R^{p}\,:\, p(\param')-p(\param)  \ge \langle \eta,\, \param'-\param \rangle\,\, \textup{for all $\theta\in \Param$} \}. 
\end{align}
    By a slight abuse of notation, we also use $\partial p(\param)$ to denote a generic subdifferential in $\partial p(\param)$. }Throughout this paper, we will denote by $(\param_{n})_{n\ge 1}$ an (possibly inexact) output of Algorithm \eqref{eq:BMM_DR_highlevel}. Fix $n\ge 1$ and $i=1,\dots,m$. Write $\param_{n}=[\theta^{(1)}_{n},\dots,\theta^{(m)}_{n}]$ and for each $\theta\in \R^{I_{i}}$, define 
    \begin{align}\label{eq:def_intermediate_params}
     \param_{n;i} &:= \left(\theta_{n}^{(1)},\cdots,\theta_{n}^{(i-1)},\theta_{n}^{(i)},\theta_{n-1}^{(i+1)},\cdots, \theta_{n-1}^{(m)}\right), \\
     (\theta, \param_{n;i}) &:= \left(\theta_{n}^{(1)},\cdots,\theta_{n}^{(i-1)},\theta,\theta_{n-1}^{(i+1)},\cdots, \theta_{n-1}^{(m)}\right).
    \end{align}
    Then we can write the function $f_{n}^{(i)}$ in \eqref{eq:EBMM}
    as $f_{n}^{(i)}(\theta):= f (\theta, \param_{n;i})$.    
    Also, we denote 
    \begin{align}\label{eq:def_radius_restricted_constraint}
		\Theta^{(i)}_{n}:=\{ \theta\in \Theta^{(i)}\,|\, \lVert \theta-\theta_{n-1}^{(i)} \rVert \le r_{n}  \},
    \end{align}
	which is the constraint set that appears in Algorithm \eqref{eq:BMM_DR_highlevel}. Denote 
	$\Lambda := \{ \param_{n}\,|\, n\ge 1 \}\subseteq \Param$. Also denote $\Lambda_{n}^{\star}$ the set of the exact output of one step of Algorithm \eqref{eq:BMM_DR_highlevel}:
	\begin{align}
		\Lambda_{n}^{\star}:=\left\{ \param_{n}^{\star}=[\theta_{n}^{(1\star)},\dots,\theta_{n}^{(m\star)}] \,:\, \begin{matrix} \textup{$\theta_{n}^{(i\star)}$ is an exact minimizer of $g_{n}^{(i)}$} \\ \textup{over $\Theta^{(i)}_{n}$ for $i=1,\dots,m$} \end{matrix} \right\}. 
	\end{align}
	We will denote a generic element of $\Lambda^{\star}_{n}$ by $\param_{n}^{\star}$. %{\color{red} Define $\tilde{O}$}

	\subsection{Orgnization} 
	
	This paper is organized as follows. We state %the BMM-DR algorithm \eqref{eq:BMM_DR_highlevel} in  Algorithm \ref{algorithm:BMM-DR} and 
 the main results in Section \ref{sec:results}. \commHL{Section \ref{sec:sketch} gives a sketch of the analysis of the complexity results.} In Section  \ref{sec:pf_iteration_complexity}, we prove the iteration complexity results stated in Theorem \ref{thm:complexity} \textbf{(i)}-\textbf{(ii)}. In Section \ref{sec:pf_asymptotic_stationarity}, we prove the asymptotic stationary result stated in Theorem \ref{thm:complexity} \textbf{(iii)}. Then we provide some applications of our theory in Section \ref{sec:app}. Finally, we present the experimental results of the applications in Section \ref{sec:exp}.

	\section{Statement of main results}
 \label{sec:results}

	\subsection{Measure of stationarity}

    Recall that a necessary condition for a point $\param^{*}\in \Param$ to be a local minimizer of the objective $F=f+p$ over $\Param$ is the following first-order optimality condition 
	\begin{align}\label{eq:stationary}
	\sup_{\param\in \Param,\, \lVert \param - \param^{*} \rVert \le 1} \, \langle -\nabla f(\param^{*}) -\partial p(\param^{*}) ,\,  \param - \param^{*} \rangle \le 0,
	\end{align}	
	where $\langle \cdot,\, \cdot \rangle$ denotes the dot product on $\R^{I_{1}+\dots+I_{m}}\supseteq \Param$. Points satisfying the above condition are called the \textit{stationary points} of $F$ over $\Param$. We propose an equivalent condition for first-order stationarity as 
    \begin{align}\label{eq:stationary2}
	\sup_{\param\in \Param,\, \lVert \param - \param^{*} \rVert \le 1} \bigg[ \,V(\param^{*},\param):= \left\langle -\nabla f(\param^{*}),\, \param-\param^{*} \right\rangle + p(\param^{*})-p(\param)\bigg]\le 0. %= 	\sup_{\param\in \Param,\, \lVert \param - \param^{*} \rVert \le 1} \, \langle -\nabla f(\param^{*}) ,\,  \param - \param^{*} \rangle  + p(\param^{*})-p(\param)\le 0, 
	\end{align}	
    To see the equivalence, first note that the convexity of $p$ yields 
    \begin{align}\label{eq:weaker}
        p(\param^{*}) - p(\param) \le \langle -\partial p(\param^{*}),\, \param-\param^{*} \rangle \quad \textup{for all $\param\in \Param$}. 
    \end{align}
    Hence \eqref{eq:stationary} implies \eqref{eq:stationary2}. Conversely, suppose \eqref{eq:stationary2} holds. Let $\varphi(\param)=V(\param^{*}, \param)$ denote the concave function in the supremum in \eqref{eq:stationary2}. Since $\varphi(\param^{*})=0$, \eqref{eq:stationary2} implies that $\param^{*}$ is a local maximizer of $\varphi$ over $\Param$. Then writing the first-optimality condition for $\param^{*}$ being a local minimizer of $-\varphi$ over $\Param$ and noting  $-\partial \varphi(\param^{*})=\nabla f(\param^{*})+\partial p(\param^{*})$ gives exactly \eqref{eq:stationary}. 
    
    An important advantage in using the equivalent stationary measure in \eqref{eq:stationary2} is that the bi-variate function $V$ is \textit{continuous} whenever the non-smooth part $p$ is continuous (e.g., $\ell_{1}$-regularization), whereas the corresponding function in \eqref{eq:stationary} is not. Note that we can still incorporate hard convex constraints $\Param$ while keeping $p$ continuous instead of viewing the discontinuous indicator function of $\Param$ as part of $p$ (see, e.g., \cite{xu2013block}). 

\iffalse
{\color{mycolor}
\begin{align}
    \sup_{\param\in \Param,\, \lVert \param - \param^{*} \rVert \le 1} \bigg[ \,V(\param^{*},\param):= \left\langle -\nabla f(\param^{*}),\, \param-\param^{*} \right\rangle + p(\param^{*})-p(\param)\bigg]\le \epsilon
\end{align}
}
\fi

    For block nonconvex nonsmooth optimization problems as we consider in \eqref{eq:block_minimization}, BCD-type algorithms may not always converge to the stationary points \cite[p.94]{Auslender1976OptimisationM}, but they converge to \textit{Nash points} that are first-order optimal with respect to perturbing a single block coordinate \cite{xu2013block}. Accordingly, we introduce the following notion of \textit{stationary-Nash} points  for the points $\param^{*}$ satisfying the following condition:
    \begin{align}\label{eq:stationary3}
    {%\small
		\sup_{\substack{\param\in \Param \\  \lVert \param - \param^{*} \rVert \le 1} } \bigg[ \,\widetilde{V}(\param^{*},\param):= \left\langle -\nabla f(\param^{*}),\, \param-\param^{*} \right\rangle +  \sum_{i=1}^{m} p(\param^{*})-\revision{p\big(\param^{*}+(\param - \param^*) \mathbf{e}^{(i)}\big)}\bigg]\le 0,
    }
	\end{align}	
    where $\theta^{(i)}$ denotes the $i$th block coordinate of $\param$ and $\mathbf{e}^{(i)}$ is the indicator vector for the $i$th block coordinate. \revision{However, it is important to notice that if the nonsmooth part $p$ is block-separable (i.e., $p(\theta^{(1)},\dots,\theta^{(m)})=\sum_{i=1}^{m} p_{i}(\theta^{(i)})$ for $p_{i}$ convex), then our notion of stationary-Nash points agree with that of stationary points since \eqref{eq:stationary2} and \eqref{eq:stationary3} coincide. While Xu and Yin \cite{xu2013block} and Razaviyayn et al. \cite{razaviyayn2014parallel} make such block-separability assumption for the nonsmooth part of the objective, we do not make such assumption unless otherwise mentioned and aim for the most general setting.}

    For iterative algorithms, a first-order optimality condition may hardly be satisfied exactly in a finite number of iterations, so it is often important to know how many iterations are required until an $\eps$-approximate solution is guaranteed to be obtained. Accordingly, we say $\param^{*}\in \Param$ is an \textit{$\eps$-approximate stationary-Nash point} of $F$ over $\Param$ if 
	\begin{align}\label{eq:stationary_approximate}
		\sup_{\param\in \Param,\, \lVert \param - \param^{*} \rVert \le 1}  \commHL{\,\widetilde{V}(\param^{*},\param)} \le \eps.
	\end{align}	
    For each $\eps>0$ we define the \textit{worst-case iteration complexity} $N_{\eps}$ of an algorithm for solving \eqref{eq:block_minimization} to be the worst-case (w.r.t. initialization $\param_{0}$) number of iterations to guarantee $\eps$-approximate stationary-Nash point of $F$ over $\Param$. \revision{Again, recall that $\widetilde{V}=V$ for block-separable $p$ so $\eps$-approximate stationary-Nash points are $\eps$-approximate stationary points in that case.} 

\revision{
It is worth expanding on some discussion on the choice of optimality measures for nonconvex nonsmooth constrained problems since there are multiple ways to choose them. They are typically obtained by relaxing various equivalent conditions for first-order stationarity. For example, the $\eps$-relaxation of the following stationarity condition 
\begin{align}\label{eq:dist}
    \textup{dist}(\textbf{0},\partial F(\param^*)+ \mathcal{N}_{\Param}(\param^*))\le 0
\end{align}
is standard, where $\mathcal{N}_{\Param}(\param^{*})$ denotes the normal cone of $\Param$ at $\param^{*}$. However, Davis and Drusvyatskiy, in their celebrated work \cite{davis2019stochastic}, noted that the resulting stationarity measure is difficult to work with due to its highly discontinuous nature. They proposed to use the norm of the Moreau envelope as an alternative measure, since it being small implies the existence of an approximate stationary point (in the sense of \eqref{eq:dist}) near the estimated parameter.}

\revision{
It appears to us that our sub-optimality measure $\widetilde{V}$ in \eqref{eq:stationary_approximate} strikes a nice balance between generality and manageability with many desirable properties. First, we have already noted that it is continuous, which allows us to push the convergence analysis much further for nonconvex nonsmooth optimization. Second, it is a direct measure of the sub-optimality of the parameter being evaluated in contrast to the near-stationary measure of Davis and Drusvyatskiy. Third, for nonconvex smooth optimization ($p=0$), it agrees with \eqref{eq:dist} (see \cite[Prop. B.1]{alacaoglu2023convergence}). Fourth, if $p$ is block-separable, then $\widetilde{V}$ coincides with $V$ \eqref{eq:weaker} so stationary-Nash point becomes the usual stationary point. Furthermore, if $f=0$, then $\widetilde{V}$ reduces to the commonly used function value gap measure in convex optimization problems. }

\revision{
We acknowledge that it is still open to obtain iteration complexity for BMM for general nonconvex nonsmooth optimization problems with the stationary measure from \eqref{eq:dist}. The only other related work we are aware of is Razaviyain et al. \cite{razaviyayn2014parallel}, which uses the $\eps$-relaxation of the gradient mapping being zero (also a near-stationarity measure) with block-separable nonsmooth part. 
} %{\color{red} (Yuchen, can you check to make sure that the measure used in \cite{razaviyayn2014parallel} is really weaker than the subgradient measure in \ref{eq:dist}?)}

	\iffalse
	{\color{magenta}	
		The function value gap for the convex minimization problem is dominated by the first-order optimality gap:
		For a convex function $p$, 
		\begin{align}
			p(\theta) - p(\theta_{n}) \ge \langle \partial p(\theta_{n}),\, \theta-\theta_{n} \rangle,
		\end{align}
		so taking infimum over $\theta\in \Theta$ gives 
		\begin{align}
			p^{*} - p(\theta_{n}) \ge \inf_{\theta\in \Theta} \langle \partial p(\theta_{n}),\, \theta-\theta_{n} \rangle.
		\end{align}
		So we get 
		\begin{align}
			p(\theta_{n}) - p^{*}  \le \sup_{\theta\in \Theta} \langle -\partial p(\theta_{n}),\, \theta-\theta_{n} \rangle.
		\end{align}
	}
	\fi

	\subsection{Assumptions}
	
	Throughout this paper, we assume the following conditions:

	\begin{customassumption}{A1} \label{assumption:A1}
		The constraint sets  $\Theta^{(i)}\subseteq \R^{I_{i}}$, $i=1,\dots,m$ are nonempty, closed, and convex (but not necessarily compact) subsets in  $\R^{I_{i}}$. 
	\end{customassumption}
	\begin{customassumption}{A2} \label{assumption:A2}
		The objective function $F=f\commHL{+p}:\Param \rightarrow \R$, where $f$ is continuously differentiable and $p$ is convex and possibly non-smooth. For each compact subset $\Param_{0}\subseteq \Param$, there exist a constant $L=L(\Param_{0})$ such that $\lVert \nabla f(x) - \nabla f(y) \rVert\le L \lVert x-y \rVert$ \commHL{and $|p(x)-p(y)| \le  L \lVert x-y \rVert$}  for all $x,y\in \Param_{0}$. %\commHL{Moreover, $p$ is Lipschitz continuous with parameter $L_p$.} 
    Also, $F^{*}:=\inf_{\param\in \Param} F(\param)>-\infty$. Furthermore, the sub-level sets $F^{-1}((-\infty, a])=\{ \param\in \Param\,:\, F(\param)\le a \}$ for $a\in \R$ are compact. 
	\end{customassumption}

	\noindent	
	In \ref{assumption:A1}, we allow the the constraint set $\Theta^{(i)}$ to be the whole space $\R^{I_{i}}$. 
	The $C^{1}$-assumption of $f$ in \ref{assumption:A2} is weaker than the $L$-smoothness assumption that is standard in the literature of BCD (see, e.g., \cite{xu2013block}).%, lu2017randomized
	
	Next, we define the \textit{majorization gap} as the function  $h_{n}^{(i)}:=g_{n}^{(i)}-f_{n}^{(i)}$  for each $n\ge 0$ and $i=1,\dots,m$. Note that $h_{n}^{(i)}\ge 0$, $h_{n}^{(i)}(\theta_{n-1}^{(i)})=0$. Hence if we assume $h_{n}^{(i)}$ is differentiable, then necessarily $\nabla h_{n}^{(i)}(\theta_{n-1}^{(i)})=0$. %We make the following assumption for the majorization gap. 
	
	\begin{customassumption}{A3} \label{assumption:A3}
       The surrogates  $g_{n}^{(i)}$ for all $n,i$ have $L_{g}$-Lipschitz continuous gradients  for some constant $L_{g}>0$: % and are $\lambda$-strongly convex  for some constant $\lambda>0$: 
            For all $\theta,\theta^{*}\in \Theta^{(i)}$, 
      \begin{align}
		&\hspace{-1cm}\textup{(Lipschitz gradients)}\hspace{1cm}	\lVert \nabla g_{n}^{(i)}(\theta) - \nabla g_{n}^{(i)}(\theta^{*})   \rVert \le L_{g} \lVert \theta-\theta^{*} \rVert. \label{eq:surrogate_L_smoothness}
		%	&\textup{(Strong convexity)}\hspace{0.2cm}	g_{n}^{(i)}(\theta) - g_{n}^{(i)}(\theta^{*}) - \langle \nabla g_{n}^{(i)}(\theta^{*}),\, \theta-\theta^{*} \rangle \ge \frac{\lambda}{2}\lVert \theta - \theta^{*}\rVert^{2}. \label{eq:surrogate_strong_convexity}
	\end{align}
    Furthermore, the surrogates $g_{n}^{(i)}$ are $\rho$-strongly convex for some $\rho\ge 0$ (allowing $\rho=0$). 
    Also, assume that either of the following holds: 
		\begin{description}[itemsep=0.1cm]
			\item[(a)] (\textit{Trust-region used}) \,\, $\sum_{n=1}^{\infty} r_{n}^{2}<\infty$ and $r_{n+1}/r_{n}=O(1)$; or %, then each surrogate $g_{n}^{(i)}$ is co
			
			\item[(b)] (\textit{Trust-region not used}) \,\,  %$\sum_{n=1}^{\infty} r_{n}^{2}=\infty$, 
			$r_{n}\equiv \infty$ and $\rho>0$.
			%In addition, $\nabla g_{n}^{(i)}$ is $L_{g}$-Lipschitz continuous for some constant $L_{g}>0$. 
		\end{description}
	\end{customassumption}
    It is straightforward to extend our analysis to the case where the smoothness parameter $L_{g}$ in \eqref{eq:surrogate_L_smoothness} %and $\lambda$ in \eqref{eq:surrogate_strong_convexity} 
    depends on the block index $i$. For simplicity of presentation, we do not pursue this straightforward generalization.

    Note that the sub-problem of block minimization in Algorithm \eqref{eq:BMM_DR_highlevel} amounts to minimizing convex majorizing surrogate $G_{n}^{(i)}$ over the constraint set $\Theta^{(i)}$ if $r_{n}=\infty$ or the intersection $\Theta^{(i)}\cap \{ \theta \,:\, \lVert \theta - \theta_{n-1}^{(i)} \rVert \le r_{n} \}$ if $r_{n}<\infty$, which are both convex sets. Hence each iteration of Algorithm \eqref{eq:BMM_DR_highlevel} can be readily executed using standard convex optimization procedures (see, e.g., \cite{bertsekas1997nonlinear}). 
    For instance, each iteration of BPGD \eqref{eq:BProxLinear} for smooth objectives can be exactly computed given that projection onto the convex constraint set $\Theta^{(i)}$ has a closed-form expression (e.g., nonnegativity constraints or threshold). 
    
    In the literature of BMM, it is often assumed that $\Delta_{n}\equiv 0$ \cite{razaviyayn2013unified,razaviyayn2014parallel, xu2013block}. %{\color{red}Yuchen, can you check?} 
    However, for many instances of Algorithm \eqref{eq:BMM_DR_highlevel}, it could be the case that the convex sub-problems can only be solved approximately. Fortunately, our analysis of Algorithm \eqref{eq:BMM_DR_highlevel} allows \textit{inexact computation} of solutions to the convex sub-problems, as long as the `optimality gaps' are summable. To be precise, we define the \textit{optimality gap} at iteration $n$ as
	\begin{align}\label{eq:def_sub_optimality_gap}
		\Delta_{n}&=\Delta_{n}(\param_{0}):= \max_{1\le i \le m}  \left( \commHL{G_{n}^{(i)} (\theta^{(i)}_{n}) - G_{n}^{(i)}(\theta_{n}^{(i\star)})} \right),  \\
        &\textup{where \,\, $\theta_{n}^{(i\star)}\in \argmin_{\theta\in \Theta^{(i)} ,\, \lVert \theta - \theta^{(i)}_{n-1} \rVert \le r_{n}} \commHL{G_{n}^{(i)} (\theta) }$}. \nonumber
	\end{align}
    Then we require the following summability of optimality gaps as in \ref{assumption:A4}.

	\begin{customassumption}{A4} \label{assumption:A4}
		The optimality gaps are summable:  $\sum_{n=1}^{\infty} \Delta_{n} <\infty$. 
	\end{customassumption}

We remark that when the surrogates are $\rho$-strongly convex and $L_{g}$-smooth, then one can satisfy \ref{assumption:A4} by using the well-known complexity result for proximal gradient descent \cite[Thm. 10.29]{beck2017first}. Namely, compute each $\theta_{n}^{(i)}$ by running proximal gradient descent for $k_{n}$ sub-iterations with stepsize $\tau<1/L_{g}$ and initialization $\theta_{n-1}^{(i)}$. Then $\Delta_{n}\le \frac{L_{g}}{2}(1-(\rho/L_{g}))^{k_{n}} \sum_{i=1}^{m} \lVert \theta_{n-1}^{(i)} - \theta_{n}^{(i\star)} \rVert^{2}$. In \eqref{eq:theta_theta_star_gap_finite}, we show that the sum of $\sum_{i=1}^{m} \lVert \theta_{n-1}^{(i)} - \theta_{n}^{(i\star)} \rVert^{2}$ over $n\ge 1$ is finite. Hence \ref{assumption:A4} is verified if $(1-(\rho/L_{g}))^{k_{n}}=O(n^{-2})$. For this, it is enough to have $k_{n}\approx \log n$. The sub-iterations of $\log n$ steps for $n=1,\dots,T$ contribute only a $\log T$ factor to the total complexity, which is negligible. %We also remark that the accumulated optimality gap $\sum_{n=1}^{\infty} \Delta_{n}$ only affects the multiplicative constant in the rate of convergence in Theorem \ref{thm:complexity}\textbf{(i)} but not the order of the rate.     

    \noindent %Since the sub-problems in Algorithm \eqref{eq:BMM_DR_highlevel} are convex, they can be solved approximately up to arbitrary prescribed precision. Hence \ref{assumption:A4} can easily be satisfied. For instance, the optimality gap shrinks linearly fast if the sub-problems in Algorithm \eqref{eq:BMM_DR_highlevel} are strongly convex by using coordinate descent algorithms or using gradient descent algorithms, where also superlinear convergence is available using accelerated methods under additional assumptions \cite{wright2015coordinate, bottou2018optimization}. 

	%For  establishing asymptotic stationarity of the iterates produced by Algorithm \eqref{eq:BMM_DR_highlevel}, we need the additional assumption on the smoothness of surrogates:
	
	 %We remark that in case $g_{n}^{(i)}=f_{n}^{(i)}$ for all $n\ge 1$ and $i=1,\dots,m$ (i.e., zero majorization gap), in which case Algorithm \eqref{eq:BMM_DR_highlevel} reduces to block coordinate descent with diminishing radius, \eqref{eq:surrogate_L_smoothness} is automatically satisfied due to Proposition \ref{prop:boundedness}. We also 

	\subsection{Statement of main results}
	
	Now we state the main result, Theorem \ref{thm:complexity}. %To our best knowledge, this gives the first worst-case rate of convergence and iteration complexity of BCD-type algorithms for nonsmooth nonconvex constrained minimization in the literature. 

	\begin{theorem}\label{thm:complexity}
		Assume \ref{assumption:A1}-\ref{assumption:A4} hold.  Let $(\param_{n})_{n\ge 0}$ be an (possibly inexact) output of Algorithm \eqref{eq:BMM_DR_highlevel}. Then the following hold: 
    \vspace{0.1cm}
		\begin{description}[itemsep=0.1cm]
			\item[(i)] (Rate of convergence) There exists constants $M,c>0$ such that for  $n\ge 1$,
		\end{description}
        %\vspace{-0.2cm}
  \begin{align}\label{eq:thm_convergence_bd}
				\min_{1\le k \le n}  	 \,\, \left[ \sup_{\param\in \Param,\, \lVert \param-\param_{k} \rVert\le 1} %\left\langle -\nabla f(\param_{k}),\, \param - \param_{k} \right\rangle + \sum_{i=1} ^{m}p_{k}^{(i)}(\theta_{k}^{(i)})-p_{k}^{(i)}(\theta^{(i)}) 
            \tilde{V}(\param_{k},\param)
   \right]  
   &\le  \revision{c \frac{  M+L_{g}+ \left( \rho^{-1}\land \sum_{k=1}^{n} r_{k}^{2} \right)   }{  (\sqrt{n}/\log n) \land \sum_{k=1}^{n}r_{k}  }},
			\end{align}
           (\commHL{See \eqref{eq:C_A3} and below for explicit expressions for the constants $M,c$.})
  \vspace{0.1cm}
	\begin{description}[itemsep=0.1cm]	
			\item[(ii)] (Worst-case iteration complexity)  If $r_{n}\equiv \infty$, then the worst-case iteration complexity $N_{\eps}$ for Algorithm \eqref{eq:BMM_DR_highlevel} satisfies \commHL{$N_{\eps} = O\big( (1+L_g+\rho^{-1}) \eps^{-2} (\log \eps^{-1})^{2} \big)$. If $r_{n}=1/( \sqrt{n} \log n)$ for $n\ge 1$, then the iteration complexity improves to $N_{\eps} = O((1+L_g)\eps^{-2} (\log \eps^{-1})^{2} )$.}
			
			\item[(iii)] (Asymptotic stationarity) Further assume that  $\sum_{n=1}^{\infty} r_{n}=\infty$. Then $(\param_{n})_{n\ge 1}$ converges to the set of stationary-Nash points of $F$ over $\Param$. \revision{In particular, if $p$ is block-separable, then $(\param_{n})_{n\ge 1}$ converges to the set of stationary points of $F$ over $\Param$. }

		\end{description}
	\end{theorem}

    %\commHL{Revsise this discussion to compare the complexity bounds on the two cases}
    Theorem \ref{thm:complexity}\textbf{(i)} provides a bound on the rate of convergence in terms of the stationarity measure introduced in \eqref{eq:stationary_approximate}. The result covers both options when trust-regions of square-summable radii are used or not throughout the iterations. When trust-region is not used, the asymptotic rate is $O(n^{-1/2}\log n)$. However, the constant $1+L_g +\rho^{-1}$ grows unbounded when the strong convexity parameter $\rho$ for the surrogates is vanishingly small. Therefore, using `flat' surrogates may hinder the rate of convergence. One may try to circumvent this issue by using `steep' surrogates (e.g., by adding large proximal terms), but this will be penalized by $L_g$ in the constant of convergence rate since $L_g\ge \rho$. See also in Fig. \ref{fig:NMF_DR} for numerical results.

    \commHL{The upper bound on the rate of convergence in case of trust-region with diminishing radius in Theorem \ref{thm:complexity}\textbf{(i)} suggests that we can get rid of the unfavorable dependence on $\rho^{-1}$ \textit{with the same surroagates} and \textit{without affecting the rate of convergence.} Indeed, by choosing radii $r_{n}=1/( \sqrt{n} \log n)$ for $n\ge 1$, we get $1 \big/ \sum_{k=1}^{n} \min\{ r_{k},1\} = O(n^{-1/2}\log n)$ and the numerator in the bound is the same constant $1+L_g$ without the term depending on $\rho^{-1}$. To our best knowledge, the best convergence rate for nonconvex nonsmooth block optimization was $O_{\E}((1+L_{g}^{2}\rho^{-1}+\rho^{-1})n^{-1/2})$ for a randomized block coordinate descent in \cite{razaviyayn2014parallel}, and there is no known bounds on convergence rate that holds almost surely, especially without the dependence on $\rho^{-1}$. }

	Theorem \ref{thm:complexity}\textbf{(ii)} gives a worst-case iteration complexity of Algorithm \eqref{eq:BMM_DR_highlevel} of producing and $\eps$-stationary point. This can be easily obtained from Theorem \ref{thm:complexity}\textbf{(i)} by setting the upper bound to be less than $\eps$.

	Lastly, Theorem \ref{thm:complexity}\textbf{(iii)} states that the iterates produced by Algorithm \eqref{eq:BMM_DR_highlevel}, possibly solving the sub-problems inexactly with summable optimality gap, asymptotically converges to the set of \commHL{stationary-Nash} points of the problem \eqref{eq:block_minimization}. \revision{In particular, when the nonsmooth part $p$ is block-separable, this result shows the first asymptotic stationarity of BMM(-DR) iterates in the literature.} 
        %In the special case when all \commHL{stationary-Nash} points of $F$ over $\Param$ are isolated,  global convergence to a single \commHL{stationary-Nash} point can be easily deduced from Theorem \ref{thm:complexity}\textbf{(iii)}.

    The most technical part of our asymptotic analysis is to handle inexact computation when bounded trust regions are used. Roughly speaking, for asymptotic analysis with trust-region, we need to show that the additional trust-region constraints `vanish' in the limit in the sense that any convergent subsequence of the iterates cannot touch the trust-region boundaries indefinitely. Allowing inexact computation of the surrogate minimization within the trust region poses an additional challenge. The analysis is given in Section \ref{sec:pf_asymptotic_stationarity}.

	\section{Sketch of analysis}
	\label{sec:sketch}

	In this section, we give a high-level description of our analysis and discuss the key difficulties and how we will handle them. 

    We first discuss key challenges in the analysis, especially for BMM with trust regions with square-summable radii. 
	For most iteration complexity analysis of first-order methods (e.g., projected gradient descent), one uses first-order optimality of each update w.r.t. the corresponding sub-problem and relates the optimality measure of the overall objective at each iteration with the amount of parameter change \cite{xu2013block}. Using this approach, we can deduce the following bound 
		\begin{align}\label{eq:sketch_main_ineq}
			\sup_{ \substack{ \param\in \Param \\  \lVert \param-\param_{n} \rVert\le (r_{n}\land 1)} }  \sum_{i=1}^{m}\left(-\langle \nabla_i f(\param_{n}) + \partial p_{n}^{(i)}(\theta_{n}^{(i)}), \theta^{(i)}-  \theta^{(i)}_{n} \rangle  \right) = O\left(  \lVert \param_{n}-\param_{n-1}\rVert \right),
		\end{align}
	where we have assumed exact sub-problem solution (i.e., $\Delta_{n}\equiv 0$) for simplicity of presentation.
	When trust-region is not used (i.e., $r_{n}\equiv \infty$), we can use square-summability of the parameter changes (i.e., RHS of \eqref{eq:sketch_main_ineq}) to deduce that the minimum of the LHS of \eqref{eq:sketch_main_ineq} among the first $n$ iterations decay as $\tilde{O}(n^{-1/2})$.

	However, if we do use trust-region (i.e., $r_{n}<\infty$), then the supremum in the LHS of \eqref{eq:sketch_main_ineq} is taken over a vanishingly small ball around $\param_{n}$. For instance, for square-summable radii, the small ball on LHS has a vanishing radius of order $\tilde{O}(n^{-1/2})$, which is the same order of the RHS of \eqref{eq:sketch_main_ineq}. Hence, one cannot deduce a similar rate of convergence result as in the case when trust-region is not used. A different approach to analysis is needed to establish the rate of convergence of BMM-DR.

    In order to circumvent the above issue, we establish the complexity results by first showing the finite first-order variation between consecutive iterates as in Prop. \ref{prop:finite_range_short_points}. Then in the key Lemma \ref{lem:first_order_optimality}, we connect the optimality measure in \eqref{eq:stationary_approximate} with the first-order variation between iterates. This gives an improved upper bound of the LHS in \eqref{eq:sketch_main_ineq}, which decays at the rate of $\tilde{O}(n^{-1})$. Therefore, even when trust-region is used, we are still able to conclude the complexity result as in Thm. \ref{thm:complexity}.

	For the asymptotic analysis with trust-region, we seek a contradiction after supposing there exists a non-stationary limit point of the set of all estimates $\{\param_{n}:n\ge 0\}$. Such a non-stationary limit point should be contained in an open ball that does not contain any other stationary points and there must be infinitely many iterates outside such ball (Prop. \ref{prop:non-stationary_nbh}). This implies that there are infinitely many `crossings' from near the non-stationary limit point to outside of the ball (see Fig. \ref{fig:pf_DR_illustration}). By using techniques developed for the complexity analysis, we can deduce from this that there exists a stationary point inside the open ball around the non-stationary point (Prop. \ref{prop:stationary_conditions}), which is a contradiction. 
	\begin{figure*}[h]
		\centering
		\includegraphics[width=0.9 \linewidth]{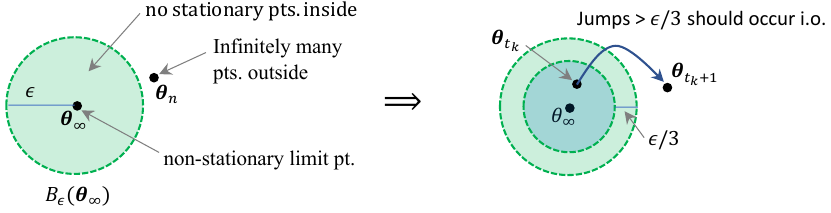}
		\caption{Illustration of the proof of Theorem \ref{thm:complexity} \textbf{(iii)} with diminishing radius.
		}
		\label{fig:pf_DR_illustration}
	\end{figure*}

\vspace{-0.5cm}
	\section{Proof of iteration complexity}
	\label{sec:pf_iteration_complexity}

	We start by recalling a classical lemma on the first-order approximation of functions with Lipschitz gradients. 
	
	\begin{lemma}[First-order approximation of functions with Lipschitz gradient]
		\label{lem:surrogate_L_gradient}
		Let $f:\Omega (\subseteq \R^{p})\rightarrow \R$ be differentiable and $\nabla f$ be $L$-Lipschitz continuous on $\Omega$. Then for each $\theta,\theta'\in \Omega$, $\left| f(\theta') - f(\theta) - \nabla f(\theta)^{T} (\theta'-\theta) \right|\le \frac{L}{2} \lVert \theta-\theta'\rVert^{2}$. 
	\end{lemma}
	
	\begin{proof}
		This is a classical lemma. See \cite[Lem 1.2.3]{nesterov1998introductory}.
	\end{proof}

    Next, we will show the iterates are asymptotically exact given the surrogates are strongly convex.
	
	\begin{prop}\label{prop:parameter_gap_decay}
		\hspace{-0.2em}Suppose \ref{assumption:A1} hold and \commHL{$g_{n}^{(i)}$ is $\rho$-strongly convex for some $\rho>0$.} Then $\frac{\rho}{2}\lVert  \theta_{n}^{(i)}-\theta_{n}^{(i\star)} \rVert^{2} \le \Delta_{n}$. 
	\end{prop}
	
	\begin{proof}
		Fix $i\in \{1,\dots,m\}$. There are two cases to consider. First, when \ref{assumption:A3}\textbf{(a)} holds, the assertion follows from a triangle inequality and that both $\theta_{n}^{(i)}$ and $\theta_{n}^{(i\star)}$ are within distance $r_{n}$ from $\theta_{n-1}^{(i)}$ since $r_{n}=o(1)$. Second, suppose \ref{assumption:A3}\textbf{(b)} holds. Then by the first-order optimality (recall that $\Theta^{(i)}_{n}$ in \eqref{eq:def_radius_restricted_constraint} is convex)
		\begin{align}\label{eq:short_point_stationarity1}
			\langle \partial G_{n}^{(i)}(\theta_{n}^{(i\star)}),\, \theta-\theta_{n}^{(i\star)} \rangle \ge 0 \quad \textup{for all $\theta\in \Theta^{(i)}_{n}$ for all $i=1,\dots,m$. }
		\end{align}
		Then by $\rho$-strong convexity of $g_{n}^{(i)}$ in  \ref{assumption:A3}\textbf{(b)} and convexity of $p_{n}^{(i)}$, 
		\begin{align}\label{eq:second_order_growth_G}
			\frac{\rho}{2}\lVert \theta_{n}^{(i)}-\theta_{n}^{(i\star)} \rVert^{2} &\le G_{n}^{(i)}(\theta_{n}^{(i)}) - G_{n}^{(i)}(\theta_{n}^{(i\star)}) - 	\langle \partial G_{n}^{(i)}(\theta_{n}^{(i\star)}),\, \theta_{n}^{(i)}-\theta_{n}^{(i\star)} \rangle \le  \Delta_{n}. 
		\end{align}
		%It follows that $\lVert  \theta_{n}^{(i)}-\theta_{n}^{(i\star)} \rVert \le \sqrt{2\Delta_{n}/\lambda}$. %According to \ref{assumption:A4}, $\Delta_{n}=o(1)$, so we have $\lVert  \theta_{n}^{(i)}-\theta_{n}^{(i\star)} \rVert=o(1)$.
	\end{proof}

	\begin{prop}[Monotonicity of objective and Stability of iterates]\label{prop:forward_monotonicity}
		Suppose \ref{assumption:A1}, \ref{assumption:A2}, and \ref{assumption:A4} hold. Then the following hold: 
		\begin{description}[itemsep=0.1cm]
			%\item[(i)] $f(\param_{n-1}) - f(\param_{n}) \ge   \frac{\lambda_{n}}{2} \lVert \param_{n-1} - \param_{n} \rVert^{2} - m\Delta_{n}$;
			\item[(i)] $F(\param_{n-1}) - F(\param_{n})  \ge -m\Delta_n$;
			\vspace{0.1cm}
			%\item[(ii)] $\sum_{n=1}^{\infty} \lambda_{n} \rVert \param_{n}-\param_{n-1} \rVert^{2}<\sup_{\param\in \Param} f(\param) + m\sum_{n=1}^{\infty} \Delta_{n} <\infty$.

			\item[(ii)] \commHL{$ \sum_{n=1}^{\infty} \sum_{i=1}^{m} G^{(i)}_{n}(\theta^{(i)}_{n-1}) - G^{(i)}_{n}(\theta^{(i\star)}_{n})  \le  F(\param_{0}) - F^{*} +m\sum_{n=1}^{\infty}\Delta_n<\infty$.}

			\vspace{0.1cm}
			\item[(iii)] If $\sum_{n=1}^{\infty} r_{n}^{2}<\infty$, then $\lVert \theta_{n}^{(i)} - \theta_{n-1}^{(i)} \rVert \le r_{n}$ for all $i,n$ and $\sum_{n=1}^{\infty} \lVert \param_{n}-\param_{n-1} \rVert^{2} 	 \le \sum_{n=1}^{\infty} r_{n}^{2} <\infty$. If 
        $g_{n}^{(i)}$ is $\rho$-strongly convex for some $\rho>0$ for all $i,n$, then
		\end{description}
  \commHL{
  \vspace{-0.5cm}
		\begin{align}\label{eq:parameter_square_sum_bd}
				\frac{\rho}{4}	\sum_{n=1}^{\infty}   \lVert \param_{n}-\param_{n-1} \rVert^{2}  \le   F(\param_{0}) - F^{*} +2m\sum_{n=1}^{\infty}\Delta_n<\infty.
			\end{align}
   }
		\hspace{1.1cm} In particular, in both cases, $\lVert \theta_{n-1}^{(i)} - \theta_{n}^{(i)}\rVert=o(1)$ for all $i=1,\dots,m$.
		
	\end{prop}
	
	\begin{proof}
		Fix $i\in \{1,\dots,m\}$. Note that 
	\begin{align}\label{eq:monotonicity_g_comparison}
	\quad	F^{(i)}_{n}(\theta^{(i)}_{n-1}) - F^{(i)}_{n}(\theta^{(i)}_{n}) &\overset{(a)}{=}   G^{(i)}_{n}(\theta^{(i)}_{n-1}) - G^{(i)}_{n}(\theta^{(i)}_{n})   +G^{(i)}_{n}(\theta^{(i)}_{n}) -  F^{(i)}_{n}(\theta^{(i)}_{n}) \nonumber \\
    &\overset{(b)}{\ge} G^{(i)}_{n}(\theta^{(i)}_{n-1}) - G^{(i)}_{n}(\theta^{(i\star)}_{n}) - \Delta_{n}.
	%&\overset{(b)}{\ge}  -\Delta_n+ G^{(i)}_{n}(\theta^{(i)}_{n}) -  F^{(i)}_{n}(\theta^{(i)}_{n})  \overset{(c)}{\ge} -\Delta_n,
	\end{align}
	where (a) follows from $G_{n}^{(i)}(\theta_{n-1}^{(i)})=F_{n}^{(i)}(\theta_{n-1}^{(i)})$ and  that $g_{n}^{(i)}$ majorizes $f_{n}^{(i)}$ so that $G_{n}^{(i)}$ majorizes $F_{n}^{(i)}$; (b)  follows from the definition of $\Delta_{n}$ in \eqref{eq:def_sub_optimality_gap}. %Noting that $G^{(i)}_{n}(\theta^{(i)}_{n-1}) - G^{(i)}_{n}(\theta^{(i\star)}_{n}) \ge 0$, 
    Summing over all $i=1,\dots,m$, it follows that 
		\begin{align}
			F(\param_{n-1}) - F(\param_{n}) 
			%&\qquad = \sum_{i=1}^{m} f(\theta_{n}^{(1)},\dots, \theta_{n}^{(i-1)},\theta_{n-1}^{(i)},\theta_{n-1}^{(i+1)}, \dots, \theta_{n-1}^{(m)}) - f(\theta_{n}^{(1)},\dots, \theta_{n}^{(i-1)},\theta_{n}^{(i)},\theta_{n-1}^{(i+1)}, \dots, \theta_{n-1}^{(m)}) \\ 
			&  = \sum_{i=1}^{m} F_{n}^{(i)}(\theta^{(i)}_{n-1}) -  F_{n}^{(i)}(\theta^{(i)}_{n}) \\
			&  \geq -m \Delta_n +  \sum_{i=1}^{m}G^{(i)}_{n}(\theta^{(i)}_{n-1}) -  G^{(i)}_{n}(\theta^{(i\star)}_{n}) \ge -m\Delta_n,
			\quad  \label{eq:proximal_monotonicity}
		\end{align}
  where the last inequality uses the definition of $\theta_{n}^{(i\star)}$. 
		This shows \textbf{(i)}. Note that \textbf{(ii)} follows by adding up the first inequality above for $n\ge 1$.

		Lastly, we show \textbf{(iii)}. If $\sum_{n=1}^{\infty} r_{n}^{2}<\infty$,  then the assertion follows immediately. Otherwise, suppose the surrogates $g_{n}^{(i)}$ are $\rho$-strongly convex for some $\rho>0$. %majorization gap $h_{n}^{(i)}$ satisfies the quadratic lower bound. Then by \eqref{eq:majorization_gap_lower_bd}, 
        Then $G_{n}^{(i)}=g_{n}^{(i)}+p_{n}^{(i)}$ is also $\rho$-strongly convex since $p_{n}^{(i)}$ is convex, so by the second-order growth property (see \eqref{eq:second_order_growth_G}) and the definition of $\Delta_{n}$,
        \begin{align}
             \frac{\rho}{2} \lVert \theta_{n-1}^{(i)} - \theta_{n}^{(i\star)}  \rVert^{2} \le G_{n}^{(i)}(\theta_{n-1}^{(i)}) - G_{n}^{(i)}(\theta_{n}^{(i\star)}).
        \end{align}
        Then by \textbf{(ii)}, 
        \begin{align}\label{eq:theta_theta_star_gap_finite}
            \sum_{n=1}^{\infty} \frac{\rho}{2} \lVert \theta_{n-1}^{(i)} - \theta_{n}^{(i\star)}  \rVert^{2}  \le   F(\param_{0}) - F^{*} +m\sum_{n=1}^{\infty}\Delta_n<\infty.
        \end{align}
        Then by Young's inequality and Prop. \ref{prop:parameter_gap_decay}, we can deduce \eqref{eq:parameter_square_sum_bd}. 
	\end{proof}

	\begin{prop}[Boundedness of iterates]\label{prop:boundedness}
		Assume \ref{assumption:A1}, \ref{assumption:A2}, \ref{assumption:A4}, and either \ref{assumption:A3}\textbf{(a)} or \ref{assumption:A3}\textbf{(b)} hold. Then there exists compact and convex subsets $S^{(i)}\subseteq \Theta^{(i)}$ for $i=1,\dots,m$ such that $\Param_{0}:=S^{(1)}\times \dots \times S^{(m)}$ contains $ \bigcup_{n=0}^{\infty} B_{\le 1}(\param_{n})$, 
  where $B_{\le 1}(x):=\{ y\in \Param\,:\, \lVert x-y \rVert\le 1 \}$. Consequently, $\nabla f$ and $p$ are $L$-Lipschitz continuous on $\Param_{0}$ for some $L>0$.  %\commHL{We are now assuming $g_{n}^{(i)}$'s are $L_{g}$-smooth, so proof should be modified}
	\end{prop}
	
	\begin{proof}
		Let $T:=m\sum_{k=1}^{\infty} \Delta_{k}$, which is finite by  \ref{assumption:A4}. Recall that by Proposition \ref{prop:forward_monotonicity}, we have $\sup_{n\ge 0} F(\param_{n}) \le F(\param_{0}) +T<\infty$. It follows that $\{ \param_{n}\,:\, n\ge 0 \}$ is a subset of the sub-level set $A_{0}:=F^{-1}((-\infty, F(\param_{0})+T])$, which is compact by  \ref{assumption:A2}.
        Let $\Pi^{(i)}$ denote the projection from $\Param$ to its $i$th block component $\Theta^{(i)}$. Then $\Pi^{(i)}(A_{0})$ is a compact subset of $\Theta^{(i)}$. Take $R^{(i)}$ to be the `unit fattening' of this compact subset: 
  \begin{align}
        R^{(i)}:= \left\{ \theta\in \Theta^{(i)}\,:\, \lVert \theta-\theta'\rVert\le 1 \,\, \textup{for some $\theta'\in \Pi^{(i)}(A_{0})$}  \right\}. 
  \end{align}
  Now let $S^{(i)}$ be the convex hull of $R^{(i)}$ for $i=1,\dots,m$. Then $S^{(i)}$ is closed and bounded, so is also a compact subset of $\Theta^{(i)}$. %Then \eqref{eq:compact_box_construction} follows for $\Param_{0}:=S^{(1)}\times \dots \times S^{(m)}$. 
  The claimed containment follows from the construction. The second part of the assertion follows from the first part along with \ref{assumption:A1}.% and \ref{assumption:A3}. 
	\end{proof}
	
	\commHL{The following proposition shows the summability of $-V(\param_{n+1},\param_{n})$.}

	\begin{prop}[Finite first variation I]\label{prop:finite_range_short_points}
		Assume \ref{assumption:A1}-\ref{assumption:A4} hold and let $L>0$ be as in Proposition \ref{prop:boundedness}.  %Suppose $\sum_{n=1}^{\infty} r_{n}^{2}<\infty$. 
        Then 
		\begin{align}
			\sum_{n=0}^{\infty}   %\left(\left\langle \nabla f(\param_{n+1}),\,  \param_{n} - \param_{n+1}  \right\rangle + p(\param_{n}) - p(\param_{n+1}) \right) 
            \commHL{-V(\param_{n+1},\param_{n})}
            & \le F(\param_{0}) - F^{*} + \frac{L}{2}\sum_{n=0}^{\infty} \lVert \param_{n}-\param_{n+1} \rVert^{2}. 
		\end{align}
  %\commHL{(Please revise the RHS of the statement w.r.t. the new proof. Are we proving the assertion without the absolute value in the LHS? If so, is this going to be enough for our later part of the arguments? It seems so to me, but please do check it carefully.)}
	\end{prop}
	
	\begin{proof}
        By \ref{assumption:A2} and Proposition \ref{prop:boundedness}, $\nabla f$ is $L$-Lipschitz continuous on the compact subset $\Param_{0}\subseteq \Param$ (in Prop. \ref{prop:boundedness}) that contains all iterates $\param_{n}$, $n\ge 0$. 
		%According to the hypothesis, it follows that $\nabla f$ over $\Param$ is Lipschitz with some uniform Lipshitz constant $L>0$. 
		Hence by Lemma \ref{lem:surrogate_L_gradient}, for all $n\ge 0$,
       % \begin{align}
       % 		-\frac{L}{2}\lVert \param_{n} - \param_{n+1} \Vert^{2}	\le f(\param_{n})-f(\param_{n+1}) -  \left\langle \nabla f(\param_{n+1}),\,  \param_{n} - \param_{n+1}  \right\rangle  \le \frac{L}{2}\lVert \param_{n} - \param_{n+1} \Vert^{2}.
	%\end{align}
  %Then 
\begin{align}
      \left\langle \nabla f(\param_{n+1}),\,  \param_{n} - \param_{n+1}  \right\rangle	&\le f(\param_{n})-f(\param_{n+1}) + \frac{L}{2}\lVert \param_{n} - \param_{n+1} \Vert^{2} \\ 
      & =F(\param_{n})-F(\param_{n+1}) \commHL{-  p(\param_{n}) + p(\param_{n+1})} + \frac{L}{2}\lVert \param_{n} - \param_{n+1} \Vert^{2}.
\end{align}
Adding up the above inequality for all $n\ge 0$ shows the assertion. %and using Proposition \ref{prop:forward_monotonicity} then gives the assertion. 
\end{proof}
\iffalse
Then by Proposition \ref{prop:forward_monotonicity}, we have
    \begin{align}
      \sum_{n=1}^{\infty}(\left\langle \nabla f(\param_{n+1}),\,  \param_{n} - \param_{n+1}  \right\rangle + & p(\param_{n}) - p(\param_{n+1}))\\	&\le (1+\frac{L}{\rho})(F(\param_{0})-F^*)+ \frac{Lm}{\rho}\sum_{n=1}^{\infty}\Delta_n<\infty.
\end{align}  
\fi

	Next, we show a key lemma for establishing the iteration complexity of our algorithm. \commHL{This lemma gives an upper bound of $\widetilde{V}(\param_{n}, \param)$ for any $\param\in \Param$ with $\lVert \param-\param_{n} \rVert\le 1$.}

	\begin{lemma}[Key lemma for iteration complexity]\label{lem:first_order_optimality}
		Assume \ref{assumption:A1}-\ref{assumption:A4} hold.  Let $b_{n}\in  [0, \min\{ r_{n}, 1\}]$ for all $n\ge 1$. Then
  \begin{align}\label{eq:optimality1_proximal0_1}
		%&	b_{n+1}	\sup_{\param\in \Param,\, \lVert \param-\param_{n} \rVert\le 1}\left(  \left\langle -\nabla f(\param_{n}),\, \param-\param_{n} \right\rangle +\sum_{i=1}^m (p(\param_{n+1;i-1})-p_{n+1}^{(i)}(\theta^{(i)})) \right) \\
 \hspace{1cm}  b_{n+1}	\sup_{\param\in \Param,\, \lVert \param-\param_{n} \rVert\le 1} \commHL{\widetilde{V}(\param_{n}, \param)} &\le  \commHL{-V(\param_{n+1},\param_{n})} + 3m L b_{n+1} \lVert \param_{n}-\param_{n+1} \rVert  \\
  &   + \left( \frac{L}{2} + mL \right) \lVert \param_{n+1}-\param_{n} \rVert^{2} + \frac{\revision{L_g} b_{n+1}^{2}}{2} + m\Delta_{n+1}
		\end{align}
  for all $n\ge 1$ for some constant $L>0$. 
	\end{lemma}

	\begin{proof}
		 Fix $n\ge 0$ and let $\param=[\theta^{(1)},\dots,\theta^{(m)}]\in \Param$ be such that $\lVert \param - \param_{n} \rVert \le \min\{b_{n+1},1\}$. Then we have 
\begin{align}
			G_{n+1}^{(i)}( \theta^{(i)}_{n+1} )    -\Delta_{n+1} &\le 
			G_{n+1}^{(i)}( \theta^{(i\star)}_{n+1} )    \le G_{n+1}^{(i)}( \theta^{(i)} ). 
		\end{align}
	Let $\Param_{0}=S^{(1)}\times \dots \times S^{(m)}$ and $L>0$ be as in Proposition \ref{prop:boundedness}. %Then $\param,\param_{n}\in \Param_{0}$ for all $n\ge 0$.  From \ref{assumption:A2} and \ref{assumption:A3}\textbf{(c)}, there exists a constant $L>0$ such that $\nabla f$, $\nabla g_{n}^{(i)}$, and $p$ are $L$-Lipschitz continuous. 
    Recall that $\nabla g_{n+1}^{(i)}(\theta_{n}^{(i)})=\nabla f_{n+1}^{(i)}(\theta_{n}^{(i)})$ by definition of the majorizing surrogates. Hence by subtracting $f_{n+1}^{(i)}(\theta_{n}^{(i)})$ from both sides and applying the $\revision{L_g}$-smoothness of $g_{n+1}^{(i)}$ on $S^{(i)}$ and Lemma \ref{lem:surrogate_L_gradient}, we get
\begin{align}
			\left\langle \nabla f_{n+1}^{(i)}(\theta_{n}^{(i)}),\,  \theta_{n+1}^{(i)}-\theta_{n}^{(i)}  \right\rangle &\le \left\langle \nabla f_{n+1}^{(i)}(\theta_{n}^{(i)}),\,  \theta^{(i)}-\theta_{n}^{(i)}\right\rangle \commHL{+p(\theta; \param_{n+1;i})-p(\param_{n+1;i-1}) }  \\
			&\hspace{-3cm} \commHL{+ p(\param_{n+1;i-1}) -p(\param_{n+1;i})}+ \frac{L}{2}\lVert \theta_{n+1}^{(i)} - \theta_{n}^{(i)} \rVert^{2}+ \frac{\revision{L_g}}{2}\lVert \theta^{(i)} - \theta_{n}^{(i)} \rVert^{2}  +  \Delta_{n+1}.
		\end{align}
%where we denote $p_{n+1}^{(i)}(\theta):=p(\theta, \param_{n+1;i})$ (see \eqref{eq:def_intermediate_params}). 

\iffalse

\begin{align}
			&-\left\langle \nabla f_{n+1}^{(i)}(\theta_{n}^{(i)}),\,  \theta^{(i)}-\theta_{n}^{(i)}\right\rangle+p_{n+1}^{(i)}(\theta_{n}^{(i)})-p_{n+1}^{(i)}(\theta^{(i)})\\  &\le -\left\langle \nabla f_{n+1}^{(i)}(\theta_{n}^{(i)}),\,  \theta_{n+1}^{(i)}-\theta_{n}^{(i)}  \right\rangle +p_{n+1}^{(i)}(\theta_{n}^{(i)})  -p_{n+1}^{(i)}(\theta_{n+1}^{(i)})    \\
			&\qquad + \frac{L}{2}\lVert \theta_{n+1}^{(i)} - \theta_{n}^{(i)} \rVert^{2} + \frac{L}{2}\lVert \theta^{(i)} - \theta_{n}^{(i)} \rVert^{2}  +  \Delta_{n+1}.
		\end{align}
Then 
\begin{align}
			-\left\langle \nabla f_{n+1}^{(i)}(\theta_{n}^{(i)}),\,  \theta^{(i)}-\theta_{n}^{(i)}\right\rangle+p_{n+1}^{(i)}(\theta_{n}^{(i)})-p_{n+1}^{(i)}(\theta^{(i)})  &\le -\left\langle \nabla f_{n+1}^{(i)}(\theta_{n}^{(i)}),\,  \theta_{n+1}^{(i)}-\theta_{n}^{(i)}  \right\rangle +p_{n+1}^{(i)}(\theta_{n}^{(i)})  -p_{n+1}^{(i)}(\theta_{n+1}^{(i)})    \\
			&\qquad + \frac{L}{2}\lVert \theta_{n+1}^{(i)} - \theta_{n}^{(i)} \rVert^{2} + \frac{L}{2}\lVert \theta^{(i)} - \theta_{n}^{(i)} \rVert^{2}  +  \Delta_{n+1}.
		\end{align}
  \fi
  
\noindent Denote $\nabla f_{n+1}(\param_n):=\left[ \nabla f_{n+1}^{(1)}(\theta_{n}^{(1)}),\dots,\nabla f_{n+1}^{(m)}(\theta_{n}^{(m)}) \right]$. Adding up these inequalities for $i=1,\dots,m$, 
		\begin{align}
			\left\langle \nabla f_{n+1}(\param_n),\, \param_{n+1} - \param_{n} \right\rangle &\le   \left\langle \nabla f_{n+1}(\param_n),\,  \param - \param_{n} \right\rangle \commHL{+ \sum_{i=1}^m p(\theta; \param_{n+1;i})-p(\param_{n+1;i-1}) }
		 \\
   &  \hspace{-1cm} \commHL{+p(\param_{n})-p(\param_{n+1}) } + \frac{L}{2}\lVert \param_{n+1} - \param_{n}  \rVert^{2}  + \frac{\revision{L_g}}{2}\lVert \param - \param_{n}\rVert^{2}  + m \Delta_{n+1}.
		\end{align}
        By Lipschitz continuity of $\nabla f$ on $\Param_{0}$, 
		\begin{align}
			\lVert   \nabla_{i} f( \param_{n} ) - \nabla f_{n+1}^{(i)}(\theta_{n}^{(i)}) \rVert ,\, \lVert   \nabla_{i} f( \param_{n+1} ) - \nabla f_{n+1}^{(i)}(\theta_{n}^{(i)}) \rVert & \le L \lVert \param_{n}-\param_{n+1} \rVert.
		\end{align}
		Noting that $\lVert \param-\param_{n}\rVert \le b_{n+1}$, 
  we can deduce 
		\begin{align}\label{eq:key_eq1}
				%\left\langle \nabla f(\param_{n+1}),\,  \param_{n+1} - \param_{n} \right\rangle 
  \qquad V(\param_{n+1},\param_{n})
   &\le   \left\langle \nabla f(\param_{n}),\, \param-\param_{n} \right\rangle +\sum_{i=1}^m p(\theta; \param_{n+1;i})-p(\param_{n+1;i-1}) 
   \\
			& %+p(\param_{n})-p(\param_{n+1})
        \hspace{-1.5cm} + \left( \frac{L}{2} + mL \right) \lVert \param_{n+1}-\param_{n} \rVert^{2} + \frac{\revision{L_g}b_{n+1}^{2}}{2} + m\Delta_{n+1}+ m L b_{n+1} \lVert \param_{n}-\param_{n+1} \rVert.
		\end{align}

%Remark: Possibly rewrite the following 'red' by these 'brown' to make the rescaling more rigorous. 
\begin{comment}
\noindent Recall the above holds for all $\param\in \Param$ with $\lVert \param-\param_{n}\rVert \le b_{n+1}$. Hence
\begin{align}\label{eq:key_eq1}
			\hspace{0.8cm}	\left\langle \nabla f(\param_{n+1}),\,  \param_{n+1} - \param_{n} \right\rangle &\le   \left\langle \nabla f(\param_{n}),\, \param-\param_{n} \right\rangle +\sum_{i=1}^m (p_{n+1}^{(i)}(\theta^{(i)})-p_{n+1}^{(i)}(\theta_{n}^{(i)}))  \\
			& + \left( \frac{L}{2} + mL \right) \lVert \param_{n+1}-\param_{n} \rVert^{2} + m L b_{n+1} \lVert \param_{n}-\param_{n+1} \rVert \\
   & + \frac{L b_{n+1}^{2}}{2} + m\Delta_{n+1} +p(\param_{n})-p(\param_{n+1}) .
		\end{align}
\end{comment}
  
 Now fix arbitrary $\param'\in \Param$ such that $\|\param' - \param_n\|\le 1$. By convexity of $\Param$, $\param := b_{n+1}\param' +(1-b_{n+1})\param_n \in \Param$ and $\|\param-\param_n\|\le b_{n+1}$. \commHL{Then by convexity of $p^{(i)}_{n+1}$, we have $p^{(i)}_{n+1}(\theta^{(i)})\le b_{n+1}p( \theta' ; \param_{n+1;i} )+ (1-b_{n+1})p(\param_{n+1;i-1})$. Therefore 
\begin{align}
    \sum_{i=1}^m (p(\param_{n+1;i-1})-p(\theta; \param_{n+1;i}) %&\ge \sum_{i=1}^m (p(\param_{n+1;i-1})-b_{n+1}p( \theta' ; \param_{n+1;i} )- (1-b_{n+1})p(\param_{n+1;i-1})) \\
    &\ge b_{n+1}\sum_{i=1}^m (p(\param_{n+1;i-1})-p(\theta'; \param_{n+1;i})).
\end{align}
Hence by noting $\param - \param_n = b_{n+1}(\param'-\param_n)$, \eqref{eq:key_eq1} yields for  all $\param'\in \Param$ with  $\|\param' - \param_n\|\le 1$,
\begin{align}\label{eq:key_eq2}
		&	b_{n+1}	\left(  \left\langle -\nabla f(\param_{n}),\, \param'-\param_{n} \right\rangle +\sum_{i=1}^m (p(\param_{n+1;i-1})-p(\theta'; \param_{n+1;i})) \right) \\
   & \quad \le  %	\left\langle \nabla f(\param_{n+1}),\,  \param_{n} - \param_{n+1} \right\rangle+p(\param_{n})-p(\param_{n+1})  
   -V(\param_{n+1}, \param_{n}) + m L b_{n+1} \lVert \param_{n}-\param_{n+1} \rVert  \\
			&\quad \qquad + \left( \frac{L}{2} + mL \right) \lVert \param_{n+1}-\param_{n} \rVert^{2} + \frac{\revision{L_g} b_{n+1}^{2}}{2} + m\Delta_{n+1}.
		\end{align}
}
Lastly, by using $L$-Lipschitz continuity of $p$ on $\Param_{0}$, we can replace the left-hand side above by $b_{n+1}\widetilde{V}(\param_{n},\param')$ with an additive error of $2b_{n+1}L \lVert \param_{n+1}-\param_{n} \rVert$. This 
%\begin{align}
%\left| (p(\param_{n+1;i-1})-p(\theta'; \param_{n+1;i})) - ( p(\param_{n+1}) - p(\param_{n+1}+\theta'\mathbf{e}_{i}) ) \right| \le 2L_{p}  \lVert \param_{n+1}-\param_{n} \rVert.
%\end{align}
gives the assertion.
	\end{proof}

	\begin{lemma}\label{lem:positive_convergence_lemma}
		\hspace{-0.5em}Let $(a_{n})_{n\ge 0}$ and $(b_{n})_{n \ge 0}$ be sequences of nonnegative real numbers such that $\sum_{n=0}^{\infty} a_{n}b_{n} <\infty$. Then $\min_{1\le k\le n} b_{k} \le \frac{\sum_{k=0}^{\infty} a_{k}b_{k}}{\sum_{k=1}^{n} a_{k}}  = O\left( \left( \sum_{k=1}^{n} a_{k} \right)^{-1} \right)$. 
	\end{lemma}
	
	\begin{proof}
		The assertion follows from noting that
		\begin{align}
			\left( \sum_{k=1}^{n}a_{k} \right) \min_{1\le k \le n} b_{k}\le \sum_{k=1}^{n} a_{k}b_{k} \le  \sum_{k=1}^{\infty} a_{k}b_{k} <\infty.
		\end{align}
	\end{proof}
	
	Now we are ready to derive the iteration complexity in Theorem \ref{thm:complexity} \textbf{\textup{(i)}}-\textbf{\textup{(ii)}}.
	
	\vspace{0.1cm}
	\begin{proof}[\text{Proof of Theorem \ref{thm:complexity}} \textbf{\textup{(i)}}-\textbf{\textup{(ii)}}]
		Suppose $b_{n}\in [0, \min\{ r_{n},1 \}]$ is such that \newline $\sum_{n=1}^{\infty} b_{n}^{2}<\infty$. Introduce the following notations: $A:= F(\param_{0})-F^{*}$ and 
		\begin{align}	\label{eq:ABCDE}
		B:=\sum_{n=1}^{N} \Delta_{n},\,  C:= \sum_{n=1}^{N} \lVert \param_{n}-\param_{n-1} \rVert^{2} ,\, D:= \sum_{n=1}^{N} b_{n}^{2},  \, E:= \sum_{n=1}^{N} r_{n}^{2}.
		\end{align}
		Then $A\in [0,\infty)$ by \ref{assumption:A2}, $B\in [0,\infty)$ by \ref{assumption:A4} and  $D\in [0,\infty)$ by the hypothesis. %and $C\in  [0,\infty)$ by Proposition \ref{prop:forward_monotonicity}.  
        Furthermore, by Cauchy-Schwarz inequality, 
		\begin{align}
			\sum_{n=0}^{N} b_{n+1} \lVert \param_{n}-\param_{n+1} \rVert  \le  \sqrt{C D}. % < \infty. 
		\end{align}
		Now summing the inequalities  in \eqref{eq:optimality1_proximal0_1} in Lemma \ref{lem:first_order_optimality} for all $n\ge 0$ and using Proposition \ref{prop:finite_range_short_points}, we get \begin{align}\label{eq:iteration_complexity_finite_sum_pf}
			&\sum_{n=0}^{N}  b_{n+1}	\sup_{\param\in \Param,\, \lVert \param-\param_{n} \rVert\le 1}  %\left\langle -\nabla f(\param_{n}),\, \param-\param_{n} \right\rangle  \\
            \widetilde{V}(\param_{n},\param) 
			%& \qquad \le   \sum_{n=0}^{\infty} 	%\left\langle -\nabla f(\param_{n+1}),\,  \param_{n+1} - \param_{n} \right\rangle   
           % V(\param_{n+1},\param_{n}) + m (L+2L_{p}) \sum_{n=0}^{\infty} b_{n+1} \lVert \param_{n}-\param_{n+1} \rVert \\
			%&\hspace{2cm} + \left( \frac{L}{2} + mL \right) \sum_{n=0}^{\infty} \lVert \param_{n+1}-\param_{n} \rVert^{2} + \frac{L }{2} \sum_{n=0}^{\infty} b_{n+1}^{2}  
        \le M_{0}, %<\infty, 
		\end{align}
		where the constant $M_{0}$ is defined as 
        \begin{align}\label{eq:M0}
        M_{0}:= %A + \frac{L D}{2} + 3mL \sqrt{C D} + \left( \frac{L}{2} + mL \right) D + \frac{L_{g}}{2} C + m B
        A + m B +  (m+1)L C  + \frac{L_{g}}{2} D  + 3mL \sqrt{C D}. 
        \end{align}
Next, by Prop. \ref{prop:forward_monotonicity} and \ref{assumption:A3}, we have
        \begin{align}\label{eq:C_A3}
        C\le \begin{cases}
				\revision{E \land  \frac{4}{\rho}(A + 2mB)} &\hspace{-0.1em}\textup{if \ref{assumption:A3}\textbf{(a)} holds,} \\
				\frac{4}{\rho}(A + 2mB) &\hspace{-0.1em}\textup{if \ref{assumption:A3}\textbf{(b)} holds}.
			\end{cases}
        \end{align}
{\color{mycolor}Therefore, for either case we have $C\in [0,\infty)$ and hence $M_0<\infty$. Now Theorem \ref{thm:complexity} \textbf{(i)} is a direct consequence of \eqref{eq:iteration_complexity_finite_sum_pf} and Lemma \ref{lem:positive_convergence_lemma}. Noting when \ref{assumption:A3}\textbf{(a)} holds we have $D\le E$. Combining \eqref{eq:M0} and \eqref{eq:C_A3} gives
\begin{align}\label{eq:M0_bound}
\small
        M_0 \le \begin{cases}
        M_1 + \frac{L(m+1)\left( 4A + 8mB \right)}{\rho} +  \frac{6mL\sqrt{(A+2mB)D}}{\rho^{1/2}} + \frac{L_g}{2} D &\hspace{-0.1em}\textup{if \ref{assumption:A3}\textbf{(b)} holds},\\
				M_1 + \min\big\{(4m+1)L E, H(\rho)\big\} + \frac{L_g}{2} E , &\hspace{-0.1em}\textup{if \ref{assumption:A3}\textbf{(a)} holds,} 	
			\end{cases}
        \end{align}
where $H(\rho)$ denotes the sum of the second and the third terms in the first case, $M_1=A+mB$ is a constant independent of $L_g$ and $\rho$. 
   } %We refer the readers to Remark \ref{remark:constants} for details on improvement of the complexity with diminishing radius, as well as how $\rho$ and $r_n$ affect the bound.

		Next, \textbf{(ii)} is a direct consequence of \textbf{(i)}. Indeed, if $r_{n}=1/(\sqrt{n}\log n)$, then the  upper bound on the rate of convergence in \textbf{(i)} is of order $O(1/\sum_{k=1}^{n} k^{-1/2}/(\log k) )= O(\log n \big/ 2\int_{1}^{n} x^{-1/2}\,dx)=O(\frac{\log n}{2\sqrt{n}})$. Similarly, if $r_{n}\equiv \infty$, then we can choose $b_{n} =n^{-1/2}/(\log n) $ for $n\ge 1$. Then we have the same rate of convergence in \textbf{(i)}. Then one can conclude by using the fact that $n\ge 2\eps^{-1} (\log \eps^{-1})^{2}$ implies $(\log n)^{2}/n \le \eps$ for all sufficiently small $\eps>0$.
	\end{proof}
\begin{remark}[Remark on diminishing radius]
     A direct comparison of the constants in \eqref{eq:M0_bound} shows the advantage of diminishing radius for improving the complexity bound by eliminating terms with $\rho^{-1}$, which is significant when $\rho$ is small. When $\rho$ is large, the algorithm with diminishing radius performs at least as well as the one without DR. Note that a large $\rho$ will be penalized by $L_g$ since $L_g\ge \rho$.  Moreover, when implementing diminishing radius, one cannot take $r_n$ arbitrarily small since $\sum_{n=1}^\infty b_n$ appears in the denominator of the complexity.
    Furthermore, the choice of diminishing radius $r_n$ (and the corresponding $b_n$) in the proof of Theorem \ref{thm:complexity} is not limited to $r_n = 1/(\sqrt{n}\log n)$.
    \iffalse
    In fact, one can take $r_n = n^{-\beta}(\log n)^{-\alpha}$ for any $\alpha\ge 0, \beta<1$. To see this, %one can first replace the infinite series in \eqref{eq:ABCDE} by a partial sum to $N$. 
    note that a similar analysis gives 
 \vspace{-0.5cm}
  \begin{align}
        \frac{\sum_{n=0}^{N} r_{n}^{2} }{\sum_{n=0}^{N} r_{n}} =   \frac{\sum_{n=0}^{N} n^{-2\beta}(\log n)^{-2\alpha} }{\sum_{n=0}^{N} n^{-\beta}(\log n)^{-\alpha} } =  \begin{cases}
            \widetilde{O}(\frac{ N^{1-2\beta} }{ N^{1-\beta} }) =\widetilde{O}(N^{-\beta}),  \;\; \textup{if $\beta<0.5$},
            \\
            \widetilde{O}(\frac{ 1 }{ N^{1-\beta} })= \widetilde{O}(N^{\beta-1}),  \;\; \textup{if $0.5<\beta<1$}.
        \end{cases} 
    \end{align}
    Combining the above and using Lemma \ref{lem:positive_convergence_lemma} give the convergence rate. Note that the optimal rate is achieved when $\beta=0.5$, which is the same as in the proof of Thm. \ref{thm:complexity}. Later in the numerical experiments, we test the performance of BMM-DR with different $\beta$, see Sec. \ref{sec:exp} for details. 
    \fi
\end{remark}

 \iffalse
    which gives 
    \begin{align}\label{eq:finite_sum}
			&\sum_{n=0}^{N}  b_{n+1}	\sup_{\param\in \Param,\, \lVert \param-\param_{n} \rVert\le 1}  
            \widetilde{V}(\param_{n},\param) 
        \le M + \sum_{n=0}^{N} r_n^2, 
		\end{align}
  for some $M>0$. 
  \fi

	\section{Proof of asymptotic stationarity} 
	\label{sec:pf_asymptotic_stationarity}

	Recall that after the update $\param_{n-1}\mapsto \param_{n}$,  each block coordinate of $\param_{n}$ and $\param_{n}^{\star}$ are within distance $r_{n}$ from the corresponding block coordinate of $\param_{n-1}$. For each $n\ge 1$, we say $\param_{n}^{\star}$ is a \textit{short point} if all of its block coordinates are strictly within $r_{n}$ from the corresponding block coordinate of $\param_{n-1}$, and  $\param_{n}^{\star}$ is said to be a \textit{long point} otherwise. Observe that if $\param_{n}^{\star}$ is a short point, then imposing the search radius restriction in Algorithm \eqref{eq:BMM_DR_highlevel} has no effect and $\param_{n}^{\star}$ is obtained from $\param_{n-1}$ by a single cycle of exact block majorization-minimization on the constraint set $\Param$. In particular, this holds if $r_{n}=\infty$ since then every $\param_{n}^{\star}\in \Lambda^{\star}$ must be a short point. 
	
    For simplicity of notations, let $h_{n}^{(i)} = g_{n}^{(i)} - f_{n}^{(i)}$ denote the majorization gap. In the following proposition, we show the majorization gap is vanishing.
	\begin{prop}[Vanishing gradients of the majorization gap]\label{prop:gradient_function_bd}
		Suppose \ref{assumption:A1} - \ref{assumption:A3} hold. Then there exists a constant $L_{h}>0$ such that for all $n\ge 1$ and $i=1,\dots,m$, 
		\begin{align}
			%\lVert \nabla h_{n}^{(i)}(\theta_{n}^{(i)})  \rVert^{2} \le 2 L_{h} h_{n}^{(i)}(\theta_{n}^{(i)}). 
   \lVert \nabla h_{n}^{(i)}(\theta_{n}^{(i)})  \rVert \le L_{h} \| \theta_{n}^{(i)}- \theta_{n-1}^{(i)}\|. 
		\end{align}
	\end{prop}
	
	\begin{proof}
		Write $h=h_{n}^{(i)}$. Let $\Param_{0}=S^{(1)}\times \dots \times S^{(m)}$ by as in Proposition \ref{prop:boundedness}. Then $\param_{n}\in \Param_{0}$ for all $n\ge 1$. %Then $\nabla h$  is $L_{h}$-Lipschitz continuous on $S^{(i)}$ for some uniform constant $L_{h}>0$ by \ref{assumption:A3}\textbf{(c)}. 
  Since $\theta_{n}^{(i)}, \theta_{n-1}^{(i)}\in S^{(i)}$ and $\nabla h(\theta_{n-1}^{(i)})=0$, we get %$ \lVert \nabla h(\theta_{n}^{(i)})  \rVert = \lVert \nabla h(\theta_{n}^{(i)})- \nabla h(\theta_{n-1}^{(i)})  \rVert \le L_h \| \theta_{n}^{(i)}- \theta_{n-1}^{(i)}\|$, as desired. 
  \begin{align}
      \lVert \nabla h(\theta_{n}^{(i)})  \rVert &= \lVert \nabla h(\theta_{n}^{(i)})- \nabla h(\theta_{n-1}^{(i)})  \rVert\\ &= \| \nabla g_{n}^{(i)}(\theta_{n}^{(i)}) - \nabla f_{n}^{(i)}(\theta_{n}^{(i)}) -\nabla g_{n}^{(i)}(\theta_{n-1}^{(i)}) + \nabla f_{n}^{(i)}(\theta_{n-1}^{(i)} )\|\\
      &\le (L_g + L) \| \theta_{n}^{(i)}- \theta_{n-1}^{(i)}\|.
  \end{align}
	\end{proof}

	\begin{prop}[Sufficient condition for stationarity I]\label{prop:long_points_stationary}
		Assume \ref{assumption:A1}-\ref{assumption:A4} hold. If $(\param_{n_{k}}^{\star})_{k\ge 1}$ is a convergent subsequence of $(\param_{n}^{\star})_{n\ge 1}$  consisting of short points, then $\lim_{k\rightarrow\infty} \param_{n_{k}} = \lim_{k\rightarrow\infty} \param_{n_{k}}^{\star}=:\param_{\infty}$ and $\param_{\infty}$ is a stationary-Nash point of $F$ over $\Param$. 
	\end{prop}
	
	\vspace{-0.3cm}
	\begin{proof}

		Note that  if $\param_{n}^{\star}$ is a short point, then each block coordinate $\theta_{n}^{(i\star)}$ lies in the interior of $\lVert \param-\param_{n-1}^{(i)} \rVert < r_{n}$, the trust region of radius $r_{n}$. Hence by the first-order optimality condition for $\theta_{n}^{(i\star)}$, 
%\begin{align}\label{eq:short_point_stationarity}
		%	\langle \nabla g_{n}^{(i)}(\theta_{n}^{(i\star)}),\, \theta-\theta_{n}^{(i\star)} \rangle \ge 0 \quad \textup{for all $\theta\in \Theta^{(i)}$ for all $i=1,\dots,m$. }
		%\end{align}		
  for all $\theta\in \Theta^{(i)}$ for all $i=1,\dots,m$,
    \begin{align}\label{eq:short_point_stationarity}
			\langle \nabla g_{n}^{(i)}(\theta_{n}^{(i\star)}),\, \theta-\theta_{n}^{(i\star)} \rangle + \commHL{p_n^{(i)}(\theta)-p_n^{(i)}(\theta_{n}^{(i\star)}) \ge \langle \partial G_{n}^{(i)}(\theta_{n}^{(i\star)}),\, \theta-\theta_{n}^{(i\star)} \rangle} \ge 0.
		\end{align}

		We wish to show that $\param_{\infty}$ is a stationary-Nash point of $F$ over $\Param$. Denote $\param_{\infty}=[\theta_{\infty}^{(1)},\dots,\theta_{\infty}^{(m)}]$. 	Fix $i\in \{1,\dots,m\}$ and $\theta\in \Theta^{(i)}$ with $\|\theta -\theta_\infty^{(i)}\|\le 1$. By Proposition \ref{prop:boundedness} and continuous differentiability of $f$ in \ref{assumption:A2}, there exists a constant $C>0$ such that $\lVert \nabla f_{n}^{(i)}(\theta_{n}^{(i)}) \rVert\le C$ for all $n,i$. Then by using Cauchy-Schwartz inequality and Proposition \ref{prop:gradient_function_bd},  %and $L_{g}$-Lipscthiz continuity  of $\nabla g_{n}^{(i)}$, 
        %{for large $n$ we have} %\commHL{why $n$ large?} 
		\begin{align}
			&\hspace{-1cm}\langle \nabla f_{n}^{(i)}(\theta_{n}^{(i)}),\, \theta-\theta_{n}^{(i)} \rangle \\
			\ge& 	\langle \nabla f_{n}^{(i)}(\theta_{n}^{(i)}),\, \theta-\theta_{n}^{(i\star)} \rangle - \lVert \nabla f_{n}^{(i)}(\theta_{n}^{(i)}) \rVert\, \lVert \theta_{n}^{(i)} - \theta_{n}^{(i\star)} \rVert \\ 
			\ge& \langle \nabla g_{n}^{(i)}(\theta_{n}^{(i)}),\, \theta-\theta_{n}^{(i\star)} \rangle - \lVert \nabla h_{n}^{(i)}(\theta_{n}^{(i)}) \rVert \, \lVert \theta - \theta_{n}^{(i)}   \rVert  -  C\, \lVert \theta_{n}^{(i)} - \theta_{n}^{(i\star)} \rVert \\
          \ge& 	\langle \nabla g_{n}^{(i)}(\theta_{n}^{(i)}),\, \theta-\theta_{n}^{(i\star)} \rangle - L_{h}\lVert \theta_{n}^{(i)}-\theta_{n-1}^{(i)} \rVert \, \lVert \theta - \theta_{n}^{(i)}   \rVert  -  C\, \lVert \theta_{n}^{(i)} - \theta_{n}^{(i\star)} \rVert.
			%&\qquad  - \lVert \edit{\nabla g_{n}^{(i)}(\theta_{n}^{(i)})}-\nabla f_{n}^{(i)}(\theta_{n}^{(i)}) \rVert\, \lVert \theta_{n}^{(i)} - \theta_{n}^{(i\star)} \rVert  - \edit{2}\lVert \nabla h_{n}^{(i)}(\theta_{n}^{(i)}) \rVert. 
		\end{align}

		We apply the last inequality for $n$ replaced with $n_{k}$. By Proposition \ref{prop:parameter_gap_decay},  we have $\param_{n_{k}}\rightarrow \param_{\infty}$ as $k\rightarrow\infty$. Also by Proposition \ref{prop:forward_monotonicity}, we have $\lVert \param_{n-1}-\param_{n}\rVert= o(1)$. So $\param_{n_{k}-1}\rightarrow \param_{\infty}$ as $k\rightarrow\infty$. Since $\lVert \theta - \theta_{\infty}^{(i)} \rVert\le 1$, it follows that $\lVert \theta - \theta_{n_{k}}^{(i)} \rVert\le 2$ for all sufficiently large $k$. Then by using continuity of $\nabla f$, $\nabla g_{n}^{(i)}$ and $p_{n}^{(i)}$ as well as \eqref{eq:short_point_stationarity}, %we deduce

\begin{align} 
			&\langle \nabla_{i} f(\param_{\infty}), \, \theta-\theta_{\infty}^{(i)} \rangle + p_{\infty}^{(i)}(\theta) - p_{\infty}^{(i)}(\theta_{\infty}^{(i)})\\
   &=
			\liminf_{k\rightarrow\infty}   	\left(\langle \nabla f_{n_{k}}^{(i)}(\theta_{n_{k}}^{(i)}),\, \theta-\theta_{n_{k}}^{(i)} \rangle + p_{n_{k}}^{(i)}(\theta) - p_{n_{k}}^{(i)}(\theta_{n_{k}}^{(i)}) \right) \\
			&\ge 	\liminf_{k\rightarrow\infty} \left(  	\langle \nabla g_{n_{k}}^{(i)}(\theta_{n_{k}}^{(i\star)}),\, \theta-\theta_{n_{k}}^{(i\star)}\rangle+ p_{n_{k}}^{(i)}(\theta) - p_{n_{k}}^{(i)}(\theta_{n_{k}}^{(i)}) \right)  \ge 0,
		\end{align}
	\noindent where $p_{\infty}^{(i)}(\theta) := p(\theta_{\infty}^{(1)},\cdots,\theta_{\infty}^{(i-1)},\theta,\theta_{\infty}^{(i+1)},\cdots,\theta_{\infty}^{(m)})$. This holds for all $i$ so we can conclude.
 
	\end{proof}

	Now we can deduce Theorem \ref{thm:complexity} \textbf{(iii)} for the case of BMM with strongly convex surrogates without diminishing radius.

	\begin{proof}[\text{Proof of Theorem \ref{thm:complexity} \textup{(iii)} under \ref{assumption:A3}\textbf{\textup{(b)}}}]

		By Proposition \ref{prop:boundedness}, the sequence of iterates $(\param_{n})_{n\ge 1}$ is bounded. Hence it suffices to show that every limit point of this sequence is a stationary-Nash point of $F$ over $\Param$. To this effect, let $(\param_{n_{k}})_{k\ge 1}$ denote an arbitrary convergent subsequence of $(\param_{n})_{n\ge 1}$. Denote $\param_{\infty}:=\lim_{k\rightarrow\infty} \param_{n_{k}}$. By Proposition \ref{prop:parameter_gap_decay}, we have $\param_{n_{k}}^{\star}\rightarrow \param_{\infty}$ as $k\rightarrow \infty$.  Since $r_{n}\equiv \infty$ under \ref{assumption:A3}\textbf{(b)}, each $\param_{n_{k}}^{\star}$ is a short point. Thus by Proposition \ref{prop:long_points_stationary}, $\param_{\infty}$ is a stationary-Nash point of $F$ over $\Param$, as desired. 
	\end{proof}
	
	In the remainder of this section, we prove Theorem \ref{thm:complexity}\textbf{(iii)} under \ref{assumption:A3}\textbf{(a)}. The key technical difficulty is to handle a sequence of inexact iterates $(\param_{n})_{n\ge 1}$ where the exact iterates $\param_{n}^{\star}$ can either be short or long points due to the non-degenerate radius $r_{n}<\infty$ of the trust-region. %\edit{Revise the rest for long points $\param_{n}^{\star}$}

	\begin{prop}[Finite first-order variation II]\label{prop:finite_range_short_points2}
		Assume \ref{assumption:A1}, \ref{assumption:A2}, \ref{assumption:A3}\textbf{(b)} and \ref{assumption:A4} hold.  Then $ \sum_{n=0}^{\infty} \widetilde{V}(\param_{n}, \param_{n+1}^{\star})
			<\infty$. 
		%\begin{align}
			%\sum_{n=0}^{\infty} \sum_{i=1}^{n}  \left| \left\langle \nabla f_{n+1}^{(i)}(\theta_{n}^{(i)}) ,\, \theta_{n}^{(i)} - \theta_{n+1}^{(i\star)}  \right\rangle \right
			%\sum_{n=0}^{\infty}	\left( \left\langle \nabla f(\param_{n}) ,\, \param_{n} - \param_{n+1}^{\star}  \right\rangle + \sum_{i=1}^m \left(p_{n+1}^{(i)}(\theta_{n}^{(i)}) - p_{n+1}^{(i)}(\theta_{n+1}^{(i\star)}) \right) \right)
         % \sum_{n=0}^{\infty} \widetilde{V}(\param_{n}, \param_{n+1}^{\star})
		%	<\infty.
		%\end{align}
	\end{prop}
	
	\begin{proof}
        Let $L_{g}:=L+L_{h}>0$ as in Proposition \ref{prop:boundedness}. 
		We first note that since $\nabla f_{n+1}^{(i)}(\theta_{n}^{(i)})=\nabla g_{n+1}^{(i)}(\theta_{n}^{(i)})$, 
		\begin{align}
			\left\langle \nabla f_{n+1}^{(i)}(\theta_{n}^{(i)}) ,\, \theta_{n}^{(i)} - \theta_{n+1}^{(i\star)}  \right\rangle  
&=  \left\langle \nabla g_{n+1}^{(i)}(\theta_{n}^{(i)}) ,\, \theta_{n}^{(i)} - \theta_{n+1}^{(i\star)}  \right\rangle  \\       %&\overset{(a)}{\le} \frac{L_{g}}{2} \lVert \theta_{n}^{(i)} - \theta_{n+1}^{(i\star)} \rVert^{2}  + g_{n+1}^{(i)}(\theta_{n}^{(i)}) - g_{n+1}^{(i)}(\theta_{n+1}^{(i\star)})  \\
			& \hspace{-3cm} \overset{(a)}{\le} \frac{L_{g}}{2} \lVert \theta_{n}^{(i)} - \theta_{n+1}^{(i\star)} \rVert^{2}  + G_{n+1}^{(i)}(\theta_{n}^{(i)}) - G_{n+1}^{(i)}(\theta_{n+1}^{(i\star)}) \commHL{-p_{n+1}^{(i)}(\theta_{n}^{(i)}) + p_{n+1}^{(i)}(\theta_{n+1}^{(i\star)})}  \\
			& \hspace{-3cm} \overset{(b)}{\le}  \frac{L_{g} r_{n+1}^{2}}{2}  + G_{n+1}^{(i)}(\theta_{n}^{(i)}) -  G_{n+1}^{(i)}(\theta_{n+1}^{(i)})   +\Delta_{n+1}\commHL{- p_{n+1}^{(i)}(\theta_{n}^{(i)}) + p_{n+1}^{(i)}(\theta_{n+1}^{(i\star)})} ,
			%& \overset{(d)}{\le}  \frac{L_{g}}{2} \lVert \theta_{n}^{(i)} - \theta_{n+1}^{(i\star)} \rVert^{2}  + g_{n+1}^{(i)}(\theta_{n}^{(i)}) -  g_{n}^{(i)}(\theta_{n}^{(i)})  + \left( g_{n}^{(i)}(\theta_{n}^{(i)}) - g_{n+1}^{(i)} (\theta_{n+1}^{(i)})  \right)  +\Delta_{n+1}.
		\end{align} 
		where (a) follows from Lemma \ref{lem:surrogate_L_gradient} and \ref{assumption:A3} and definition of $G_{n+1}^{(i)}$ and (b) follows from the trust-region constraint and the definition of $\Delta_{n+1}$  in \eqref{eq:def_sub_optimality_gap}. 
		
		\commHL{Next, recalling  $G_{n+1}^{(i)}(\theta_{n}^{(i)})=F_{n+1}^{(i)}(\theta_{n}^{(i)})$ and $G_{n}^{(i)}\ge F_{n}^{(i)}$, we get 
		\begin{align}
			\sum_{i=1}^{m}	G_{n+1}^{(i)}(\theta_{n}^{(i)}) - G_{n+1}^{(i)}(\theta_{n+1}^{(i)}) &= 		\sum_{i=1}^{m} F_{n+1}^{(i)}(\theta_{n}^{(i)}) - G_{n+1}^{(i)}(\theta_{n+1}^{(i)}) \\
			&\le  		\sum_{i=1}^{m} F_{n+1}^{(i)}(\theta_{n}^{(i)}) - F_{n+1}^{(i)}(\theta_{n+1}^{(i)}) 
			%&= f(\theta_{n+1}^{(1)},\dots, \theta_{n+1}^{(i-1)}, \theta_{n}^{(i)},\dots,\theta_{n}^{(m)})  - f(\theta_{n+1}^{(1)},\dots, \theta_{n+1}^{(i-1)}, \theta_{n+1}^{(i)}, \theta_{n}^{(i+1)},\dots,\theta_{n}^{(m)})   \\
			 = F(\param_{n}) - F(\param_{n+1}). 
		\end{align}		
		Then summing the above inequality over $n\ge 0$ is at most %a telescoping sum, which is at most 
    $F(\param_{0})-F^{*}$, so this shows
		\begin{align}\label{eq:key_lemma_II_claim1}
			\sum_{n=0}^{\infty}	\sum_{i=1}^{m}	G_{n+1}^{(i)}(\theta_{n}^{(i)}) - G_{n+1}^{(i)}(\theta_{n+1}^{(i)}) \le F(\param_{0}) - F^{*}. 
		\end{align}
        }
		
		Lastly, using $L$-Lipschitz continuity of $\nabla f$ on $\Param_{0}$ (see Prop \ref{prop:boundedness} and \ref{assumption:A2}), 
		\begin{align}
			&\hspace{-1cm}\sum_{i=1}^{m} 	 \left\langle \nabla f_{n+1}^{(i)}(\theta_{n}^{(i)}) ,\, \theta_{n}^{(i)} - \theta_{n+1}^{(i\star)}  \right\rangle   \\ 
   =& 	 \left\langle \left[\nabla f_{n+1}^{(1)}(\theta_{n}^{(1)}),\dots, \nabla f_{n+1}^{(m)}(\theta_{n}^{(m)})\right] ,\, \param_{n} - \param_{n+1}^{\star}  \right\rangle   \\
			\ge& 	\left\langle \nabla f(\param_{n}) ,\, \param_{n} - \param_{n+1}^{\star}  \right\rangle   - m L \lVert \param_{n}-\param_{n+1} \rVert \, \lVert \param_{n}-\param_{n+1}^{\star} \rVert \\
			\ge& 	 \left\langle \nabla f(\param_{n}) ,\, \param_{n} - \param_{n+1}^{\star}  \right\rangle   -m L r_{n+1}^{2}. 
		\end{align}
		Therefore, combining the above inequalities, we conclude 
    \begin{align}\label{eq:f_n_lipschitz_transfer}
			&\sum_{n=0}^{\infty}	\left( \widetilde{V}(\param_{n}, \param_{n+1}^{\star})=\left\langle \nabla f(\param_{n}) ,\, \param_{n} - \param_{n+1}^{\star}  \right\rangle + \sum_{i=1}^m \left(p_{n+1}^{(i)}(\theta_{n}^{(i)}) - p_{n+1}^{(i)}(\theta_{n+1}^{(i\star)}) \right)\right)\\ &\le \left(\frac{m L_{g}}{2} + m L \right) \sum_{n=1}^{\infty} r_{n}^{2} + (F(\param_{0})- F^{*}) + m \sum_{n=1}^{\infty} \Delta_{n} < \infty.
		\end{align} 
	\end{proof}

	\begin{lemma}[Key lemma for asymptotic stationarity for inexact trust-region]\label{lem:first_order_optimality2} 
% \commHL{(Revise for $p$, use the $V$ notation)}
		Assume \ref{assumption:A1}, \ref{assumption:A2}, \ref{assumption:A3}\textbf{(b)} and \ref{assumption:A4}.  Then there exists $c>0$ such that for all $n\ge 0$, 

\begin{align}\label{eq:optimality1_proximal0}
	%\min\{ r_{n}, 1\}\sup_{\param\in \Param,\, \lVert \param-\param_{n} \rVert\le 1}  \left\langle -\nabla f(\param_{n}),\, \param-\param_{n} \right\rangle   \le \left\langle -\nabla f(\param_{n}),\, \param_{n+1}^{\star}-\param_{n} \right\rangle + c r_{n+1}^{2}.
 \min\{ r_{n}, 1\}\sup_{\param\in \Param,\, \lVert \param-\param_{n} \rVert\le 1} \,  \widetilde{V}(\param_{n}, \param)  \le  \widetilde{V}(\param_{n}, \param_{n+1}^{\star}) + c r_{n+1}^{2}.
		\end{align}
  
	\end{lemma}

	\begin{proof}
       	Let $b_{n}:= \min\{ r_{n}, 1\}$ for all $n\ge 1$. Fix $n\ge 0$ and let $\param=[\theta^{(1)},\dots,\theta^{(m)}]\in \Param$ be such that $\lVert \param - \param_{n} \rVert \le b_{n+1}\le r_{n+1}$. %By  convexity of  $\Theta^{(i)}$, $\theta_{n}^{(i)}+a(\theta^{(i)}-\theta_{n}^{(i)})\in \Theta^{(i)}$ for all $a\in [0,1]$.  
		Then we have 
		\begin{align}
		\commHL{	G_{n+1}^{(i)}( \theta^{(i\star)}_{n+1} )    \le G_{n+1}^{(i)}( \theta^{(i)} ). }
		\end{align}
		 Let $\Param_{0}=S^{(1)}\times \dots \times S^{(m)}$ be as in Proposition \ref{prop:boundedness}. From \ref{assumption:A3}, each surrogate $g_{n}^{(i)}$ has $L_g$-Lipschitz continuous gradient on $S^{(i)}$. Also, $\param,\param_{n+1}^{\star}, \param_{n}\in \Param_{0}$ for all $n\ge 1$. Recall that  $\nabla g_{n+1}^{(i)}(\theta_{n}^{(i)})=\nabla f_{n+1}^{(i)}(\theta_{n}^{(i)})$ since $g_{n}^{(i)}\ge f_{n}^{(i)}$ and $ g_{n+1}^{(i)}(\theta_{n}^{(i)})= f_{n+1}^{(i)}(\theta_{n}^{(i)})$. Hence by subtracting $g_{n+1}^{(i)}(\theta_{n}^{(i)})$ from both sides and applying the $L_g$-smoothness of $g_{n}^{(i)}$ on $S^{(i)}$ (see Lemma \ref{lem:surrogate_L_gradient}), 
		\begin{align}
			&\hspace{-1cm}\left\langle \nabla f_{n+1}^{(i)}(\theta_{n}^{(i)}),\,  \theta_{n+1}^{(i\star)}-\theta_{n}^{(i)}  \right\rangle \commHL{+ p_{n+1}^{(i)}( \theta^{(i\star)}_{n+1} )} \\
            \le& \left\langle \nabla f_{n+1}^{(i)}(\theta_{n}^{(i)}),\,  \theta^{(i)}-\theta_{n}^{(i)}\right\rangle   + \frac{L_g}{2}\lVert \theta_{n+1}^{(i\star)} - \theta_{n}^{(i)} \rVert^{2} + \frac{L_g}{2}\lVert \theta^{(i)} - \theta_{n}^{(i)} \rVert^{2} \commHL{+ p_{n+1}^{(i)}( \theta^{(i)})} \\
			\le& \left\langle \nabla f_{n+1}^{(i)}(\theta_{n}^{(i)}),\,  \theta^{(i)}-\theta_{n}^{(i)}\right\rangle   + L_g r_{n+1}^{2}\commHL{+p_{n+1}^{(i)}( \theta^{(i)})}.
		\end{align}
		Adding up these inequalities for $i=1,\dots,m$,  
		\begin{align}
			&\hspace{-1cm}\sum_{i=1}^{m} 	\left\langle \nabla f_{n+1}^{(i)}(\theta_{n}^{(i)}),\,  \theta_{n+1}^{(i\star)}-\theta_{n}^{(i)}  \right\rangle \commHL{+ p_{n+1}^{(i)}(\theta_{n+1}^{(i\star)}) - p_{n+1}^{(i)}(\theta^{(i)})} \\
   \le&   \left\langle \left[ \nabla f_{n+1}^{(1)}(\theta_{n}^{(1)}),\dots,\nabla f_{n+1}^{(m)}(\theta_{n}^{(m)}) \right],\,  \param - \param_{n} \right\rangle + m L_g r_{n+1}^{2} \\
			\overset{(a)}{\le}& \left\langle \nabla f(\param_{n}),\, \param-\param_{n} \right\rangle  + m L \lVert \param_{n+1}-\param_{n} \rVert\cdot  \lVert \param -\param_{n} \rVert + m L_g r_{n+1}^{2} \\
			\overset{(b)}{\le}&  \left\langle \nabla f(\param_{n}),\, \param-\param_{n} \right\rangle  + c r_{n+1}^{2}
		\end{align}
		for $c:=m^{2} L+mL_g$, where (a) uses that  $\nabla_{i} f$ is $L$-Lipschitz in the $i$th block coordinate (Proposition \ref{prop:boundedness}) for each $i=1,\dots,m$ and (b) follows since $\lVert \param_{n+1}-\param_{n} \rVert \le m r_{n+1}$ and $\lVert \param-\param_{n} \rVert \le r_{n+1}$.  The above inequality holds for all $\param\in \Param$ with $\lVert \param - \param_{n} \rVert\le b_{n+1}$. Furthermore, note that 
		\begin{align}
			&\hspace{-1cm}\sum_{i=1}^{m} 	\left\langle \nabla f_{n+1}^{(i)}(\theta_{n}^{(i)}),\,  \theta_{n+1}^{(i\star)}-\theta_{n}^{(i)}  \right\rangle \\
   =& \left\langle \nabla f_{n+1}^{(1)}(\theta_{n}^{(1)}),\dots, \nabla f_{n+1}^{(m)}(\theta_{n}^{(1)})  ,\, \param_{n+1}^{\star} -\param_{n} \right\rangle \\
			\ge& \left\langle \nabla f (\param_{n}) ,\, \param_{n+1}^{\star} -\param_{n} \right\rangle  - m L \lVert \param_{n}-\param_{n+1} \rVert\, \lVert \param_{n+1}^{\star}-\param_{n} \rVert \\
			\ge& \left\langle \nabla f (\param_{n}) ,\, \param_{n+1}^{\star} -\param_{n} \right\rangle  - m L r_{n+1}^{2}.
		\end{align}		
		Thus it follows that 
		\begin{align}
  \sup_{\param\in \Param, \lVert \param-\param_{n} \rVert\le b_{n+1}} \widetilde{V}(\param_{n}, \param ) \le \widetilde{V}(\param_{n}, \param_{n+1}^{\star}) + (c+m L) r_{n+1}^{2}. 
		\end{align}
		Then using the same argument as in Lemma \ref{lem:first_order_optimality}, we can conclude the assertion. 
	\end{proof}

	\begin{prop}[Sufficient condition for stationarity II]\label{prop:stationary_conditions}
		Assume \ref{assumption:A1}, \ref{assumption:A2}, \ref{assumption:A3}\textbf{(b)} and \ref{assumption:A4}. If there exists a subsequence $(\param_{n_{k}})_{k\ge 1}$ that if $\sum_{k=1}^{\infty}  \lVert \param_{n_{k}}- \param_{n_{k}+1}^{\star} \rVert = \infty$, 
		then there is a further subsequence $(s_{k})_{k\ge 1}$ of $(n_{k})_{k\ge 1}$ so that  $\lim_{k\rightarrow \infty} \param_{s_{k}}$ exists and is a stationary-Nash point of $F$ over $\Param$.
	\end{prop}
	
	\begin{proof}
		Denote $\eta_{n}:=\param_{n}-\param^{\star}_{n+1}$. Then under the hypothesis, by Prop. \ref{prop:finite_range_short_points2} 
	\begin{align}\label{eq:stationary_conditions}
		%\sum_{k=1}^{\infty} \lVert \eta_{n_{k}} \rVert \,   \left( \left\langle \nabla f(\param_{n_{k}}),\,  \frac{  \eta_{n_{k}}  }{\lVert \eta_{n_{k}} \rVert}  \right\rangle + \sum_{i=1}^m \left(\frac{p_{n_k+1}^{(i)}(\theta_{n_k}^{(i)}) - p_{n_k+1}^{(i)}(\theta_{n_k+1}^{(i\star)})}{\lVert \eta_{n_{k}} \rVert} \right) \right)  <\infty.
\commHL{ \liminf_{k\rightarrow\infty}\,\,  \sum_{k=1}^{\infty}  \frac{1}{\lVert \eta_{n_{k}} \rVert}\widetilde{V}(\param_{n_{k}}, \param_{n_{k}+1}^{\star}) = 0.}
	\end{align}
	%Hence the former condition in \eqref{eq:stationary_conditions}  implies the latter condition in \eqref{eq:stationary_conditions}. Thus it suffices to show that this latter condition implies the assertion. Assume this condition, and
	Let $(t_{k})_{k\ge 1}$ be a subsequence of $(n_{k})_{k\ge 1}$ for which the  liminf in \eqref{eq:stationary_conditions} is achieved. According to Prop. \ref{prop:boundedness}, by taking a further subsequence, we may assume that $\param_{\infty}=\lim_{k\rightarrow \infty} \param_{t_{k}}$ exists. Then $\param_{\infty}\in \Param$ since $\Param$ is closed by \ref{assumption:A1}.

		Now suppose for contradiction that $\param_{\infty}$ is not a stationary-Nash point of $F$ over $\Param$.
  Then there exists ${\param}'\in\Param$ and $\delta>0$ such that $\widetilde{V}(\param_{\infty}, \param') >\delta>0$. Denote $\param^{t}:=t \param' + (1-t)\param_{\infty}$ for $t\in [0,1]$. Then by convexity of $p$, 
		\begin{align}
			\widetilde{V}(\param_{\infty}, \param^{t})%&=\langle  \nabla f(\param_{\infty}),\, \param^{t}-\param_{\infty} \rangle + p(\param_{\infty}) - p(\param^{t})  \\
            & \ge t	\widetilde{V}(\param_{\infty}, \param')> t\, \delta >0 \quad \textup{for all $t\in [0,1]$}. 
		\end{align}
		Choose $t$ sufficiently small so that $\lVert \param^{t} -\param_{\infty} \rVert<1/2$ and denote $\param^{*}$ for such $\param^{t}$. Note that $\param^{*}\in \Param$ by convexity of $\Param$.

		\commHL{By \ref{assumption:A1}, the function $\widetilde{V}(\cdot,\cdot)$ in \eqref{eq:stationary3} is continuous. Hence we have 
 $\widetilde{V}(\param_{t_{k}}, \param^{*}) \rightarrow \widetilde{V}(\param_{\infty}, \param^{*})>t\delta>0$ as $k\rightarrow\infty$. }
        %{\color{red} The first term may not go to zero since subgradients are not continuous}
		%Noting that $\lVert \param_{n}-\param_{n-1} \rVert = O(r_{n}) = o(1)$, we see that the right-hand side goes to zero as $k\rightarrow \infty$. %Hence for all sufficiently large $k\ge 1$, we have $\left\langle \nabla f(\param_{t_{k}}),\, \param^{*}-\param_{t_{k}} \right\rangle< -\delta/2$. 
		Hence by Lemma \ref{lem:first_order_optimality2},  there is a constant $c>0$ such that for all $n\ge 1$ and $k$ sufficiently large, 
\begin{equation}
		\frac{t\delta}{2} \le \widetilde{V}(\param_{t_{k}},\, \param^{*}) \le	\sup_{\substack{\param\in \Param \\ \lVert \param-\param_{t_{k}} \rVert\le 1}} \widetilde{V}(\param_{t_{k}}, \param) \le \frac{\lVert \eta_{t_{k}} \rVert}{\min \{ r_{t_{k}} ,1 \}}	  \frac{\widetilde{V}(\param_{t_{k}}, \param_{t_{k}+1}^{\star})}{\lVert \eta_{t_{k}} \rVert}   +  \frac{c r_{t_{k}+1}^{2}}{\min \{ r_{t_{k}} ,1 \}}.
\end{equation}
		Since $\lVert \param_{n+1}^{\star} - \param_{n} \rVert\le mr_{n+1}$, by using the hypothesis, the right-hand side above converges to zero as $k\rightarrow\infty$, which is a contradiction. 
	\end{proof}

	The following result is crucial in establishing global convergence to stationary-Nash points (Theorem \ref{thm:complexity} \textbf{(i)}) and it is the only place where we use non-summability of the radii $r_{n}$'s in the proof of Theorem \ref{thm:complexity}.

	\begin{prop}[Local structure of a non-stationary-Nash limit point]\label{prop:non-stationary_nbh}
		Assume \ref{assumption:A1}, \ref{assumption:A2}, \ref{assumption:A3}\textbf{(b)} and \ref{assumption:A4}. Suppose there exists a non-stationary-Nash limit point $\param_{\infty}$ of $\Lambda$. Then there exists $\eps>0$ such that the $\eps$-neighborhood $B_{\eps}(\param_{\infty}):=\{ \param\in \Param\,|\, \lVert \param-\param_{\infty} \rVert<\eps\}$ with the following properties: 
    \vspace{0.1cm}
		\begin{description}
			\item[(a)] $B_{\eps}(\param_{\infty})$ does not contain any stationary-Nash points of $\Lambda$. 
			\vspace{0.1cm}
			\item[(b)] There exists infinitely many $\param_{n}^{\star}$'s outside of $B_{\eps}(\param_{\infty})$.
		\end{description}
	\end{prop}
	
	\begin{proof}
        We first show that there exists $\eps>0$ such that $B_{\eps}(\param_{\infty})$ satisfies  \textbf{(a)}. Suppose for contradiction that there exists no such $\eps>0$. Then we have a sequence $(\param_{\infty;k})_{k\ge 1}$ of stationary-Nash points of $\Lambda$ that converges to $\param_{\infty}$. Fix $\param\in \Param$. \commHL{Note that by continuity of $\tilde{V}(\cdot,\cdot)$ and $\widetilde{V}(\param_{\infty;k},\param)\le 0$ since $\param_{\infty;k}$ is a stationary-Nash point of $F_{\infty;k}$ over $\Param$, we have $\widetilde{V}(\param_\infty,\param)\le 0$.} Since $\param\in \Param$ was arbitrary, this implies that $\param_{\infty}$ is a stationary-Nash point of $F$ over $\Param$, a contradiction.

		%Let $A$ denote the event that there exists a non-stationary-Nash limit point $\param_{\infty}$ of $\Lambda$. 
        Next, we show that we can assume by choosing $\eps>0$ smaller, if necessary, that $B_{\eps}(\param_{\infty})$ that does not contain any short points of $\Lambda$.   Suppose not. Then for all $\eps>0$ sufficiently small, $B_{\eps}(\param_{\infty})$ contains some short point of $\Lambda$.  Then there exists a subsequence of short points of $\Lambda$ converging to $\param_{\infty}$, but then $\param_{\infty}$ has to be stationary-Nash by Proposition \ref{prop:long_points_stationary}, which contradicts the hypothesis.

		Now we assume that  $B_{\eps}(\param_{\infty})$ has no short point of $\Lambda$ and also satisfies  \textbf{(a)}.  We will show that $B_{\eps/2}(\param_{\infty})$ satisfies \textbf{(b)}. Then $B_{\eps/2}(\param_{\infty})$ satisfies  \textbf{(a)}-\textbf{(b)}, as desired. Suppose for contradiction there are only finitely many $\param_{n}^{\star}$'s outside of $B_{\eps/2}(\param_{\infty})$. Then there exists an integer $M\ge 1$ such that $\param_{n}^{\star}\in B_{\eps/2}(\param_{\infty})$ for all $n\ge M$. Then each $\param_{n}^{\star}$ for $n\ge M$ is a long point of $\Lambda$. By definition, it follows that $\lVert \param_{n}^{\star}-\param_{n-1}\lVert \ge r_{n}$ for all $n\ge M$. Then since $\sum_{n=1}^{\infty} r_{n}=\infty$, by Proposition \ref{prop:stationary_conditions} there exists a subsequence $(n_{k})_{k\ge 1}$ such that $\param_{\infty}':=\lim_{k\rightarrow \infty} \param_{n_{k}} $ exists and is stationary-Nash.  But since $\param'_{\infty}\in B_{\eps}(\param)$, this  contradicts \textbf{(a)}. This shows the assertion. 
	\end{proof}

	We are now ready to give proof of the asymptotic stationarity result stated in  Theorem \ref{thm:complexity} \textbf{(iii)}. 
	
	\vspace{0.1cm}
	\begin{proof}[\text{Proof of Theorem \ref{thm:complexity} \textup{(iii)} under \ref{assumption:A3}\textbf{\textup{(a)}}}]
		
		%\commHL{Below we need to consider $p_{n}$}
        %{\color{red} Maybe we show asymptotic convergence only for smooth objectives} 
		Suppose for contradiction that there exists a non-stationary-Nash limit point $\param_{\infty}\in \Param$ of $\Lambda$. We may choose $\eps>0$ such that $B_{\eps}(\param_{\infty})$ satisfies the conditions \textbf{(a)}-\textbf{(b)} of Proposition \ref{prop:non-stationary_nbh}. Choose $M\ge 1$ large enough so that $m r_{n}<\eps/3$ whenever $n\ge M$. We call an integer interval $I:=[\ell,\ell']$ a \textit{crossing} if $\param_{\ell}\in B_{\eps/3}(\param_{\infty})$, $\param_{\ell'}\in B_{2\eps/3}(\param_{\infty})$, and no proper subset of $I$ satisfies both of these conditions. By definition, two distinct crossings have empty intersections. Fix a crossing $I=[\ell,\ell']$, it follows that by triangle inequality,
		\begin{align}\label{eq:upcrossing_ineq}
			\sum_{n=\ell}^{\ell'-1} \lVert \param_{n+1}-\param_{n} \rVert \ge \lVert \param_{\ell'}-\param_{\ell} \rVert \ge \eps/3. 
		\end{align}
		Note that since $\param_{\infty}$ is a limit point of $\Lambda$, $\param_{n}$ visits $B_{\eps/3}(\param_{\infty})$ infinitely often. Moreover, by condition \textbf{(b)} of Proposition \ref{prop:non-stationary_nbh}, $\param_{n}^{\star}$ also exits $B_{\eps}(\param_{\infty})$ infinitely often. But since $\lVert \param_{n}-\param_{n}^{\star} \rVert \le mr_{n}<\eps/3$, it follows that $\param_{n}$ exits $B_{2\eps/3}(\param_{\infty})$ infinitely often. It follows that there are infinitely many crossings.  Let $t_{k}$ denote the $k^{\textup{th}}$ smallest index that appears in some crossing. Then $t_{k}\rightarrow \infty$ as $k\rightarrow \infty$, and by \eqref{eq:upcrossing_ineq},
		\begin{align}
			\sum_{k=1}^{\infty} \lVert \param_{t_{k}+1} - \param_{t_{k}} \rVert \ge (\text{$\#$ of crossings}) \, \frac{\eps}{3} = \infty.
		\end{align}
       By triangle inequality, 
        \begin{align}
             \lVert \param_{t_{k}+1}^{\star} - \param_{t_{k}} \rVert & \ge  \lVert \param_{t_{k}+1} - \param_{t_{k}} \rVert  -  \lVert \param_{t_{k}+1}^{\star} - \param_{t_{k}+1} \rVert \ge   \lVert \param_{t_{k}+1} - \param_{t_{k}} \rVert  - \Delta_{t_{k}}.
        \end{align}
        Since $\Delta_{n}$'s are summable by Assumption \ref{assumption:A4}, it follows that $\sum_{k=1}^{\infty} \lVert \param_{t_{k}+1}^{\star} - \param_{t_{k}} \rVert = \infty$. 
        
        Then by Proposition \ref{prop:stationary_conditions}, there exists a  subsequence $(s_{k})_{k\ge 1}$ of $(t_{k})_{k\ge 1}$ such that $\param_{\infty}':=\lim_{k\rightarrow \infty} \param_{s_{k}}$ exists and is a \commHL{stationary-Nash} point of $F$ over $\Param$. However, since $\param_{t_{k}}\in B_{2\eps/3}(\param_{\infty})$ for all $k\ge 1$ by the choice of $t_{k}$, we have $\param_{\infty}'\in B_{\eps}(\param_{\infty})$. This contradicts the condition \textbf{(a)} of Proposition \ref{prop:non-stationary_nbh} since $B_{\eps}(\param_{\infty})$  cannot contain any stationary-Nash point in $\Lambda$. %This shows that there is no non-stationary-Nash limit point $\param_{\infty}\in \Param$ of $\Lambda$. %See Figure \ref{fig:pf_DR_illustration} for an illustration of the proof. 
	\end{proof}

	\vspace{0.2cm}

	\section{Applications of the main result}  
 \label{sec:app}

\subsection{Nonnegative Matrix Factorization (NMF) }
\label{sec:NMF}

Given a $p\times N$ data matrix $X\in \R^{p\times N}$ and an integer parameter $r\ge 1$, consider the following  \textit{constrained matrix factorization} problem
	\begin{align}\label{eq:NMF}
		\argmin_{W\in \Theta^{(1)}\subseteq \R^{p\times r}, \, H\in \Theta^{(2)}\subseteq \R^{r\times N}} \lVert X - WH \lVert_{F}^{2} + \lambda_1 \lVert H \rVert_{1}, 
	\end{align}
	where the two factors $W$ and $H$ are called \textit{dictionary} and \textit{code} matrices of $X$, respectively, and $\lambda_1 \ge 0$ is a $\ell_{1}$-regularizer for the code matrix $H$. Depending on the application contexts, we may impose some \textit{constraints} $\Theta^{(1)}$ and $\Theta^{(2)}$ on the dictionary and code matrices, respectively, such as nonnegativity or some other convex constraints. An interpretation of this approximate factorization is that the $r$ columns of $W$ give an approximate basis for spanning the $N$ columns of $X$, where the columns of $H$ give suitable linear coefficients for each approximation \cite{olshausen1997sparse}. % engan1999frame, lewicki2000learning, elad2006image, lee2005acquiring . In \textit{dictionary learning} problems, one seeks a sparse representation of the columns of $X$ with respect to and an over-complete dictionary $W$, for which one can take $r>p$  and $\lambda>0$. 

The Nonnegative Matrix Factorization (NMF) \cite{lee2001algorithms} is a special instance of the constrain matrix factorization problem above where $\Theta^{(1)}=\R_{\ge 0}^{p\times r}$ and $\Theta^{(2)}=\R_{\ge 0}^{r\times N}$ and $\lambda_1 =0$ and $X\in \R_{\ge 0}^{p\times N}$ is given. %{\color{red} Dimension doesn't match, need to add transpose.} 
NMF has numerous applications in text analysis, image reconstruction, medical imaging, bioinformatics, and many other scientific fields  \citep{  berry2007algorithms, taslaman2012framework}. %berry2005email, ren2018non, boutchko2015clustering,sitek2002correction, chen2011phoenix
The use of nonnegativity constraint in NMF is crucial in obtaining a ``parts-based" representation of the input signal \cite{lee2001algorithms}.

In \cite{lee2001algorithms}, the following multiplicative update (MU) is studied for NMF:
\begin{align}\label{eq:MU_NMF}
\hspace{-2.3cm}\textbf{MU}\hspace{1cm}
\begin{cases}
H_{n+1} \leftarrow H_{n} \odot (W_{n}^{T}  X) \oslash (W_{n}^{T} W_{n}  H_{n} ) \\
W_{n+1}^{T} \leftarrow W_{n}^T \odot (H_{n}  X^{T}) \oslash (H_{n} H_{n}^{T}  W_{n}^{T} ),
\end{cases}
\end{align}
where $\odot$ and $\oslash$ denote entry-wise multiplication and division. Given a nonnegative initialization $(W_{0},H_{0})$, the iterate \eqref{eq:MU_NMF} generates a sequence of non-negative factor matrices $(W_{n}, H_{n})_{\ge 0}$. In \cite{lee2001algorithms}, it was shown that the objective value of the NMF monotonically decreases under the iterate \eqref{eq:MU_NMF}, but it has not been proven that the convergence is toward the set of stationary points of the NMF problem. There are some works that propose modified versions of \eqref{eq:MU_NMF} and show asymptotic convergence to stationary points (e.g., \cite{lin2007convergence}). Furthermore, to the best of our knowledge, there has not been any result on the rate of convergence of any variants of MU \eqref{eq:MU_NMF}. %nakano2010convergence,takahashi2014global, zhao2017unified

Here we propose \textit{MU with Regularization} (MUR), which falls under our BMM (Alg. \eqref{eq:BMM_DR_highlevel}) and satisfies the hypothesis of our main result, Theorem \ref{thm:complexity}. Fix regularization parameters $\delta\ge 0$ (thresholding parameter), $\rho\ge 0$ (proximal regularization parameter). Consider the following variant of MU: 
\begin{align}\label{eq:MUPR_NMF}
\textbf{MUR}\hspace{0.5cm}
\begin{cases}
\tilde{H}_{n}\leftarrow H_{n}\lor \delta, \\ 
H_{n+1} \leftarrow (\tilde{H}_{n} \odot (W_{n}^{T}  X + \rho \tilde{H}_{n}) \oslash ((W_{n}^{T} W_{n} + \rho \mathbf{I})  \tilde{H}_{n} ), \\
\tilde{W}_{n}\leftarrow \tilde{W}_{n}\lor \delta, \\ 
W_{n+1}^T \leftarrow W_{n}^T \odot (H_{n}  X^{T} + \rho \tilde{W}_{n}^T) \oslash ((H_{n} H_{n}^{T}+  \rho \mathbf{I})  \tilde{W}_{n}^T ),
\end{cases}
\end{align}
where $(\, \cdot \lor \delta)$ is the operation of taking maximum with $\delta$ entry-wise and  $\mathbf{I}$ denotes the $r\times r$ identity matrix. Note that by setting $\delta=\rho=0$, \eqref{eq:MUPR_NMF} reduces to the standard MU in \eqref{eq:MU_NMF}. The following corollary shows that the MUR \eqref{eq:MUPR_NMF} algorithm for NMF, as long as  $\rho,\delta>0$, retains the asymptotic convergence and the rate of convergence stated in Theorem \ref{thm:complexity}.

\begin{corollary}[Convergence of MUR for NMF] \label{cor:NMF}
		Fix a matrix $X\in \R_{\ge 0}^{p \times N}$. Let $(\param_{n})_{n\ge 0}$, $\param_{n}:=(W_{n}, H_{n})$ be generated by \eqref{eq:MUPR_NMF} with arbitrary initialization $\param_{0}$ in a compact set $\Param\subseteq\R_{\ge 0}^{p\times r}\times \R_{\ge 0}^{r\times N}$. Denote $f(W,H):=\frac{1}{2}\lVert X- WH \rVert_{F}^{2}$. Suppose the thresholding and proximal regularization parameters $\delta,\rho$ are strictly positive. Then Theorem \ref{thm:complexity}\textbf{(i)}-\textbf{(iii)}  holds for $(\param_{n})_{n\ge 0}$. 
  \iffalse Then the following hold: 
		\begin{description}[itemsep=0.1cm]
			\item[(i)] (Rate of convergence) There exists an explicit constant $M>0$ such that for  $n\ge 1$,
			\begin{align}\label{eq:thm_convergence_bd}
				\min_{1\le k \le n}  	 \,\, \left[ \sup_{\param\in \Param,\, \lVert \param-\param_{k} \rVert\le 1} \left\langle -\nabla f(\param_{k}),\, \param - \param_{k} \right\rangle \right]  &\le \frac{ M  \log n}{ 2\sqrt{n}  }.
			\end{align}

			\item[(ii)] (Worst-case iteration complexity)  The worst-case iteration complexity $N_{\eps}$ for the MUR algorithm \eqref{eq:MUPR_NMF}  satisfies $N_{\eps} = O(\eps^{-2} (\log \eps^{-1})^{2} )$.
			
			\item[(iii)] (Asymptotic stationarity)   $(\param_{n})_{n\ge 1}$ converges to the set of stationary points of the NMF problem \eqref{eq:NMF}.
			
		\end{description}
\fi
\end{corollary}

\commHL{While we restrict the factors $(W,H)$ for NMF to live inside a compact subset $\Param$ Corollary \ref{cor:NMF} for a technical reason, this does not lose any generality in terms of the objective values if $\Param$ consists of all nonnegative factor matrices with Frobenius norm bounded by any constant at least $2r\lVert X\rVert_{F}^{1/2}$. Indeed, if $(W, H)$ is any pair of nonnegative factors such that $\lVert X-WH\rVert_{F}\le \lVert X\rVert_{F}$.
then one can rescale the columns of $W$ and the corresponding rows of $H$ suitably so that their norm is at most $2\lVert X\rVert_{F}^{1/2}$ and the rescaled factors  have the same objective value $\lVert X-WH \rVert_{F}^{2}$ (see \cite[Sec. 6]{lyu2022online}).}

To justify Corollary \ref{cor:NMF}, we first explain why MUR \eqref{eq:MUPR_NMF} can be viewed as a special instance of the BMM algorithm (Alg. \eqref{eq:BMM_DR_highlevel}).  Consider the following convex sub-problem of \eqref{eq:NMF}: 
$
\min_{H\ge 0} \left( f_{n}(H):=\frac{1}{2}\lVert X- W_{n}H \rVert_{F}^{2} \right).
%\left( f_{n}\left( H \right):=\frac{1}{2}\left\|X^{j}-W_n H^{j}\right\|_F^2+\frac{1}{2} \sum_{k \neq j}\left\|X^k-W_n H^{k}_n\right\|_F^2 \right)
$
For any matrix $A$ and positive integer $j$ at most the number of columns in $A$, denote $A^{j}$ to be its $j$th column. Then $f_{n}(H)=\sum_{j=1}^{N} f_{n}^{j}(H^{j})$, where  $f_{n}^{j}(\mathbf{h}):= \frac{1}{2}\lVert X^{j} - W_{n} \mathbf{h} \rVert_{F}^{2}$. 

Fix parameters $\rho,\delta\ge 0$. Define a function $g_{n}:\R^{r\times N}\rightarrow \R$ by $g_{n}(H):=\sum_{j=1}^{N} g_{n}^{j}(H^{j})$, where
\begin{align}\label{eq:surrogate_MU}
 g_n^{j}(\h)&:= 
f^j_{n}\left( \h_{n}\right)+\left(\h-\h_n\right)^T \nabla_{\h} f^j_{n}\left( \h_{n}\right)
+\frac{1}{2}\left(\h-\h_n\right)^T \left( \mathbf{D} +\rho \mathbf{I} \right) \left(\h-\h_n\right), 
\end{align}
where $\h_{n}:=H_{n}^{j}$ and $\mathbf{D} =\mathbf{D}^{j}_{n}(\delta,\rho)$ is the $r\times r$ diagonal matrix given by 
\begin{align}\label{eq:D_def}
%\mathbf{D}:=\operatorname{diag}\left(\frac{\left[\left(W_n\right)^T W_n \h_n\right]_1}{\left[\h_n\right]_1 + \delta}, \cdots, \frac{\left[\left(W_n\right)^T W_n \h_n\right]_K}{\left[\h_n\right]_K + \delta}\right) + \rho \mathbf{I}, 
\mathbf{D}_{a,b} := \mathbf{1}(a=b) \frac{[W_{n}^{T} W_{n} (\h_{n} \lor \delta ) ]_{a} }{ [\h_{n} \lor \delta]_{a}  }.
\end{align}
\commHL{Note that taking $\delta=0$ in \eqref{eq:D_def}, $g_{n}^{j}$ defined above becomes the majorizer of $f_{n}^{j}$ used in \cite{lee2001algorithms}.}
The thresholding parameter $\delta\ge 0$ prevents the denominator above from vanishing when $H^{j}_{n}$ has zero on some coordinates. 
\commHL{Since the thresholding is applied to $\h_{n}$ in both the numerator and the denominator in \eqref{eq:D_def} and since we do not threshold $W_{n}^{T}W_{n}$, it turns out that $g_{n}^{j}$ is still a majorizer of $f_{n}^{j}$ for $\delta>0$. We claim that the following properties hold:}
\begin{enumerate}[itemsep=0.1cm, label={[\texttt{\alph*}]}]
    
    \item  $\nabla f$ is $L$-Lipschitz continuous for some $L>0$ when restricted onto a compact set.
    \label{NMF_claim_0}

    \item  $g_{n}\ge f_{n}$ and $g_{n}(H_{n})=f_{n}(H_{n})$; $g_{n}$ is a quadratic function. 
    \label{NMF_claim_1}
    
    \item $H_{n+1}$ in \eqref{eq:MUPR_NMF} is an exact minimizer of $g_{n}$ over $\R_{\ge 0}^{p\times r}$. 
     \label{NMF_claim_2}
    
    \item  $g_n$ is $\rho$-strongly convex.
     \label{NMF_claim_3}

     \item If $\delta>0$, then there exists a constant $L_{h}>0$ such that $h_{n}:=g_{n}-f_{n}$ has $L_{h}$-Lipschitz continuous gradient for some constant $L_{h}>0$. 
      \label{NMF_claim_4}

\end{enumerate}
A similar construction of surrogate function and claim will hold for $W_{n+1}$ by symmetry. Points \ref{NMF_claim_1}-\ref{NMF_claim_2} verify that \eqref{eq:MUPR_NMF} is a particular instance of BMM (Alg. \eqref{eq:BMM_DR_highlevel}). Given this, we can easily verify that the hypothesis of Theorem \ref{thm:complexity} holds for the NMF problem \eqref{eq:NMF} and the algorithm MUR \eqref{eq:MUPR_NMF}. Indeed, \commHL{Assumptions \ref{assumption:A1} and \ref{assumption:A2} hold trivially, except the sub-level set compactness for NMF holds if and only if $\Param$ is compact, which holds by the hypothesis.} The smoothness property of the majorizer in \ref{assumption:A3} holds by point \ref{NMF_claim_4}. \ref{assumption:A3}\textbf{(b)} holds for $\rho>0$ due to point \ref{NMF_claim_3} (This is why we should require $\rho>0$ in Corollary \ref{cor:NMF}). Lastly, \ref{assumption:A4} holds by point \ref{NMF_claim_2}. This is enough to deduce Corollary \ref{cor:NMF} from Theorem \ref{thm:complexity}. 

\begin{proof}[Proof of points \textup{\ref{NMF_claim_0}-\ref{NMF_claim_4}}]

We first justify  \ref{NMF_claim_0}.  Note
    \begin{align}
       \nabla f(W,H) - \nabla f(W',H') =  \left[ W^{T}W(H-H'), (W-W')HH^{T} \right].
    \end{align}
Thus, if we restrict $(W,H)$ on a compact subset of the parameter space, then $\nabla f$ is $L$-Lipschitz continuous for some constant (depending on the compact subset) $L>0$. This verifies \ref{NMF_claim_0}.

\commHL{Next, we show point \ref{NMF_claim_1}.} Clearly $g_{n}$ is a quadratic function.  Next, fix $j\in \{1,\dots, N\}$. Note that we can expand $f_{n}^{j}(\mathbf{h})$ at $\h:=H_{n}^{j}$ as a quadratic function. Subtracting this from $g_{n}^{j}$,
\begin{align}\label{eq:pf_surrogate_NMF_diff}
    g_{n}^{j}(\h) - f_{n}^{j}(\h) = \frac{1}{2}\left(\h-\h_n\right)^T \left( \mathbf{D}^j_{n}  - W_{n}^{T}W_{n} + \rho \mathbf{I} \right)\left(\h-\h_n\right). 
\end{align}
We claim that the matrix $\mathbf{D}^{j}_{n}  - W_{n}^{T}W_{n}$ is positive semidefinite. To justify the claim, it is enough to show that the following rescaled matrix is positive semidefinite: 
\begin{align}
    Q := \textup{diag}( \tilde{H}_{n}^{j} ) \, (\mathbf{D}_{n}^{j} - W_{n}^{T}W_{n}) \,  \textup{diag}(\tilde{H}_{n}^{j}).
\end{align}
Indeed, this can be shown from the following observation (similar computation was used in the proof of \cite[Lem. 2]{lee2001algorithms} \commHL{with $\tilde{H}_{n}^{j}$ replaced with $H_{n}^{j}$}):  For all $\x\in \R^{r}$, 
\begin{align}
    \x^{T} Q   \x 
    %&= \left( \sum_{a}  [W_{n}^{T} W_{n} \tilde{H}_{n}^{j}]_{a} [\tilde{H}_{n}^{j}]_{a}\x_{a}^{2} \right)   -  \sum_{a,  b } \x_{a} [\tilde{H}_{n}^{j}]_{a} [W_{n}^{T}W_{n}]_{a,b}  [\tilde{H}_{n}^{j}]_{b} \x_{b}  \\
    %&= \sum_{a,b}  [W_{n}^{T} W_{n}]_{a,b}  [\tilde{H}_{n}^{j}]_{a} [\tilde{H}_{n}^{j}]_{b}  \x_{a}^{2}  - \sum_{a,  b } \x_{a} [\tilde{H}_{n}^{j}]_{a} [W_{n}^{T}W_{n}]_{a,b} [\tilde{H}_{n}^{j}]_{b} \x_{b}  \\
    %&= \sum_{a,b}  [W_{n}^{T} W_{n}]_{a,b}  [\tilde{H}_{n}^{j}]_{a}  [\tilde{H}_{n}^{j}]_{b}  \left( \frac{\x_{a}^{2}+\x_{b}^{2}}{2} - \x_{a}\x_{b} \right)  \\
    &= \sum_{a,b}  [W_{n}^{T} W_{n}]_{a,b}  [\tilde{H}_{n}^{j}]_{a}  [\tilde{H}_{n}^{j}]_{b}  \left( \x_{a}-\x_{b}  \right)^{2}/2 \ge 0.
\end{align}
Now, since $Q$ is positive semidefinite, the identity \eqref{eq:pf_surrogate_NMF_diff} implies points \ref{NMF_claim_1}. Point  \ref{NMF_claim_3} follows by noting that $\mathbf{D}_n^j$ is a nonnegative diagonal matrix.

Next, we justify point \ref{NMF_claim_2}. First note that minimizing $g_{n}(H)$ over $H\in \R_{\ge 0}^{r\times N}$ can be done separately over the columns of $H$. Note that 
\begin{align}\label{eq:NMF_surrogate_grad}
    %\nabla  g_{n}^{j} (H^{j}) = % -\nabla_{H^{j}} f_{n}^{j}\left( H_{n}^{j} \right) + \left( H^{j} \right)^{T} \mathbf{D}^{j}_{n} = 0\\
    %- W_{n}^{T}X^{j} +  \mathbf{D}^{j}_{n} H^{j}= \mathbf{0}, 
     \nabla  g_{n}^{j} (H^{j}) 
    %&= - W_{n}^{T}X^{j} + W_{n}^{T}W_n \tilde{H}^j_n +  (\mathbf{D}^{j}_{n}+\rho\mathbf{I}) (H^{j}-\tilde{H}^j_n) \\
    &= - W_{n}^{T}X^{j}  +  (\mathbf{D}^{j}_{n}+\rho\mathbf{I}) H^{j}-\rho\tilde{H}^j_n.
\end{align}
The global minimizer of  $g_{n}^{j}$ is given by the solution  to $\nabla  g_{n}^{j}  =\mathbf{0}$, which is 
\begin{align}
    H^{j}_{n+1}:= (\mathbf{D}_{n}^{j} + \rho \mathbf{I})^{-1} \left( W_{n}^{T}X^{j}+ \rho \tilde{H}_{n}^{j} \right)
\end{align}
Assuming $H_{n}$ and $W_{n}$ are entry-wise nonnegative, recursively, $H_{n+1}^{j}$ is also entry-wise nonnegative. Hence $H_{n+1}^{j}$ above is the global minimizer of $g_{n}^{j}$ on $\R_{\ge 0}^{r}$. By collecting the columns $j=1,\dots, N$, it follows that $H_{n+1}$ in \eqref{eq:MUPR_NMF} is the global minimizer of $g_{n}$ over $\R_{\ge 0}^{r\times N}$. 

Lastly, we justify point \ref{NMF_claim_4}. From \eqref{eq:pf_surrogate_NMF_diff} it suffices to show that 
$\mathbf{D}^j_{n}  - W_{n}^{T}W_{n}$ has bounded maximum eigenvalue. 
%the majorization gap $h_{n}=g_{n}-f_{n}$ is a quadractic function with Hessian 
\commHL{A key difficulty in analyzing the original MU algorithm in \cite{lee2001algorithms} is that, even if one assumes that the factors $(W,H)$ live in a compact set, the definition of $\mathbf{D}_{n}^{j}$ in \eqref{eq:D_def} involves denominator the coordinates of $\h_{n}$, which could vanish while the numerator $[W_{n}^{T}W_{n}\h_{n}]_{a}$ remain bounded. Our up-thresholding of $\h_{n}$ by $\delta>0$ exactly prevents this issue while maintaining majorization.} 
%At this point, we know that \eqref{eq:MUPR_NMF} is an instance of BMM (Alg. \eqref{eq:BMM_DR_highlevel}). We have also verified the hypothesis of  Proposition \ref{prop:boundedness}. This proposition shows that the iterates $(W_{n},H_{n})_{n\ge 1}$ (and hence $\tilde{H}_{n}$s) are contained in a bounded set. 
\commHL{Indeed, by the hypothesis,} the iterates $(W_{n},H_{n})_{n\ge 1}$ (and hence $\tilde{H}_{n}$s) are contained in a bounded set. Take $L$ to be 
\begin{align}
    L:=\sup_{n\ge 1} \max_{1\le j \le N} \left( \lVert \mathbf{D}_{n}^{j} \rVert_{\infty} + \rho\right) \le \sup_{n\ge 1} \,\delta^{-1} \left( \lVert W_{n}^{T}W_{n} \tilde{H}_{n} \rVert_{\infty} +\rho \right)<\infty. 
\end{align}
Then $L$ uniformly upper bounds the largest eigenvalue of $ \mathbf{D}^j_{n}  - W_{n}^{T}W_{n} + \rho \mathbf{I}$ in \eqref{eq:pf_surrogate_NMF_diff}. This shows \ref{NMF_claim_4}. 
\end{proof}
\iffalse
To see how MUR improves MU, one first rewrites the algorithms in gradient descent form as follows,
\begin{align}
    &\textbf{MU}: \qquad \qquad H_{n+1} \leftarrow H_n - H_{n} \oslash (W_{n}^{T} W_{n}  H_{n} ) \odot \nabla_H f(W_n,H_n) \label{eq:MU_gd},\\
    &\textbf{MUR}: \qquad \quad\; H_{n+1} \leftarrow \tilde{H}_n - \tilde{H}_{n} \oslash \left[ (W_{n}^{T} W_{n} +\rho \mathbf{I})  \tilde{H}_{n} \right] \odot \nabla_H f(W_n,\tilde{H}_n) \label{eq:MUR_gd}.
\end{align}
Then the numerator and denominator of the step size $H_{n} \oslash (W_{n}^{T} W_{n}  H_{n} )$ in \eqref{eq:MU_gd} could be zero, which potentially prevents MU to converge to stationary points \cite{lin2007projected,gonzalez2005accelerating}. MUR \eqref{eq:MUR_gd} addresses this issue with the help of regularization parameters $\rho$ and $\delta$. 
\fi
\iffalse
Another modification of MU was proposed in Lin \citep{lin2007convergence}. %, which solves the issue of MU with different regularization tricks and methodologies. 
It was shown that this modified MU algorithm converges to the set of stationary points of the NMF problem, but no bound on the iteration complexity was given.   
\fi

In Section \ref{sec:MUR_numerics}, we provide numerical experiments showing the advantage of MUR when dealing with sparse data.

\subsection{Applications to Constrained  Tensor Factorization}
	\label{sec:NCPD}
	
    As matrix factorization is for unimodal data, \textit{nonnegative tensor factorization} (NTF) provides a powerful and versatile tool that can extract useful latent information out of multi-model data tensors. As a result, tensor factorization methods have witnessed increasing popularity and adoption in modern data science \cite{shashua2005non}.  %zhou2016accelerating, sun2017provable, zafeiriou2009algorithms,   rambhatla2020provable

	\iffalse
	\begin{wrapfigure}[12]{r}{0.6\textwidth}
		\vspace{-0.6cm}
		\includegraphics[width=0.6\textwidth]{Figures/NMF_scheme.pdf}
		\vspace{-0.8cm}
		\caption{Illustration of NMF (top) and NTF (bottom). $m$-mode tensors are observed $n$ times. One seeks $m$ loading matrices as well as a code matrix that gives an approximate decomposition of all data.}
		\label{fig:NTF_illustration}
	\end{wrapfigure}
     \fi
    Suppose a data tensor $\X \in  \R^{I_{1}\times \cdots \times I_{m}}$ is given and fix an integer $R\ge 1$. In the  CANDECOMP/PARAFAC (CP) decomposition of $\X$ \cite{kolda2009tensor}, we would like to find \textit{loading matrices} $U^{(i)} \in \R^{I_{i}\times R}$ for $i=1,\dots, m$ such that the sum of the outer products of their respective columns approximate $\X$: $\X \approx  \sum_{k=1}^{R} \bigotimes_{i=1}^{m} U^{(i)}[:,k]$, where $U^{(i)}[:,k]$ denotes the $k^{\textup{th}}$ column of the $I_{i}\times R$ loading matrix matrix $U^{(i)}$ and $\bigotimes$ denotes the outer product. As an optimization problem, the above CP decomposition model can be formulated as the following the \textit{constrained tensor factorization} problem: 
	\begin{align}\label{eq:NTF} 
		\argmin_{U^{(i)}\in \Theta^{(i)}; i=1,\dots,m}   F(U^{(1)},\dots, U^{(m)}) :=  \left\lVert \X - \sum_{k=1}^{R} \bigotimes_{i=1}^{m} U^{(i)}[:,k] \right\rVert_{F}^{2} \hspace{-0.7em}+ \sum_{i=1}^{m}\lambda_{i} \lVert U^{(i)} \rVert_{1},
	\end{align}
	where $\Theta^{(i)}\subseteq \R^{I_{i}\times R}$ denotes a closed and convex constraint set and $\lambda_{i}\ge 0$ is a $\ell_{1}$-regularizer for the $i^{\textup{th}}$ loading matrix $U^{(i)}$ for $i=1,\dots, m$.  In particular, by taking  $\lambda_{i}=0$ and  $\Theta^{(i)}$ to be the set of nonnegative $I_{i}\times R$ matrices for $i=1,\dots, m$,  \eqref{eq:NTF} reduces to the \textit{nonnegative CP decomposition} (NCPD) \cite{shashua2005non}. %, zafeiriou2009algorithms
    
    \iffalse
    Also, It is easy to see that \eqref{eq:NTF} is equivalent to 
	\begin{align}\label{eq:NTF_dict}
		\argmin_{U^{(1)}\in \Theta^{(1)},\dots, U^{(m)}\in \Theta^{(i)}} \left\lVert \X - \Out(U^{(1)}, \dots, U^{(m-1)}) \times_{m} (U^{(m)})^{T} \right\rVert_{F}^{2} + \sum_{i=1}^{m}\lambda_{i} \lVert U^{(i)} \rVert_{1}, 
	\end{align}
	which is a \textit{CP-dictionary-learning} problem introduced in \cite{lyu2022online}. Here 
	$\times _{m}$ denotes the mode-$m$ product (see \cite{kolda2009tensor}) the outer product of loading matrices $U^{(1)},\dots, U^{(m)}$ is defined as 
	\begin{align}\label{eq:def_out}
		\Out(U^{(1)},\dots,U^{(m)}) := \left[\bigotimes_{k=1}^{m} U^{(k)}[:,1],\,  \dots \,, \bigotimes_{k=1}^{m} U^{(k)}[:,R]  \right] \in \R^{I_{1}\times \dots \times I_{m} \times R}. 
	\end{align}
	Namely, we can think of the $m$-mode tensor $\mathbf{X}$ as $I_{m}$ observations of $(m-1)$-mode tensors, and the $R$ rank-1 tensors in $\Out(U^{(1)},\dots,U^{(m)})$ serve as dictionary atoms, whereas the transpose of the last loading matrix $U^{(m)}$ can be regarded as the code matrix (see Figure \ref{fig:NTF_illustration}). In particular, assuming $m=2$, \eqref{eq:def_out} becomes the constrained matrix factorization problem in \eqref{eq:NMF}.  
    \fi
	
	The constrained tensor factorization problem \eqref{eq:NTF} falls under the framework of BCD in \eqref{eq:BCD_factor_update}, since the objective function $F$ in \eqref{eq:NTF} is convex in each loading matrix $U^{(i)}$ for $i=1,\dots, m$. Indeed, BCD is a popular approach for both NMF and NTF problems \cite{kim2014algorithms}. Namely, when we apply BCD \eqref{eq:BCD_factor_update} for \eqref{eq:NTF}, each block update amounts to solving a quadratic problem under convex constraint. BCD for \eqref{eq:NTF}  is known as the form of \textit{alternating least squares} (ALS), \commHL{which can be implemented by using a few steps of projected gradient descent for solving each subproblem.} For NMF, ALS (or vanilla BCD in \eqref{eq:BCD_factor_update}) is known to converge to stationary points \cite{grippo2000convergence}. However, for NTF with $m\ge 3$ modes, global convergence to stationary points of ALS is not guaranteed in general \cite{kolda2009tensor} and requires some additional regularity conditions \cite{bertsekas1997nonlinear, grippo2000convergence}. Using our general framework of BMM-DR, we propose the following iterative algorithms for constrained tensor factorization: For each $n\ge 1$ and $i=1,\dots,m$, 
    \begin{align}\label{eq:BMM_DR_CTF}
    &\hspace{1cm}\textbf{BMM-DR for CTF}
    \begin{cases}
         \mathbf{A} \leftarrow \Out(U^{(1)}_{n-1}, \dots, U^{(i-1)}_{n-1}, U^{(i+1)}_{n-1},\dots, U^{(m-1)}_{n-1} ) \in \R^{(I_{1}\times \cdots \times   I_{i-1}\times  I_{i+1} \times \cdots\times  I_{m}) \times R } \\
          B \leftarrow \textup{unfold}(\mathbf{A}, m) \in \R^{(I_{1}\cdots I_{i-1} I_{i+1} \cdots I_{m}) \times R } \\
          g_{n}^{(i)}(U) \leftarrow \textup{Majorizing surrogate of $f_{n}^{(i)}(U):=\lVert \textup{unfold}(\mathbf{X}, i) -  B U^{T}     \rVert_{F}^{2}$} \\
      U_{n}^{(i)}\in \argmin_{U \in \Theta^{(i)},\, \lVert U-U_{n-1}^{(i)} \rVert_{F}\le r_{n}}  \, G_{n}^{(i)}(U):= g_{n}^{(i)}(U) + \lambda_{i} \lVert  U  \rVert_{1},
    \end{cases}
	\end{align}
    where $\textup{unfold}(\cdot, i)$ denotes the mode-$i$ tensor unfolding (see \cite{kolda2009tensor}) and $r_{n}\in [0,\infty]$ for $n\ge 1$ denotes radius of trust-region. 
    The iterate \eqref{eq:BMM_DR_CTF} specializes in various tensor factorization algorithms depending on the choice of the surrogate $g_{n}^{(i)}$. Namely, first, there are four ALS-type algorithms: 
    \begin{enumerate}[itemsep=0.1cm, itemindent=-0.5cm, label={(\textbf{\alph*})}]
        \item  \label{eq:ALS_CTF} (ALS) \quad $g_{n}^{(i)}=f_{n}^{(i)}$, $r_{n}\equiv \infty$.
        \item  \label{eq:ALS_DR_CTF} (ALS with DR)  \quad $g_{n}^{(i)}=f_{n}^{(i)}$, $\sum_{n=1}^{\infty} r_{n}=\infty, \sum_{n=1}^{\infty} r_{n}^{2}<\infty$. 
        \item \label{eq:ALS_PR_CTF}(ALS with PR) \quad $g_{n}^{(i)}(U)=f_{n}^{(i)}(U) +\frac{\lambda}{2}\lVert U-U_{n-1}^{(i)} \rVert_{F}^{2}$, $\lambda>0$, and $r_{n}\equiv \infty$.  
        \item  \label{eq:ALS_DR_PR_CTF}(ALS with DR+PR) \quad $g_{n}^{(i)}(U)=f_{n}^{(i)}(U)        +\frac{\lambda}{2}\lVert U-U_{n-1}^{(i)} \rVert_{F}^{2}$, $\lambda>0$, $\sum_{n=1}^{\infty} r_{n}=\infty$, and $\sum_{n=1}^{\infty} r_{n}^{2}<\infty$. 
    \end{enumerate}
    Next, specialize \eqref{eq:NTF} into NCPD, where $\Theta^{(i)}=\R^{I_{i}\times R}_{\ge 0}$  and $\lambda_{i}=0$ for $i=1,\dots m$. There are two MU-type algorithms for NCPD, which can be derived similarly as in the NMF case (see Sec. \ref{sec:NMF}): 
    \begin{enumerate}[itemsep=0.1cm, itemindent=-0.5cm, resume, label={(\textbf{\alph*})}]
        \item  \label{eq:MU_NCPD} (MU)\quad  %$g_{n}^{(i)}=g_{n}$ in \eqref{eq:surrogate_MU} with $\delta=\rho=0$, $H=U^{T}$, $X=\textup{unfold}(\mathbf{X}, i)$, and $W_{n}=B$. This yields the following factor update: 
       $\displaystyle  U_{n+1}^{(i)} \leftarrow U_{n}^{(i)} \odot (B^{T} \textup{unfold}(\mathbf{X}, i)) \oslash (B^{T} B  U_{n}^{(i)} )$. 

         \item   \label{eq:MUR_NCPD} (MUR) \quad  %$g_{n}^{(i)}=g_{n}$ in \eqref{eq:surrogate_MU} with $\delta,\rho>0$, $H=U^{T}$, $X=\textup{unfold}(\mathbf{X}, i)$, and $W_{n}=B$. This yields the following factor update: 
        $\displaystyle
          \begin{cases}
            \tilde{U} \leftarrow U_{n}^{(i)} \lor \delta \\
           U_{n+1}^{(i)} \leftarrow \tilde{U} \odot \left( (B^{T} \textup{unfold}(\mathbf{X}, i)) + \rho \tilde{U} \right) \oslash ((B^{T}B + \rho \mathbf{I})  \tilde{U} ).
        \end{cases}
        $
    \end{enumerate}
    
    For many instances of BMM-DR for CTF listed above, we can apply our general result (Theorem \ref{thm:complexity}) to deduce their convergence and complexity. %Unfortunately, Theorem \ref{thm:complexity} does not cover the case when the objective is non-differentiable, so we may need to assume zero $L_{1}$-regularization in \eqref{eq:NTF}. 
    In order to guarantee asymptotic convergence to stationary points and iteration complexity as stated in Theorem \ref{thm:complexity}, we need to assume that either diminishing radius or strongly convex surrogate is used. This rules out the vanilla ALS \ref{eq:ALS_CTF} for general CTF and MU \ref{eq:MU_NCPD} for NCPD. The following corollary holds for all the other options listed above. 
	
	\begin{corollary}[Convergence of BMM-DR for CTF]\label{cor:NTF}
    %Suppose no sparsity regularizaiton is assumed ($\lambda_{1}=\dots=\lambda_{m}=0$) in \eqref{eq:NTF}.
	Let $\param_{n}=[U^{(1)}_{n},\dots,U^{(m)}_{n}]$, $n\ge 0$ be generated by \eqref{eq:BMM_DR_CTF} with one of the options \ref{eq:ALS_DR_CTF}, \ref{eq:ALS_PR_CTF}, \ref{eq:ALS_DR_PR_CTF} (for general convex constraints on factor matrices) or \ref{eq:MUR_NCPD} (for NCPD).  Then \textbf{\textup{(i)}}-\textbf{\textup{(iii)}} of Theorem \ref{thm:complexity} holds for $(\param_{n})_{n\ge 0}$. 
	\end{corollary}
	
	\iffalse
	\begin{remark}
		In recent joint work with Strohmeier and Needell \cite{lyu2022online}, the first author has developed an online version of constrained tensor factorization and convergence analysis under mild assumptions. The technique of using diminishing radius was first introduced in the reference to encourage the stability of iterates in an online setting. 
	\end{remark}
	\fi

\subsection{Block Projected Gradient Descent}

In the introduction, we discussed that the block projected gradient descent (BPGD) \eqref{eq:BProxLinear} is a special instance of BMM with smooth objectives where prox-linear surrogates are exactly minimized over convex constraint sets. 
 Therefore, our general result (Theorem \ref{thm:complexity}) implies that the BPGD algorithm \eqref{eq:BProxLinear} converges asymptotically to the stationary points (not only Nash equilibrium) and also has iteration complexity of \revision{$\widetilde{O}((1+\rho+\rho^{-1})\eps^{-2})$}, under the hypothesis of Theorem \ref{thm:complexity}. See Corollary \ref{cor:BPGD}.

	\begin{corollary}[Complexity of BPGD]\label{cor:BPGD}
        Let $(\param_{n})_{n\ge 0}$ be generated by the block projected gradient descent updates \eqref{eq:BProxLinear} with stepsize $1/\rho$. 
        Suppose \ref{assumption:A1} and \ref{assumption:A2} hold, where the convex constraint sets $\Theta^{(i)}$ are not necessarily bounded. If $\rho>L$, then \textbf{\textup{(i)}}-\textbf{\textup{(iii)}} of Theorem \ref{thm:complexity} holds for $(\param_{n})_{n\ge 0}$.  In particular, the iteration complexity for smooth nonconvex objectives with convex constraints is $\tilde{O}((1+\rho+\rho^{-1})\eps^{-2})$. 
	\end{corollary}
{\color{black}
Note that though we only state the convergence results for smooth objectives in Corollary \ref{cor:BPGD}, our general convergence and complexity results hold for nonsmooth nonconvex objectives with prox-linear updates as in \eqref{eq:BProxLinear} (when $p\neq 0$).

%\commHL{We already have a survey in the related works section. Due to space limitation, I am commenting it out.}

%To the best of our knowledge, Corollary \ref{cor:BPGD} is the first complexity result of BPGD for general nonsmooth (nonconvex) objectives with cyclic updates. 
%In \cite{xu2013block}, the block PGD for nonconvex and nonsmooth objectives is studied. The authors established complexity results assuming that the objective function satisfy the KL-inequality. Our analysis do not make such assumption.  In \cite{beck2013convergence}, the authors studied block PGD for convex and smooth objectives. The convergence of BPGD in terms of the function values for convex objectives with complexity $\tilde{O}(\eps^{-1})$ is established. Since the objectives studied in these works either satisfy KL-inequality or convexity, their analysis can not be directly transferred to our settings. Recently, iteration complexity of $\tilde{O}(\eps^{-2})$ of block gradient descent on compact Riemannian manifolds has been obtained \cite{peng2023block}. 
    %, gutman2023coordinate
    %These results concern constraints given by compact Riemannian manifold without boundary, which applies neither for the unconstrained Euclidean space (for being unbounded) nor compactly constrained Euclidean space (for having boundary). 
}

	\vspace{0.3cm}
	\section{Experimental Validation}
 \label{sec:exp}
\commHL{
In this section, we compare the performance of BMM-DR with other classical algorithms on different problems. The performance of the algorithms is evaluated by the \textit{relative reconstruction error} defined as $\textup{Error} = \|\mathbf{X} - \hat{\mathbf{X}}\|/\|\mathbf{X}\|$ where $\mathbf{X}$ is the given data matrix (tensor) and $\hat{\mathbf{X}}$ is the reconstructed data matrix (tensor). For algorithms using diminishing radius, we take $r_{n}=c'n^{-\beta}/\log n$ for $n\ge 1$, where $c'>0$ is a constant.
}

\subsection{BMM-DR for NMF}

We compare the performance of MU, BCD, BMM \ref{eq:ALS_PR_CTF}, BCD-DR, and BMM-DR \ref{eq:ALS_DR_PR_CTF} on the NMF task. The advantage of implementing DR becomes significant when matrices are ill-conditioned in NMF. In the experiments, the synthetic data matrix $\mathbf{X}\in \R^{100\times 50}_{\ge 0}$ is generated by the product of $W\in \R^{100\times 7}_{\ge 0}$ and $H\in \R^{7\times 50}_{\ge 0}$ where $W$ is set to have a large condition number of order $10^7$. This makes the block-marginal objective functions have very small strong convexity parameters, if not zero. 
Each algorithm is run 10 times with random initialization, and the averaged reconstruction error is computed. 
\iffalse
\begin{figure}[htb]
    \centering % <-- added
\begin{subfigure}{0.328\textwidth}
  \includegraphics[width=\linewidth]{Figures/NMF_ill1.pdf}
  %\caption{image1}
  \label{fig:11}
\end{subfigure}\hfil % <-- added
\begin{subfigure}{0.332\textwidth}
  \includegraphics[width=\linewidth]{Figures/NMF_ill3.pdf}
  %\caption{image2}
  \label{fig:21}
\end{subfigure}\hfil % <-- added
\begin{subfigure}{0.34\textwidth}
  \includegraphics[width=\linewidth]{Figures/NMF_ill2.pdf}
  %\caption{image3}
  \label{fig:31}
\end{subfigure}
\vspace{-1cm}
\caption{ Comparison of BMM-DR with BMM on NMF. $\beta=0.1$ is the diminishing radius parameter used for all algorithms with DR and $\rho$ is the proximal regularization parameter. The average relative reconstruction error with standard deviation is shown by the lines and shaded regions of respective colors.}
\vspace{-0.7cm}
\label{fig:NMF_DR}
\end{figure}
\fi
\begin{figure}[htb]
    \centering % <-- added
\begin{subfigure}{0.338\textwidth}
  \includegraphics[width=\linewidth]{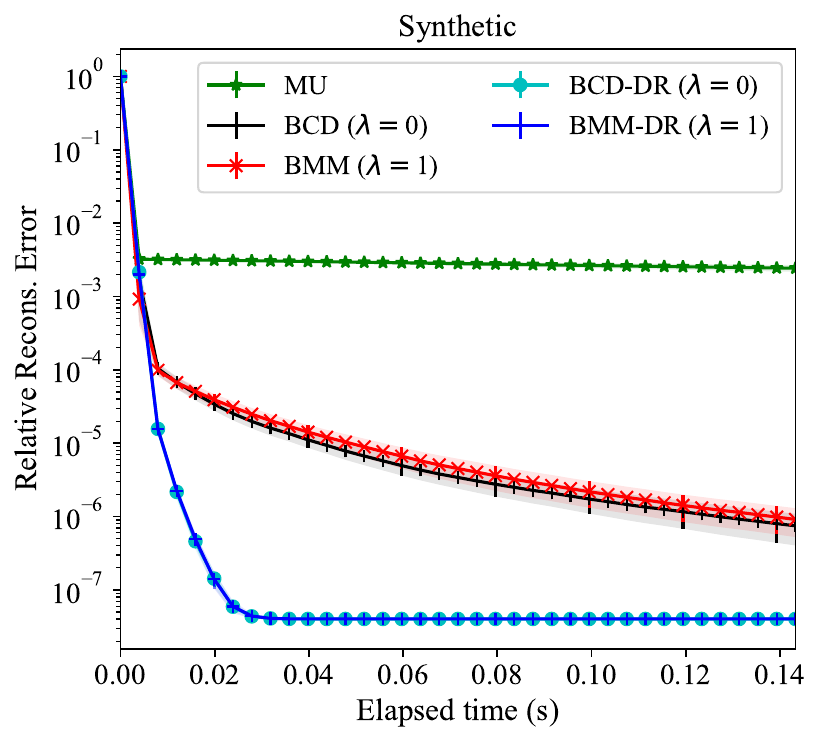}
  %\caption{image1}
  %\label{fig:11}
\end{subfigure}\hfil % <-- added
\begin{subfigure}{0.33\textwidth}
  \includegraphics[width=\linewidth]{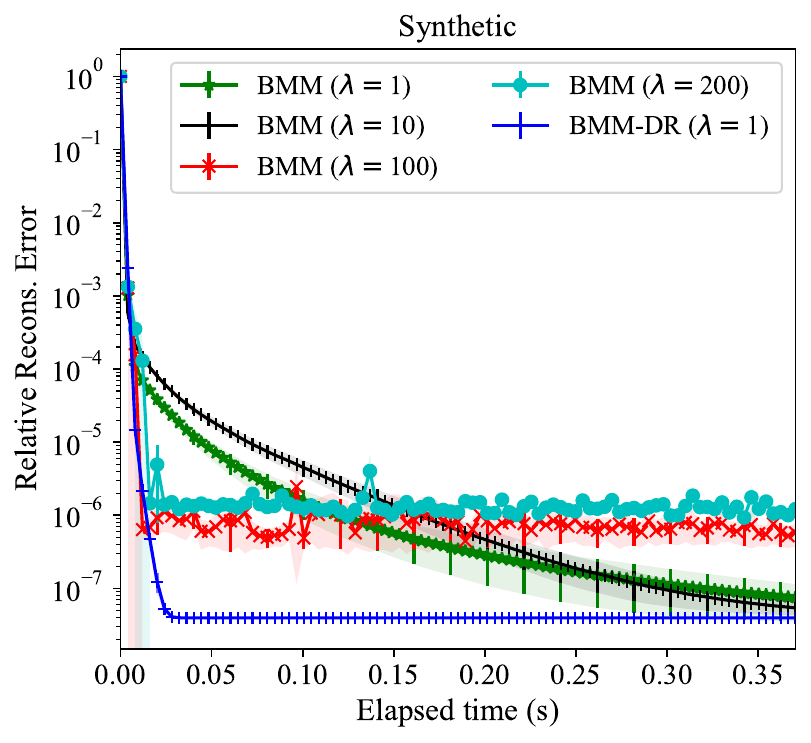}
  %\caption{image3}
  %\label{fig:31}
\end{subfigure}\hfil % <-- added
\begin{subfigure}{0.33\textwidth}
  \includegraphics[width=\linewidth]{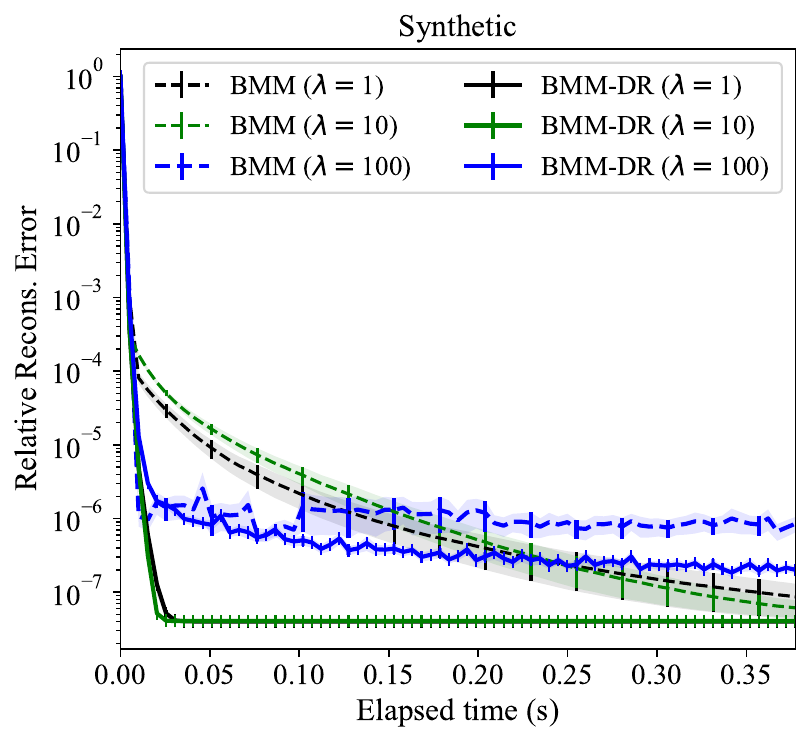}
  %\caption{image2}
  %\label{fig:21}
\end{subfigure}
\vspace{-0.7cm}
\caption{ Comparison of BMM-DR with BMM on NMF. $\beta=0.5$ is the diminishing radius parameter used for all algorithms with DR and $\lambda$ is the proximal regularization parameter. The average relative reconstruction error with standard deviation is shown by the lines and shaded regions of respective colors.}
\vspace{-0.7cm}
\label{fig:NMF_DR}
\end{figure}

The first two plots in Figure \ref{fig:NMF_DR} show the practicality of BMM-DR with square-summable radii. 
Namely, in the left panel, BCD-DR and BMM-DR converge significantly faster to a more accurate solution than other standard algorithms without trust-region (MU, BCD, and BMM). In the middle panel, we compare a single instance of BMM-DR with various instances of BMM depending on the proximal regularization parameter $\lambda$. 
%The experiments in Figure \ref{fig:NMF_DR} confirms these theoretically anticipated behaviors. 
While excessively large proximal regularization ($\lambda=100,200$) seems to compromise the performance of BMM, all four instances of BMM are outperformed by BMM-DR with $\lambda=1$. 

Lastly, we discuss the experiments in the rightmost panel, where we make a direct comparison between BMM (dashed lines) and BMM-DR (solid lines) with the same surrogates is presented. Notice that both the surrogate's strong convexity and smoothness parameters, $\rho$ and $L_{g}$, increase additively by increasing $\lambda$. Hence from our complexity bound in Theorem \ref{thm:complexity}, it is expected to see worse performance using steep surrogates corresponding to excessively large $\lambda$ (e.g., $\lambda=100$) for both BMM and BMM-DR, but for flat surrogates with small $\lambda$ (e.g., $\lambda=1,10$), only the performance of BMM should be hindered, since the complexity bound involves involve the inverse of the strong-convexity parameter. The experiments in the right panel indeed confirm this theoretically expected behavior.

 \subsection{Nonnegative CP-decomposition}

 In this section, we compare the performance of our proposed BMM-DR algorithm \ref{eq:ALS_DR_CTF}-\ref{eq:ALS_DR_PR_CTF} and \ref{eq:MUR_NCPD} for the task of NCPD \eqref{eq:NTF} (with no $L_{1}$-regularization) against two most popular approaches in practice: 1) ALS \ref{eq:ALS_CTF}; and 2) Multiplicative Update (MU)\ref{eq:MU_NCPD}. We consider a synthetic tensor data $\X_{\text{synth}}$ and a real-world tensor data $\X_{\textup{Cifar10}}$.

For each data tensor $\mathbf{X}$ of shape $I_{1}\times I_{2} \times I_{3}$, we used the aforementioned algorithms to learn three loading matrices $U_{i}\in \R^{I_{i}\times R}$ for $i=1,2,3$. We set the number of columns $R=2$ for synthetic data and $R=10$ for the Cifar10 dataset. Each algorithm is run 10 times with independent random initial data in each case, and the plot shows the average relative reconstruction error %(the square root of $f_{\textup{CTF}}$ in \eqref{eq:NTF} divided by $\lVert \X \rVert_{F}$) 
with the standard deviation in shades.

 \begin{figure*}[h]
			\centering
			
			\begin{subfigure}[b]{0.24\textwidth}
				\centering
				\includegraphics[width=\textwidth]{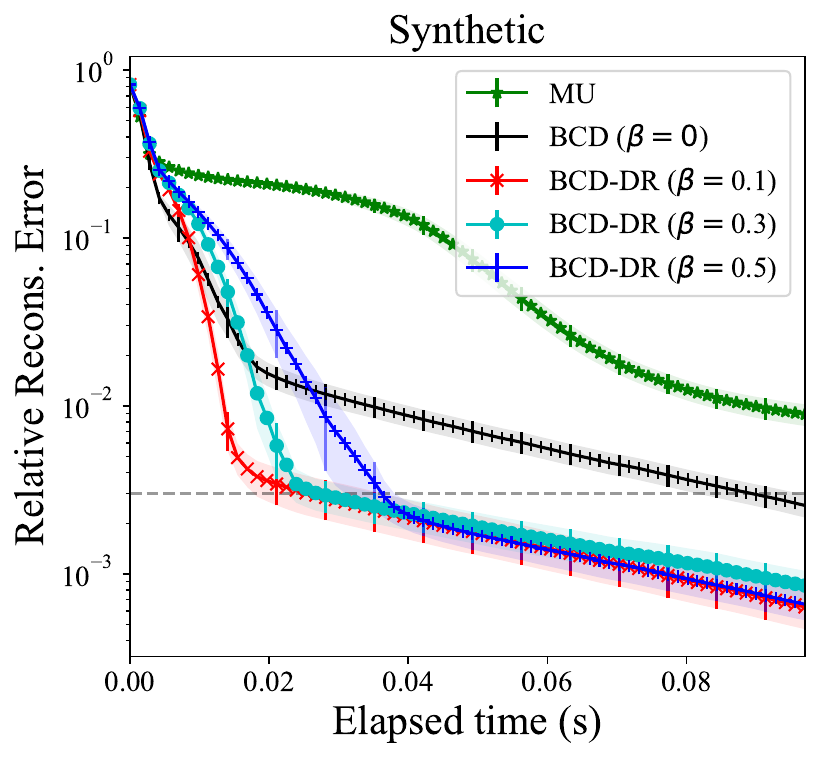}
				\label{fig:subfig1}
			\end{subfigure}
			\hfill
			\begin{subfigure}[b]{0.24\textwidth}
				\centering
				\includegraphics[width=\textwidth]{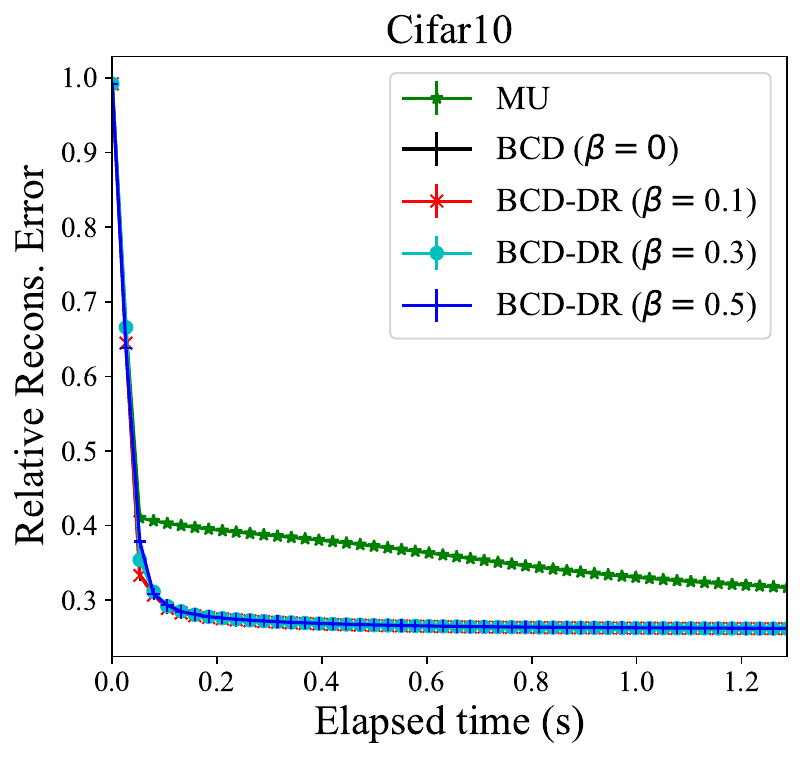}
				\label{fig:subfig2}
			\end{subfigure}
			\hfill
                \begin{subfigure}[b]{0.24\textwidth}
				\centering
				\includegraphics[width=\textwidth]{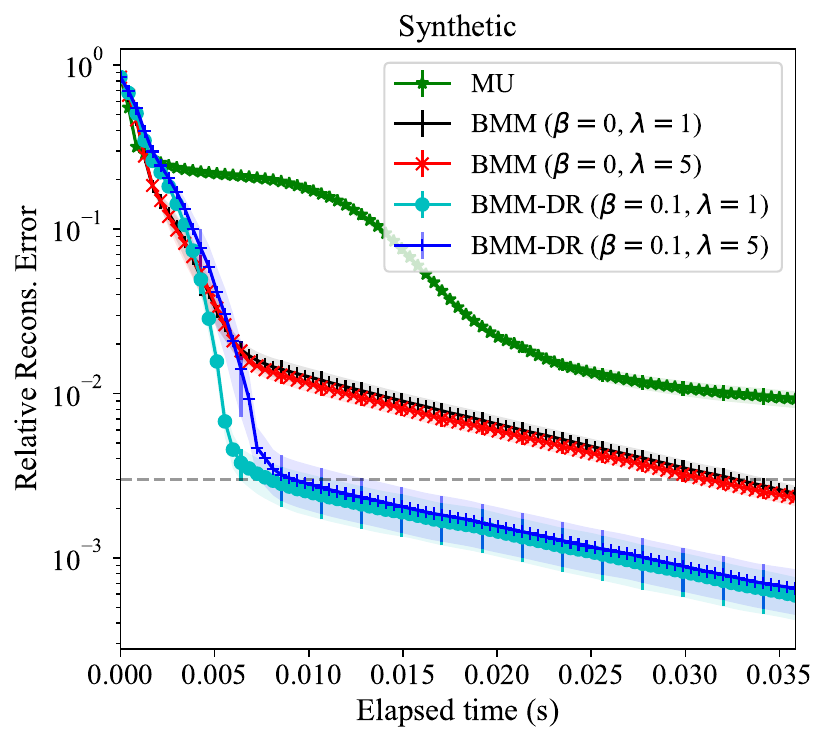}
				\label{fig:subfig2}
			\end{subfigure}
			\hfill
                \begin{subfigure}[b]{0.24\textwidth}
				\centering
				\includegraphics[width=\textwidth]{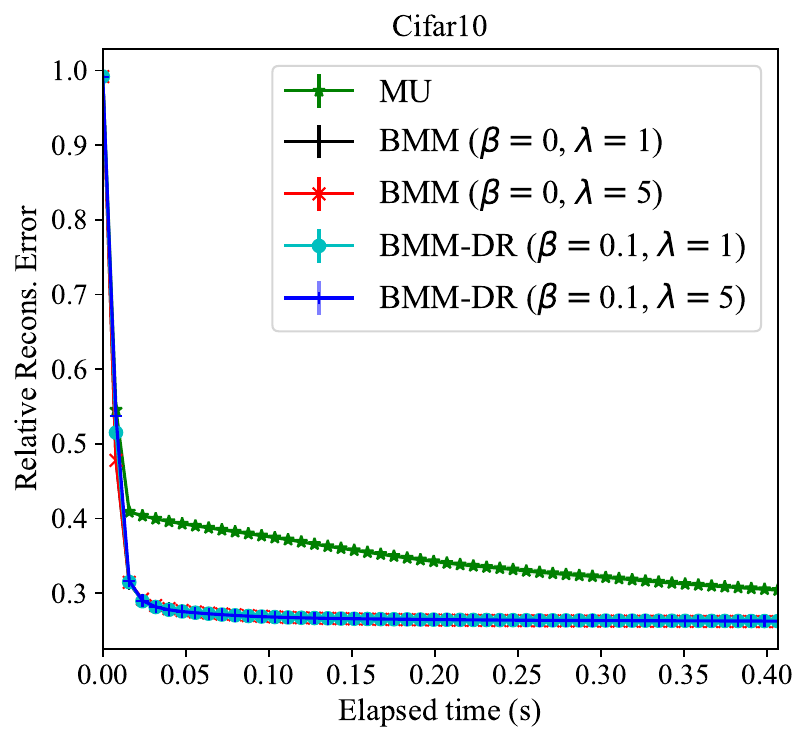}
				\label{fig:subfig2}
			\end{subfigure}
			\hfill

   \vspace{-0.6cm}
			\caption{Comparison of the performance of BMM-DR (Algorithm \eqref{eq:BMM_DR_highlevel}) and MUR against BCD and MU for the nonnegative CP-decomposition (NCPD) problem. BCD (equivalent to ALS) is implmented as \eqref{eq:BMM_DR_CTF}\ref{eq:ALS_CTF} with $r_{n}=\infty$ for $n\ge 1$. BCD-DR is implemented as \eqref{eq:BMM_DR_CTF}\ref{eq:ALS_DR_CTF} with $c'=\lVert \mathbf{X} \rVert_{F}/(1.5\times 10^{5})$ for synthetic data and $c'=\lVert \mathbf{X} \rVert_{F}/(3\times 10^{5})$ for Cifar10 data  ($\mathbf{X}$ denoting the data tensor). BMM \ref{eq:ALS_PR_CTF} is implemented with a proximal regularizer with parameter $\lambda$. BMM-DR \ref{eq:ALS_DR_PR_CTF} is implemented on top of BMM with diminishing radius parameter $\beta$ and the same $c'$ as BCD-DR. The average relative reconstruction error with standard deviation is shown by the solid lines and shaded regions of respective colors.
                }
                \vspace{-0.3cm}
			\label{fig:benchmark}
		\end{figure*}

For synthetic data, as shown in Figure \ref{fig:benchmark}, BCD-DR with proper diminishing radius parameters $\beta$ and $c'$ is significantly faster than MU and also the standard vanilla BCD in terms of elapsed time. %The choice of $c'$ is important to achieve the best performance of BCD-DR. 
Here we take $c'=\lVert \mathbf{X} \rVert_{F}/(1.5\times 10^{5})$ where $\mathbf{X}$ denote the synthetic data tensor, and $1.5\times 10^{5}$ is the number of elements in the tensor. %Similarly, a proper value of $\beta$ is also crucial for fast convergence. Here, as shown in Figure \ref{fig:benchmark}, 
BCD-DR with $\beta=0.1$ attains its best performance. In the third plot, a direct comparison between BMM and BMM-DR with the same surrogates is shown. One can observe that while the effect of different proximal parameters $\rho$ is negligible, applying diminishing radius improves the performance of BMM with the same surrogates.

\hspace{-0.3em}For the Cifar 10 data set, the same experiments are conducted with $c'=\lVert \mathbf{X} \rVert_{F}/(3\times 10^{5})$. All BCD-DR and BMM-DR outperform MU in terms of elapsed time and demonstrate a comparable convergence rate to the vanilla BCD. Diminishing radius (DR) does not accelerate the convergence as in the synthetic data case. %due to the difference in the observed tensor data. In the synthetic data case, the tensor data is generated by loading matrices, so one is able to recover the loading matrices with a small relative reconstruction error. 
In fact, the acceleration from DR becomes significant when the relative reconstruction error is of order $10^{-2}$. However, decomposing real-world tensors to loading matrices with such a small relative reconstruction error may not be possible. Hence, the acceleration from DR is not observed in the Cifar 10 data set case.

\subsection{Comparison between MU and MUR for NMF}
\label{sec:MUR_numerics}

In this section, we compare the performance of our MUR \eqref{eq:MUPR_NMF} for the task of NMF against the original MU \eqref{eq:MU_NMF}. We consider synthetic dense data $\mathbf{X}_{\textup{synth}}$, synthetic sparse data $\mathbf{X}_{\textup{synth-sp}}$ with $20\%$ nonzero elements and real-world data $\mathbf{X}_{\textup{MNIST}}$. %(see details in supplementary material)

%\vspace{-0.3cm}
\begin{figure}[htb]
    \centering % <-- added
\begin{subfigure}{0.335\textwidth}
  \includegraphics[width=\linewidth]{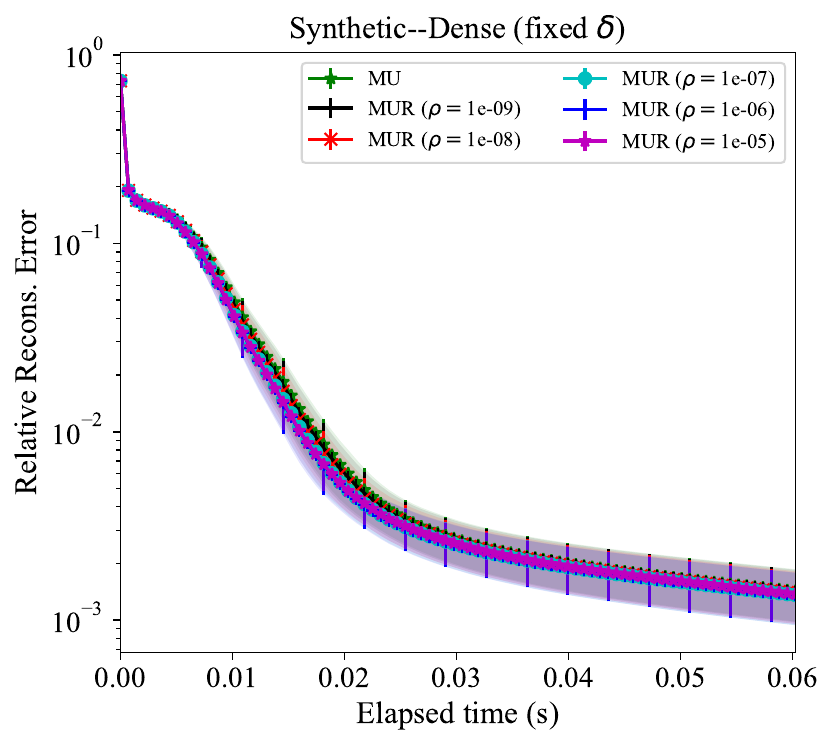}
  %\caption{image1}
  \label{fig:1}
\end{subfigure}\hfil % <-- added
\begin{subfigure}{0.333\textwidth}
  \includegraphics[width=\linewidth]{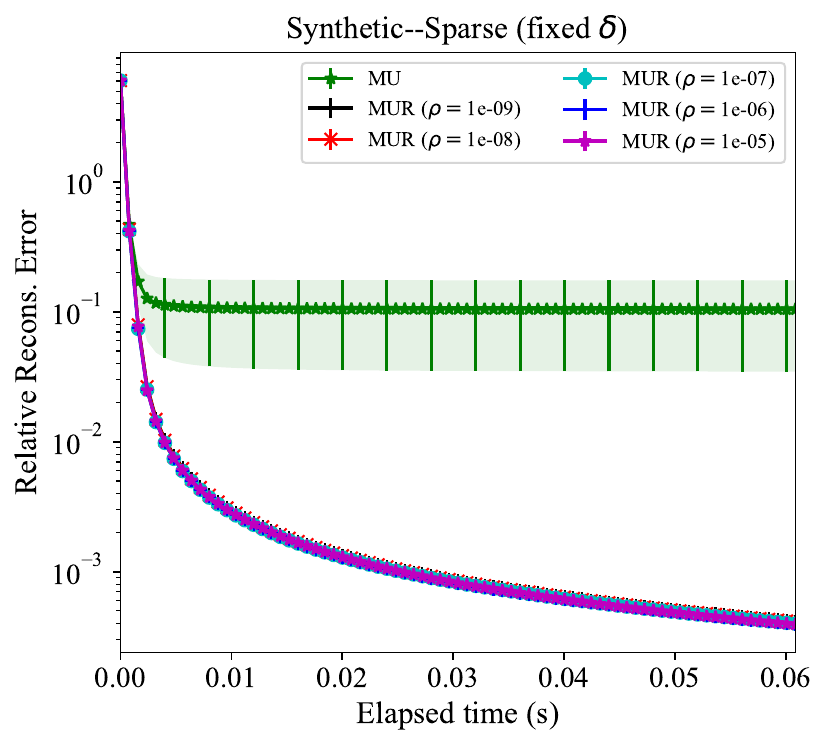}
  %\caption{image2}
  \label{fig:2}
\end{subfigure}\hfil % <-- added
\begin{subfigure}{0.33\textwidth}
  \includegraphics[width=\linewidth]{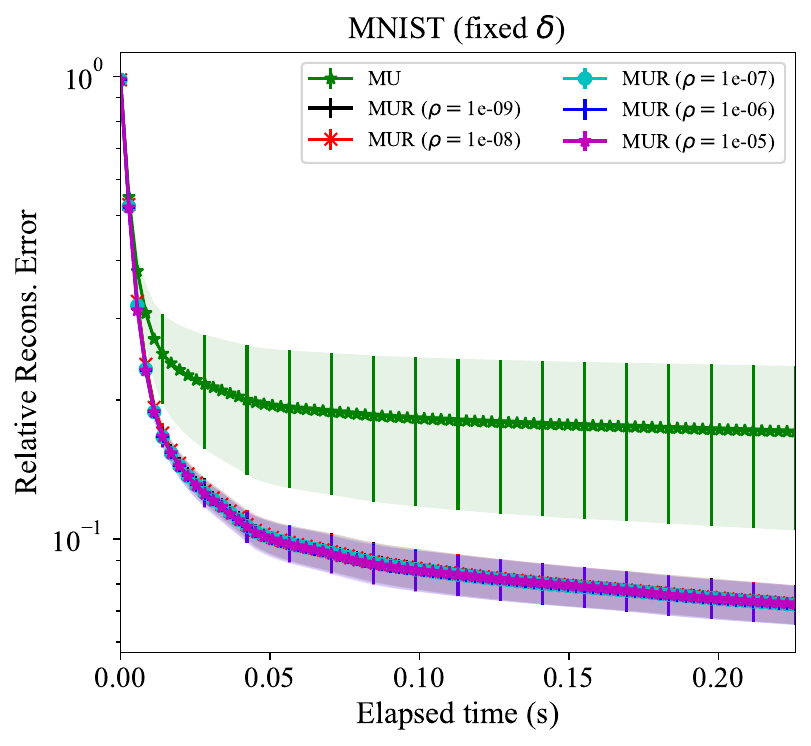}
  %\caption{image3}
  \label{fig:3}
\end{subfigure}

\vspace{-1.4em}
\begin{subfigure}{0.33\textwidth}
  \includegraphics[width=\linewidth]{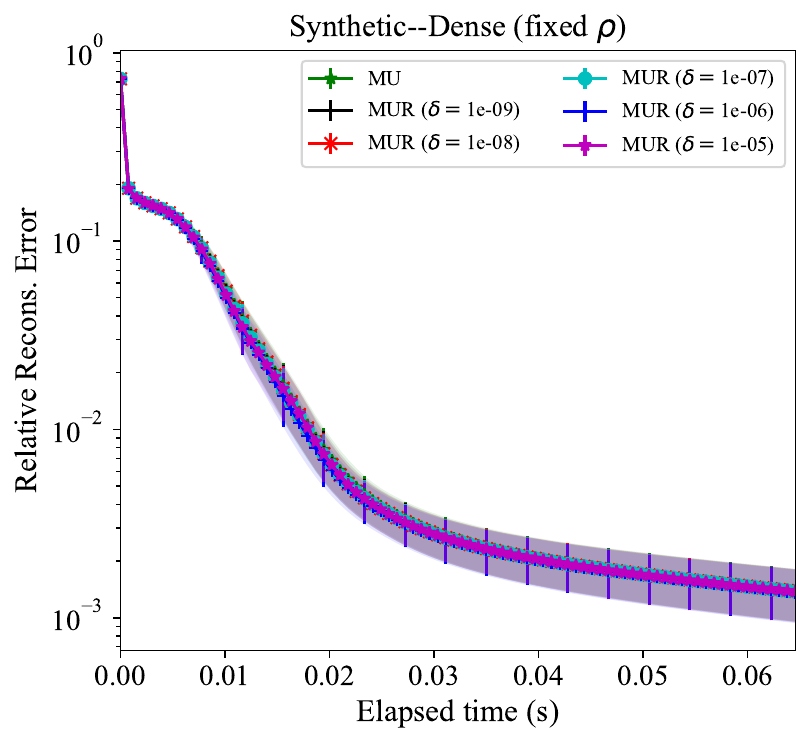}
  %\caption{image4}
  \label{fig:4}
\end{subfigure}\hfil % <-- added
\begin{subfigure}{0.33\textwidth}
  \includegraphics[width=\linewidth]{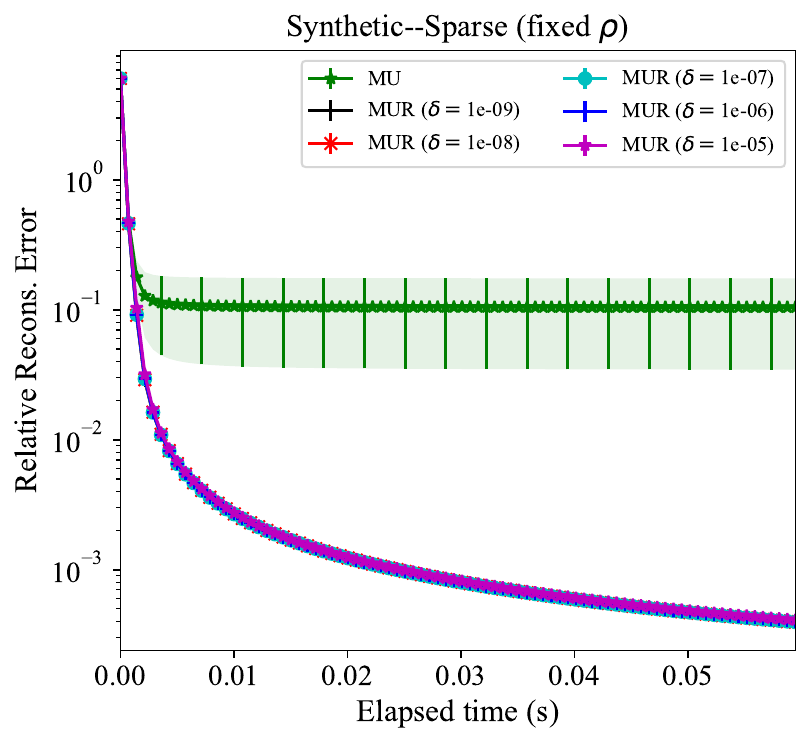}
  %\caption{image5}
  \label{fig:5}
\end{subfigure}\hfil % <-- added
\begin{subfigure}{0.33\textwidth}
  \includegraphics[width=\linewidth]{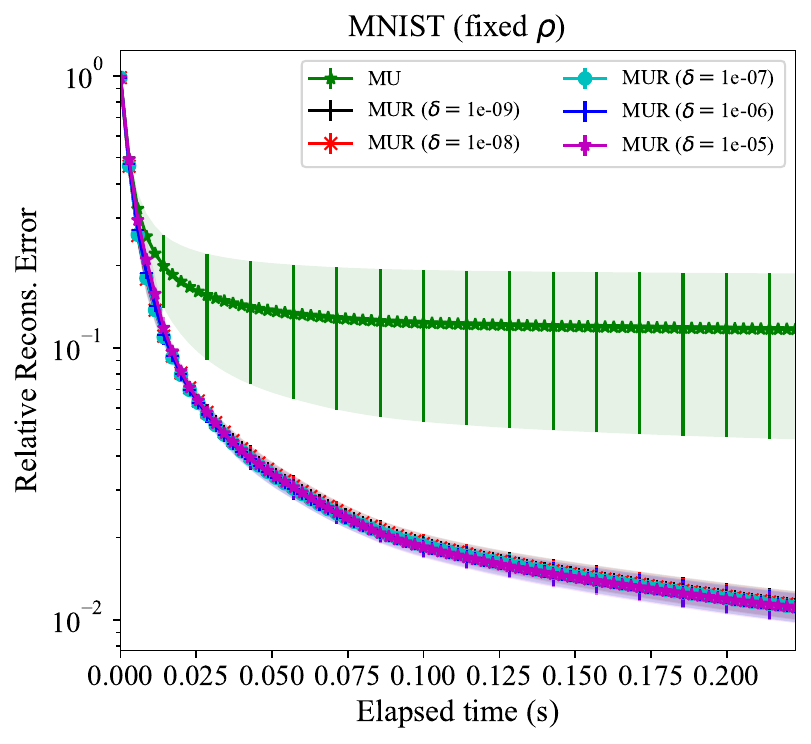}
  %\caption{image6}
  \label{fig:6}
\end{subfigure}
\vspace{-0.7cm}
\caption{Comparison of the performance of MUR  for the nonnegative matrix factorization (NMF) problem against MU. For MUR in the first (second) row, $\delta$ ($\rho$) is fixed as $10^{-8}$. The number of columns of loading matrices is set to be $r=2$ for synthetic data and $r=15$ for MNIST data. The average relative reconstruction error with standard deviation is shown by the solid lines and shaded regions of respective colors.}
\vspace{-0.7cm}
\label{fig:NMF}
\end{figure}

In numerical experiments, we use MU and MUR to learn nonnegative matrices $W \in \R^{100\times 2}_{\ge 0}$ and $H \in \R^{2\times 50}_{\ge 0}$ with synthetic data, and $W \in \R^{28\times 15}_{\ge 0}$ and $H \in \R^{15\times 28}_{\ge 0}$ with MNIST data. MU and MUR with various threshold parameters $\delta$ and regularization parameter $\rho$ are run $100$ times in each experiment with random initial data. The average relative reconstruction error with standard deviation is computed and shown in Figure \ref{fig:NMF} with solid lines and shaded regions. As shown in Figure \ref{fig:NMF}, in the synthetic data case without the sparsity feature, MU and MUR show similar convergence speeds. In the sparse data case, for both synthetic and MNIST data, MUR with various parameters significantly outperforms MU. In fact, writing MU in gradient descent form \cite{lin2007convergence}, the step size of gradient descent updates involves $H$ and $W$ in both the denominator and numerator, whose elements could possibly be zero especially when the data is sparse. A zero numerator of the step size results in no change during updates, while a zero denominator would lead to blow-up issues. These challenges contribute to the comparatively poorer performance of MU when compared to MUR. %MUR overcomes the issues with the help of parameters $\rho$ and $\delta$, as discussed in detail in Section \ref{sec:NMF}.

	\section*{Acknowledgements}
	HL is partially supported by NSF DMS-2206296. YL is supported by NSF Award DMS-2023239. The authors appreciate valuable comments from Stephen Wright.

		\vspace{0.3cm}
		\small{
\newcommand{\etalchar}[1]{$^{#1}$}
\providecommand{\bysame}{\leavevmode\hbox to3em{\hrulefill}\thinspace}
\providecommand{\MR}{\relax\ifhmode\unskip\space\fi MR }
% \MRhref is called by the amsart/book/proc definition of \MR.
\providecommand{\MRhref}[2]{%
  \href{http://www.ams.org/mathscinet-getitem?mr=#1}{#2}
}
\providecommand{\href}[2]{#2}

		\vspace{0.5cm}
		\addresseshere

\end{document}